\documentclass[12pt]{amsart}

\usepackage{amsmath,amsthm,amsfonts,amssymb,mathrsfs}
\usepackage[all]{xy}
\usepackage[total={6.5in,8.75in}, top=1.2in, left=1.0in, includefoot]{geometry}
\usepackage{amssymb}
\usepackage{latexsym}
\usepackage{amsmath,amsthm}
\usepackage{aeguill}
\usepackage{amssymb}
\usepackage{mathrsfs}
\usepackage{hyperref}
\usepackage{etoolbox}
\usepackage{enumerate}
\usepackage[dvipsnames]{xcolor}
\usepackage{tikz}[2013/12/13]

\makeatletter
\pretocmd{\chapter}{\addtocontents{toc}{\protect\addvspace{15\p@}}}{}{}
\pretocmd{\section}{\addtocontents{toc}{\protect\addvspace{3\p@}}}{}{}
\makeatother

\makeatletter
\def\@tocline#1#2#3#4#5#6#7{\relax
  \ifnum #1>\c@tocdepth 
  \else
    \par \addpenalty\@secpenalty\addvspace{#2}%
    \begingroup \hyphenpenalty\@M
    \@ifempty{#4}{%
      \@tempdima\csname r@tocindent\number#1\endcsname\relax
    }{%
      \@tempdima#4\relax
    }%
    \parindent\z@ \leftskip#3\relax \advance\leftskip\@tempdima\relax
    \rightskip\@pnumwidth plus4em \parfillskip-\@pnumwidth
    #5\leavevmode\hskip-\@tempdima
      \ifcase #1
       \or\or \hskip .5em \or \hskip 1em \else \hskip 1.5em \fi%
      #6\nobreak\relax
    \dotfill\hbox to\@pnumwidth{\@tocpagenum{#7}}\par
    \nobreak
    \endgroup
  \fi}
\makeatother
\newcommand{\C}{\mathbb{C}}
\newcommand{\N}{\mathbb{N}}
\newcommand{\Z}{\mathbb{Z}}

\newcommand{\Q}{\mathbb{Q}}
\newcommand{\F}{\mathbb{F}}
\newcommand{\A}{\mathbb{A}}

\newcommand{\cG}{\mathcal{G}}
\newcommand{\cH}{\mathcal{H}}

\newcommand{\cO}{\mathcal{O}}
\newcommand{\cS}{\mathcal{S}}
\newcommand{\cT}{\mathcal{T}}

\newcommand{\m}{\mathfrak{m}}

\newcommand{\Gal}{\operatorname{Gal}}

\newcommand{\rk}{\operatorname{rk}}

\renewcommand{\sc}{\operatorname{sc}}
\renewcommand{\ss}{\operatorname{ss}}
\newcommand{\der}{\operatorname{der}}
\newcommand{\red}{\operatorname{red}}
\newcommand{\rank}{\operatorname{rank}}

\newcommand{\Hom}{\operatorname{Hom}}
\newcommand{\Aut}{\operatorname{Aut}}
\newcommand{\End}{\operatorname{End}}

\newcommand{\conn}{\operatorname{conn}}

\newcommand{\SO}{\mathrm{SO}}

\newcommand{\GL}{\mathrm{GL}}
\newcommand{\PGL}{\mathrm{PGL}}
\newcommand{\SL}{\mathrm{SL}}
\newcommand{\PSL}{\mathrm{PSL}}
\newcommand{\Sp}{\mathrm{Sp}}
\newcommand{\GSp}{\mathrm{GSp}}
\newcommand{\GO}{\mathrm{GO}}

\def\bM{\mathbf{M}}
\def\bG{\mathbf{G}}
\def\bH{\mathbf{H}}

\def\bT{\mathbf{T}}

\def\bZ{\mathbf{Z}}
\def\bU{\mathbf{U}}

\def\bpx{\begin{pmatrix}}
\def\epx{\end{pmatrix}}

\newcommand\s{\mathsection}

\newcommand\uU{\underline{U}}

\newcommand\uM{\underline{M}}
\newcommand\uP{\underline{P}}
\newcommand\uG{\underline{G}}

\newcommand\uH{\underline{H}}

\newcommand\uJ{\underline{J}}
\newcommand\uS{\underline{S}}
\newcommand\uT{\underline{T}}

\newcommand\uZ{\underline{Z}}

\newcommand{\finiteG}{G}

\newcommand{\hto}{\hookrightarrow}

\newtheorem{thm}{Theorem}[section]

\newtheorem{cor}[thm]{Corollary}

\newtheorem{prop}[thm]{Proposition}
\newtheorem{lemma}[thm]{Lemma}

\newtheorem{remark}[thm]{Remark}

\begin{document}

\title[]{Monodromy of subrepresentations and irreducibility of low degree automorphic Galois representations}

\author{Chun Yin Hui}
\email{chhui@maths.hku.hk}
\email{pslnfq@gmail.com}
\address{
Department of Mathematics\\
The University of Hong Kong\\
Pokfulam, Hong Kong
}


\begin{abstract}
Let $X$ be a smooth, separated, geometrically connected scheme defined over a number field $K$ and
 $\{\rho_\lambda:\pi_1(X)\to\GL_n(E_\lambda)\}_\lambda$  
a system of semisimple $\lambda$-adic 
representations of the \'etale fundamental group of $X$ such that for each closed point $x$ of $X$, the specialization 
$\{\rho_{\lambda,x}\}_\lambda$ is a compatible system of Galois representations under mild local conditions. 
For almost all $\lambda$, we prove that
any type A irreducible subrepresentation   
of $\rho_\lambda\otimes_{E_\lambda} \overline \Q_\ell$ is residually irreducible.
When $K$ is totally real or CM, $n\leq 6$, and $\{\rho_\lambda\}_\lambda$ is the compatible system of Galois representations of $K$
 attached to a regular algebraic, polarized,
cuspidal automorphic representation of $\GL_n(\A_K)$, for almost all $\lambda$ we prove that $\rho_\lambda\otimes_{E_\lambda}\overline\Q_\ell$ is
(i) irreducible and (ii) residually irreducible if in addition $K=\Q$. 
\end{abstract}

\maketitle
\tableofcontents

\section{Introduction}\label{s1}
In this article, we establish big monodromy results for subrepresentations of some semisimple compatible systems 
of $\lambda$-adic representations of \'etale fundamental groups (e.g., Galois groups of number fields) 
and study with these results the conjectural irreducibility of some automorphic Galois representations 
of totally real or CM number fields.

Let $\Pi$ be a profinite group and   
$\sigma:\Pi\to \GL_m(\overline\Q_\ell)=\GL(W)$ an $\ell$-adic representation. 
The image $\sigma(\Pi)$ is called the \emph{monodromy group} of $\sigma$.
The Zariski closure of the monodromy group in $\GL_W$ is called the \emph{algebraic monodromy group} $\bG_W$ of $\sigma$.
The representation $\sigma$ is said to be \emph{of type A} if  the group
$\bG_W$ is of type A (e.g., isomorphic to $\GL_2$, see $\mathsection$\ref{NT} for definition) and $\sigma$ is
\emph{residually irreducible} if its semisimplified residual representation $\bar\sigma^{\ss}:\Pi\to\GL_m(\overline\F_\ell)$ is irreducible.
Our main big monodromy results, Theorems \ref{thmA}, \ref{ae}, and \ref{general}, compare
the images of $\sigma$ and $\bar\sigma^{\ss}$ under the setting of compatible system of representations.
A prototypical result on the monodromy of compatible system of Galois representations 
is the following well-known theorem of J.-P. Serre.
Let $K$ be a number field, $\overline K$ an algebraic closure of $K$, and $\Gal_K:=\Gal(\overline K/K)$ the absolute Galois group of $K$.

\begin{thm}\label{Serre0} \cite[Ch. VI]{Se98},\cite[Theorem 2]{Se72}
Let $\{\rho_\ell:\Gal_K\to\GL_2(\Q_\ell)\}_\ell$ be the compatible system of $\ell$-adic 
representations attached to an elliptic curve $E/K$ with no complex multiplication over $K$ (i.e., $\mathrm{End}_K(E)=\Z$).
The following assertions hold.
\begin{enumerate}[(i)]
\item The Galois representation $\rho_\ell:\Gal_K\to\GL_2(\Q_\ell)$ is irreducible for all $\ell$
and residually irreducible for all sufficiently large $\ell$. 
\item  Moreover, if $E$ has no complex multiplication over $\overline K$ (i.e., $\mathrm{End}_{\overline K}(E)=\Z$), then 
the algebraic monodromy group of $\rho_\ell$ is $\GL_2$ for all $\ell$ and 
the monodromy (resp. residual image) of $\rho_\ell$ is isomorphic to $\GL_2(\Z_\ell)$ (resp. $\GL_2(\F_\ell)$) for all sufficiently large $\ell$.
\end{enumerate}
\end{thm}

Let $X$ be a smooth, separated, geometrically connected scheme over the number field $K$.
The \'etale fundamental groups of $X$ and $X_{\overline K}:=X\times_K\overline K$ (with respect to some geometric point)
and the absolute Galois group $\Gal_K:=\Gal(\overline K/K)$ of $K$
form a short exact sequence of profinite groups
\begin{equation}\label{diag}
\xymatrix{
1 \ar[r] &\pi_1(X_{\overline K}) \ar[r] & \pi_1(X) \ar[r] & \mathrm{Gal}_K \ar[r]& 1\\
&&& \Gal_{K(x)} \ar@{_{(}->}[u]^{i} \ar@{_{(}->}[lu]^{sp_x}}
\end{equation}
such that any closed point $x$ of $X$ induces (up to inner automorphism of $\pi_1(X)$) 
a splitting $sp_x$ of the open subgroup $\Gal_{K(x)}:=\Gal(\overline K/K(x))$ of $\Gal_K$,
where $K(x)$ is the residue field of $x$.
Fix a number field $E$ and denote by $\Sigma_E$ the set of finite places of $E$. 
We call a system  
\begin{equation}\label{introsys}
\{\rho_\lambda:\pi_1(X)\to\GL_n(E_\lambda)\}_{\lambda\in\Sigma_E}
\end{equation}
of $\lambda$-adic representations of $\pi_1(X)$
an \emph{$E$-rational compatible system} 
if for every closed point $x\in X$, the  system of $\lambda$-adic representations of the number field $K(x)$
\begin{equation}\label{spsys}
\{\rho_{\lambda,x}:=\rho_\lambda\circ sp_x:\Gal_{K(x)}\to\GL_n(E_\lambda)\}_{\lambda\in\Sigma_E}
\end{equation}
 is an \emph{$E$-rational Serre compatible system} ($\mathsection\ref{scs}$). 
Set $\ell$ to be the residue characteristic of $\lambda$ and
write $\bar\epsilon_\ell:\Gal_K\to\F_\ell^*$ as the mod $\ell$ cyclotomic character.
Suppose \eqref{introsys} is semisimple (i.e., $\rho_\lambda$ is semisimple for all $\lambda$),
Theorem \ref{thmA} (comparing with Theorem \ref{Serre0}(i)) gives conditions for an irreducible
subrepresentation $\sigma_\lambda$ of $\rho_\lambda\otimes_{E_\lambda}\overline\Q_\ell$\footnote{From now on, if $\phi_\lambda:\Pi\to\GL_m(E_\lambda)$ is a group representation
then the notation
$\phi_\lambda\otimes\overline\Q_\ell$ is the tensor product $\phi_{\lambda}\otimes_{E_\lambda}\overline\Q_\ell$ with respect to some field embedding $E_\lambda\to\overline\Q_\ell$ over $\Q_\ell$.} 
to be residually irreducible, where $E_\lambda\to\overline\Q_\ell$ is a field embedding 
over $\Q_\ell$.

\begin{thm}\label{thmA}
Let $\{\rho_\lambda\}_\lambda$ be a semisimple $E$-rational compatible system \eqref{introsys} of $\pi_1(X)$.
Suppose for every closed point $x\in X$, there exist integers $N_1,N_2\geq 0$ and a finite extension $K'/K(x)$ such that the following conditions hold for the Serre compatible system $\{\rho_{\lambda,x}\}_\lambda$.
\begin{enumerate}[(a)]
\item (Bounded tame inertia weights): for almost all $\lambda$ \footnote{This means ``for all but finitely many $\lambda\in\Sigma_E$''.}
and each finite place $v$ of $K(x)$ above $\ell$, 
the tame inertia weights of the local representation 
$(\bar\rho_{\lambda,x}^{\ss}\otimes\bar\epsilon_\ell^{N_1})|_{\Gal_{K(x)_v}}$ belong to $[0,N_2]$.
\item (Potential semistability): for almost all $\lambda$ and each finite place $w$ of $K'$ not above $\ell$,
the semisimplification of the local representation $\bar\rho_{\lambda,x}^{\ss}|_{\Gal_{K_{w}'}}$ is unramified.
\end{enumerate}
For almost all $\lambda$, if $\sigma_\lambda$ is a type A irreducible subrepresentation of $\rho_\lambda\otimes\overline\Q_\ell$, 
then  $\sigma_\lambda$ is residually irreducible.
\end{thm}

\begin{cor}\label{corA}
Let $\{\rho_\lambda\}_\lambda$ be the semisimple $E$-rational compatible system of $\pi_1(X)$ in Theorem \ref{thmA}.
For almost all $\lambda$, if $\sigma_\lambda$ is an irreducible subrepresentation of $\rho_\lambda\otimes\overline\Q_\ell$ 
of dimension less than or equal to $3$,
then $\sigma_\lambda$ is residually irreducible.
\end{cor}

It is conjectured that for $E$-rational compatible systems 
$\{\rho_\lambda\}_\lambda$ of geometric origin 
(e.g., arising from higher direct image sheaves $R^w f_*\Q_\ell$ on $X$, where $f:Y\to X$ is a smooth projective morphism), 
the conclusion of Theorem \ref{thmA}
holds regardless of the type A condition on the subrepresentation 
$\sigma_\lambda:\pi_1(X)\to\GL_m(\overline\Q_\ell)=\GL(W_\lambda)$,
see \cite[$\mathsection11$]{Se94} for conjectures on maximal motives; \cite{La95}, \cite[Proposition 5.3.2]{BLGGT14}, and \cite{PSW18} for some density one results; \cite[Theorem 1.2]{HL20} for a criterion of Galois maximality;
\cite{HL13} and \cite{Ca15,Ca19,CM20} for adelic openness results.
The key feature of the type A condition is that if $\bH$ is a subgroup of a connected semisimple type A group $\bG$
such that $\bH$ and $\bG$ have the same semisimple rank,
then $\bH$ is equal to $\bG$. By a specialization result of Serre, the proof of Theorem \ref{thmA} 
can be reduced to the number field case, i.e., $X=\mathrm{Spec}(K)$.
As a result, $\pi_1(X)$ will only show up again in the proof of Theorem \ref{thmA}
and the rest of the article concerns Galois representations of $K$.
A key step of Theorem \ref{thmA} is to prove that for all sufficiently large $\ell$,
the finite Galois image of the semisimplified reduction $\bar\sigma_\lambda^{\ss}$ 
is approximated by a connected reductive subgroup $\uG_{W_\lambda}$ 
(called the \emph{algebraic envelope}\footnote{The method of algebraic envelopes $\uG_\ell$ was pioneered by J.-P. Serre in \cite{Se86} 
to study Galois representations for abelian varieties. The semisimple part of $\uG_\ell$ was constructed by Nori's theory \cite{No87}.} 
of $\sigma_\lambda$) 
of $\GL_{m,\overline\F_\ell}$ 
with many nice properties (see Theorems \ref{ae} and \ref{general}, comparing with Theorem \ref{Serre0}(ii)).
The big monodromy results in this article 
are motivated by some group theoretic results in \cite{HL20} and generalize the works \cite{Hu15} and \cite{HL16} in three aspects:
\begin{itemize}
\item big monodromy results are established on 
subrepresentations $\sigma_\lambda$ of $\rho_\lambda\otimes\overline{\Q}_\ell$; 
\item the algebraic monodromy group of $\rho_\lambda$ need not be connected\footnote{This condition is independent of $\lambda$.};
\item the system \eqref{introsys} needs only mild conditions on the local Galois representations of $\rho_\lambda$. 
\end{itemize}
The last aspect enables us to construct algebraic envelopes for some automorphic compatible systems (see $\mathsection$\ref{ic}).

It is conjectured that Galois representations (conjecturally \cite{Cl90}) associated to an algebraic cuspidal automorphic
representation $\pi$ of $\GL_n(\A_K)$ over a number field $K$ are 
absolutely irreducible (see \cite{Ra08}).
For classical modular forms and Hilbert modular forms, the results are proved by Ribet \cite{Ri77} and Taylor \cite{Ta95} (see also \cite{Di05}).
When $n=3$, $K$ is CM, and $\pi$ is polarized, the result is proved in \cite[Theorem 2.2.1]{BR92}. 
For some particular cases when $n=4$, see \cite{Ra13},\cite{DZ20},\cite{We22}.
Let $F$ be a totally real or CM number field. Due to the work of 
many mathematicians\footnote{See for example, the paragraphs before and after \cite[Theorem 2.1.1]{BLGGT14}.}, 
attached to a \textit{regular algebraic, polarized, cuspidal} 
automorphic representation $\pi$ of $\GL_n(\A_F)$ is a \emph{strictly compatible system} ($\mathsection\ref{wcs}$) of $F$ 
defined over a CM field $E$ containing all embeddings of $F$ in $\overline\Q$.
Enlarging $E$ to a bigger CM field if necessary, the strictly compatible system can be written as a 
semisimple $E$-rational Serre compatible system 
\begin{align}\label{introsys2}
\begin{split}
\{\rho_{\pi,\lambda}: \Gal_F\to \GL_n(E_\lambda)\}_\lambda. 
\end{split}
\end{align}

When $n\leq 5$ and $F$ is totally real (resp. $n\leq 6$ and $F$ is CM), it is proved that 
$\rho_{\pi,\lambda}$ is absolutely irreducible for $\lambda$ lying above a set of rational primes $\ell$ of
Dirichlet density one  \cite[Theorem 3.2]{CG13} (resp. \cite[Theorem 2]{Xi19}). 
For general $n$, the absolute irreducibility of $\rho_{\pi,\lambda}$
is proved for $\lambda$ lying above a set of rational primes $\ell$ of 
positive Dirichlet density  \cite[Theorem D]{PT15} and for 
$\lambda$ lying above a set of rational primes $\ell$ of Dirichlet density one 
if $\pi$ is extremely regular \cite[Theorem D]{BLGGT14}. 
Inspired by some recent works (\cite{BLGGT14}, \cite{CG13}, \cite{PT15}, \cite{Xi19} etc.) on the conjecture,
we obtain new cases of the irreducibility conjecture for \eqref{introsys2} when $n\leq 6$,
by combining various potential automorphy and Galois theoretic results, including our big monodromy results for subrepresentations.

\begin{thm}\label{thmB}
Let $F$ be a totally real or CM field, $n\leq 6$, and $\{\rho_{\pi,\lambda}\}_\lambda$ 
the semisimple Serre compatible system \eqref{introsys2} 
attached to a regular algebraic, polarized, cuspidal automorphic representation $\pi$ of $\GL_n(\A_F)$. 
The following assertions hold.
\begin{enumerate}[(i)]
\item The representation $\rho_{\pi,\lambda}\otimes\overline\Q_\ell$ is irreducible  for almost all $\lambda$.
\item If in addition $F=\Q$, the representation $\rho_{\pi,\lambda}\otimes\overline\Q_\ell$ is residually irreducible for almost all $\lambda$.
\end{enumerate}
\end{thm}

Denote by $\bG_\lambda$ the algebraic monodromy group of $\rho_{\pi,\lambda}$ 
and by $\bG_\lambda^{\der}$ the derived group of $\bG_\lambda^\circ$.
We illustrate the strategy of proving Theorem \ref{thmB}(i) by a simple case.
Suppose $F$ is totally real, $n=4$, and $\bG_{\lambda_0}^{\der}=\Sp_4\subset\GL_4$ for some $\lambda_0$.
Since the formal character of $\bG_\lambda^{\der}$ is independent of $\lambda$ (Theorem \ref{Hui1}),
$\bG_\lambda^{\der}\cong\SL_2\times\SL_2\subset\GL_4$ when $\rho_{\pi,\lambda}\otimes\overline\Q_\ell=W_\lambda\oplus W_\lambda'$ 
is reducible. In this case and if $\ell$ is sufficiently large, $\Gal_{F(\zeta_\ell)}$ 
is residually irreducible (Theorem \ref{general}(v)) on the two dimensional factors $W_\lambda$ and $W_\lambda'$,
which are also polarized and odd (Proposition \ref{esd2}).
By a potential automorphy result \cite[Theorem C]{BLGGT14}, the factors $W_\lambda$ and $W_\lambda'$ 
are automorphic after a totally real extension $F'/F$. This implies that $\{\rho_{\pi,\lambda}\}_\lambda$
after field extension $F'/F$, is a sum of two two-dimensional compatible systems,
contradicting the Lie-irreducibility of $\rho_{\pi,\lambda_0}\otimes\overline\Q_{\ell_0}$. 
Therefore,  $\rho_{\pi,\lambda}\otimes\overline\Q_\ell$ is irreducible for $\ell\gg0$.
The proof of Theorem \ref{thmB}(ii) needs an extra input from
the strong form of Serre's modularity conjecture \cite{Se87},\cite{Ed92},\cite{KW09a,KW09b}.
For some four-dimensional geometric compatible systems of $\Q$, big residual image results
have been obtained via this technique \cite{DV08,DV11}. 
We expect that the residual irreducibility of $\rho_{\pi,\lambda}$ for almost all $\lambda$
can be proven in a similar way, if some strong form of Serre's conjecture holds for the field $F$
(see e.g., \cite{Fi99} and \cite{BDJ10}).


The structure of this article is as follows.
In section $2$, we present the terminologies (e.g., algebraic envelopes, formal bi-character, etc.) 
and previous results needed for proving the main theorems.
The big monodromy results (e.g., Theorem \ref{thmA} and Theorem \ref{general})
are then established in section $3$ by adapting the group theoretic methods employed in \cite{Hu15},\cite{HL16,HL20}
to subrepresentations of the compatible system \eqref{introsys}. 
We prove Theorem \ref{thmB} in section $4$ via a reduction result of \cite{Xi19} and
a case-by-case analysis on the possibilities of the formal character of $\bG_{\lambda}^{\der}$ in the Lie-irreducible situation. 
We also deduce some $\lambda$-independence results on the faithful representations of $\bG_\lambda^{\circ}$ 
(Propositions \ref{lambdaindep} and \ref{lambdaindepCM}).

\section{Preliminaries}\label{s2}
\subsection{Notation and terminologies}\label{NT} 
We adopt the following notation and terminologies throughout the article.
\begin{itemize}
\item $K$ (resp. $E$) is a number field and $\Sigma_K$ (resp. $\Sigma_E$) denotes its set of finite places.
\item Fix an algebraic closure $\overline K$ of $K$ and let $\Gal_K:=\Gal(\overline K/K)$ be the absolute Galois group.
\item For $\lambda\in\Sigma_E$, set $\ell$ to be the residue characteristic of $\lambda$.
\item If $F$ is a field, denote by $\overline F$ an algebraic closure of $F$.
\item Let $R\to S$ be a morphism of commutative unital rings
and $X$ (or $X_R$) an $R$-scheme. Denote by $X_S$ the 
base change $X\times_R S$.
\item Let $\bG$ be a linear algebraic group defined over a field $F$. 
Denote by $\bG^\circ$ the identity component of $\bG$, by $\bG^{\der}$ the 
derived group $[\bG^\circ,\bG^\circ]$, and by $\bU$ the unipotent radical of $\bG$.
The group $\bG$ is \emph{reductive} if $\bU$ is trivial \cite[$\mathsection11$]{Bo91}.
\item The \emph{rank} of $\bG$, denoted by $\rank\bG$, means 
the rank of $\bG_{\overline F}$ in usual sense. The \emph{semisimple rank} of $\bG$ means the rank of $(\bG/\bU)^{\der}$.
\item The group $\bG/F$ is said to be of \emph{type A}
if the root system of $(\bG/\bU)^{\der}_{\overline F}$ is a (possibly empty) product of type A irreducible root subsystems. 
\item Let $\Gamma$ be a compact topological group, $L$ an algebraic field extension of $\Q_\ell$, $V$ a finite-dimensional $L$-vector
space equipped with the $\ell$-adic topology, and
 $\rho:\Gamma\to \GL(V)$ a continuous representation.
Denote by $\mathcal{O}$ the ring of integers of $L$ 
and $\m$ its maximal ideal.
Since there is an $\mathcal{O}$-lattice of $V$ stabilized by the compact $\Gamma$, 
we obtain by (mod $\m$) reduction and semisimplification
an $\mathcal{O}/\m$-representation denoted by $\bar\rho^{\ss}$ or $\overline V^{\ss}$,
which is independent of the choice of the $\mathcal{O}$-lattice by the Brauer-Nesbitt Theorem 
\cite[Theorem 30.16]{CR88}. We call $\bar\rho^{\ss}$ (or $\overline V^{\ss}$)
the \emph{semisimplified residual representation} or \emph{semisimplified reduction} of $\rho$ (or $V$).
\end{itemize}

\subsection{Some group theoretic results}\label{algp}
Let $\bG/F$ be a linear algebraic group. 

\begin{prop}\label{prepare}
Suppose $F$ is algebraically closed and $\bT$ is a maximal torus of $\bG$.
Then the intersection $\bT^{\ss}:=\bT\cap \bG^{\der}$ is also a maximal torus of $\bG^{\der}$.
\end{prop}

\begin{proof}
We may assume $\bG$ is connected. Since 
the morphism $\bG\to\bG/\bU$ 
when restricted to $\bG^{\der}$ is surjective onto $(\bG/\bU)^{\der}$
and when restricted to $\bT$ is injective, 
we may also assume $\bG$ is reductive.
Every maximal torus of $\bG^{\der}$ is contained in a maximal torus of $\bG$ which is conjugate to $\bT$,
so $\bT^{\ss}$ contains a maximal torus $\bT'$ of the semisimple $\bG^{\der}$.
The centralizer of $\bT'$ in $\bG^{\der}$ is $\bT'$ itself, so $\bT^{\ss}$ is contained in $\bT'$, and
we are done.
\end{proof}

\begin{prop}\label{lem1}
Suppose $\bG\subset\GL_{n,F}$ is a subgroup.
Let $\Gamma$ be a subgroup of $\bG(F)\subset\GL_n(F)$ 
such that $\Gamma$ and $\bG$ are semisimple on the ambient space $F^n$ 
and have the same commutants: $\End_\Gamma(F^n)=\End_{\bG}(F^n)$.
The following assertions hold.
\begin{enumerate}[(i)]
\item If $F^n=\oplus_i W_i$ is an irreducible decomposition of $\bG$ on $F^n$, then 
this is also an irreducible decomposition of $\Gamma$. 
Moreover, $W_i\cong W_j$ as representations of $\bG$ iff $W_i\cong W_j$ as representations of $\Gamma$.
\item There is a bijective correspondence between subrepresentations of 
$\bG$ on $F^n$ and subrepresentations of $\Gamma$ on $F^n$,
given by restriction.
\end{enumerate}
\end{prop}

\begin{proof}
(i) The condition on commutants implies $\End_\Gamma(W_i)=\End_{\bG}(W_i)$.
Since the right-hand side is a division algebra, so is the left-hand-side.
Together with the fact that $\Gamma$ is semisimple on $W_i$, it follows that $\Gamma$ is irreducible on $W_i$.
If $\phi:W_i\to W_j$ is a $\Gamma$-isomorphism, then one finds a $\Gamma$-homomorphism 
$$\Phi: F^n \twoheadrightarrow W_i\stackrel{\phi}{\rightarrow}W_j\hookrightarrow F^n$$
which is an element of $\End_\Gamma(F^n)$. Since $\End_\Gamma(F^n)=\End_{\bG}(F^n)$,
the map $\Phi$ can be extended to a $\bG$-homomorphism which implies that 
$W_i\cong W_j$ as $\bG$-representations. 

(ii) Injectivity follows from assertion (i). For surjectivity, suppose $V$ is a $\Gamma$-subrepresentation of $F^n$. 
There exists projection map to $V$ in $\End_\Gamma(F^n) = \End_{\bG}(F^n)$,
so the image $V$ of this projector is a $\bG$-subrepresentation.
\end{proof}

\subsection{Formal character and bi-character}
Following \cite[$\mathsection2.6.1$]{HL20}, we define formal characters 
and formal bi-characters 
of subgroups of $\GL_n$ which are not necessarily reductive. 
Let $F$ be a field and $\bG\subset\GL_{n,F}$ a connected subgroup.
Suppose first $F$ is algebraically closed.
Let $\bT$ be a maximal torus of $\bG$. 
The subtorus
$$\bT\subset\GL_{n,F}$$ 
taken up to $\GL_{n,F}$-conjugation, is called the \emph{formal character} of $\bG$.
Let $\bT'$ be a maximal torus of $\bG^{\der}$ contained in $\bT$.
Then $\bT'=\bT^{\ss}$ by Proposition \ref{prepare}.
The chain of subtori 
$$\bT'=\bT^{\ss}\subset\bT\subset\GL_{n,F}$$
taken up to $\GL_{n,F}$-conjugation, 
is called the \emph{formal bi-character} of $\bG$. 
Both definitions are independent of the 
choice of the maximal torus $\bT$.

In general, the formal character (resp. bi-character)  of a
 subgroup $\bG$ of $\GL_{n,F}$ is defined to be the formal character (resp.\ bi-character) of $\bG^\circ$.
For a general field $F$, the formal character (resp. bi-character) of a subgroup $\bG\subset\GL_{n,F}$ 
is defined to be the formal character (resp. bi-character) of 
$\bG_{\overline F}$\footnote{The notion of formal character (resp. formal bi-character) of a reductive subgroup $\bG\subset\GL_{n,F}$ 
we adopt in this article is 
defined in an absolute sense (i.e., by passing to $\overline F$), 
which is different from the one defined in \cite[Definitions 2.2, 2.3]{Hu18}
as a $F$-subtorus $\bT\subset\GL_{n,F}$ (resp. a chain of $F$-subtori $\bT^{\ss}\subset\bT\subset\GL_{n,F}$) 
such that $\bT$ is a maximal torus of $\bG$ and $\bT^{\ss}$ is 
a maximal torus of $\bG^{\der}$.}.
Note that the formal bi-characters of $\bG\subset\GL_{n,F}$ and its semisimplification
are equal.

Let $I$ be a non-empty set and $\{\bG_i\subset\GL_{n_i,F_i}\}_i$ be a family of subgroups 
 indexed by $I$ with no restriction on the dimensions $n_i$ and fields $F_i$. 
The family is said to have \emph{the same formal character} (resp. \emph{bi-character})
if $n_i$ are all equal to some integer $n$ and there is a diagonal $\Z$-subtorus $\bT_{\Z}\subset\mathbb{G}_{m,\Z}^n$ 
(resp. chain of $\Z$-subtori $\bT'_\Z\subset\bT_\Z\subset\mathbb{G}_{m,\Z}^n$)
such that the base change 
$$\bT_{\overline F_i}\subset\GL_{n,\overline F_i} \hspace{.05in}
(\mathrm{resp.}\hspace{.05in} \bT'_{\overline F_i}\subset\bT_{\overline F_i}\subset\GL_{n,\overline F_i})$$ 
is the formal character (resp. bi-character) of $\bG_i\subset\GL_{n,F_i}$ for all $i$.
This defines an equivalence relation. For example,
the equivalence classes of formal characters in $\GL_n$ correspond to $\Z$-subtori of $\mathbb{G}_{m,\Z}^n$. 
The family $\{\bG_i\subset\GL_{n_i,F_i}\}_i$ is said to have \emph{bounded formal \hbox{(bi-)characters}}
if their formal \hbox{(bi-)characters} fall into finitely many equivalence classes. 
These definitions allow us to compare formal \hbox{(bi-)characters} defined over different fields;
we refer to \cite[$\mathsection2$]{Hu15},\cite[$\mathsection2.6$]{HL20}
for equivalent definitions. 

\subsection{Compatible system of Galois representations}\label{csgr}
Let $K$ and $E$ be number fields.
Denote by $S_\ell\subset\Sigma_K$ 
the subset consisting of finite places $v$ above 
a rational prime $\ell$ and by $\F_v$ the residue field corresponding to $v\in\Sigma_K$. 
For $\lambda\in\Sigma_E$, 
denote by $E_\lambda$ the $\lambda$-adic completion of $E$ and by $\ell$  
the residue characteristic of $\lambda$. 

\subsubsection{$E$-rationality and weight}\label{rat}
 A continuous $\ell$-adic representation 
\begin{equation}\label{Mrat}
\rho_\lambda:\Gal_K\to \GL_n(\overline E_\lambda)
\end{equation}
 of $K$ is said to be \emph{$E$-rational}
if there exists a finite subset $S\subset\Sigma_K$ such that the following properties hold.
\begin{enumerate}[(i)]
\item The representation $\rho_\lambda$ is unramified outside $S$.
\item If $v\in\Sigma_K\backslash S$, the coefficients of the polynomial 
$P_{v,\rho_\lambda}(T):=\mathrm{det}(\mathrm{Id}_n-\rho_\lambda(Frob_v)T)$ of 
the Frobenius class of $v$ belong to $E$.
\end{enumerate}

\begin{remark}
Condition (i) and Chebotarev's density theorem imply
that $\{\rho_\lambda(Frob_v):~v\in\Sigma_K\backslash S\}$ is a dense 
subset of $\rho_\lambda(\Gal_K)$.
\end{remark}

If $\rho_\lambda$ is $E$-rational, we say that 
$\rho_\lambda$ is \emph{rational} if $E=\Q$ and $\rho_\lambda$ is 
\emph{pure of weight $w\in\Z$} if for each $v\in\Sigma_K\backslash S$, 
all the roots of  $P_{v,\rho_\lambda}(T)\in E[T]$
in $\overline E$  have absolute value $|\F_v|^{w/2}$ under every embedding of $\overline E$ in $\C$.

If $\{\rho_{\lambda_i}:\Gal_K\to \GL_{n_i}(\overline E_{\lambda_i})\}_{i\in I}$
is a family of $n_i$-dimensional $E$-rational $\ell$-adic representations
of $K$ index by $I$, we say that the family is
semisimple (resp. pure of weight $w$) if this is true for every member $\rho_{\lambda_i}$.

\subsubsection{Compatible system in the sense of Serre}\label{scs}
 A family of $n$-dimensional $E$-rational $\ell$-adic representations  
\begin{equation}\label{sys}
\{\rho_\lambda:\Gal_K\to \GL_n(E_\lambda)\}_\lambda
\end{equation}
of $K$ indexed by $\Sigma_E$ is called a 
\emph{Serre compatible system\footnote{A Serre compatible system \eqref{sys} is the same as 
a \emph{strictly compatible system} defined in \cite[Chapter 1]{Se98}. The term is used to avoid 
confusion with a weakly/strictly compatible system in $\mathsection\ref{wcs}$.}} (SCS in short)
of $K$ defined over $E$ unramified outside a finite subset $S\subset\Sigma_K$
if the following properties hold.
\begin{enumerate}[(i)]
\item For each $\lambda$, the representation $\rho_\lambda$ is unramified outside $S\cup S_\ell$
and if $v\in\Sigma_K\backslash (S\cup S_\ell)$ then 
$P_{v,\rho_\lambda}(T)\in E[T]$.
\item For each pair of finite places $\lambda,\lambda'$ of $E$ with residue characteristics $\ell,\ell'$ 
and each finite place $v\in \Sigma_K\backslash (S\cup S_\ell\cup S_{\ell'})$, 
the equality $P_{v,\rho_\lambda}(T)=P_{v,\rho_{\lambda'}}(T)$ holds.
\end{enumerate}

\subsubsection{Weakly/strictly compatible system}\label{wcs}
A \emph{weakly compatible system} (WCS in short) of $n$-dimensional
$\ell$-adic representations of $K$ defined over $E$ and unramified outside 
a finite subset $S\subset\Sigma_K$ (see \cite[$\mathsection5.1$]{BLGGT14}) is a family of continuous semisimple representations
\begin{equation}\label{wsys}
\{\rho_\lambda:\Gal_K\to\GL_n(\overline E_\lambda)\}_\lambda
\end{equation}
indexed by $\Sigma_E$ with the following properties.
 
\begin{enumerate}[(i)]
\item The conditions (i) and (ii) in $\mathsection$\ref{scs} are satisfied.
\item Each $\rho_\lambda$ is de Rham at all places $v$ above $\ell$ and
moreover crystalline if $v\notin S$.
\item For each embedding $\tau: K\hookrightarrow \overline E$, the set 
of $\tau$-Hodge-Tate weights (see \cite[$\mathsection1$]{PT15}) of $\rho_{\lambda}$ is independent of $\lambda$ and any $\overline E\hookrightarrow \overline E_\lambda$ over $E$.
\end{enumerate}

\noindent The weakly compatible system \eqref{wsys} is called a \emph{strictly compatible system} 
(StCS in short) if in addition, the condition below holds.

\begin{enumerate}[(iv)]
\item For $v$ not above $\ell$, 
the semisimplified Weil-Deligne representation 
$\iota \text{WD}(\rho_{\lambda}|_{\Gal_{K_v}})^{F-ss}$ 
is independent of $\lambda$ and $\iota:\overline{E}_\lambda\stackrel{\cong}{\rightarrow} \C$.
\end{enumerate}

\subsubsection{Operations on compatible systems}\label{op}
Let $K'\subset K\subset K''$ and $E'\subset E\subset E''$ be number fields and denote the finite places of $E'$ 
by $\lambda'$ (resp. of $E''$ by $\lambda''$). Suppose $\{\rho^1_\lambda\}_\lambda$ and $\{\rho^2_\lambda\}_\lambda$
are Serre compatible systems (resp. weakly compatible systems) of $K$ defined over $E$ unramified outside a finite $S\subset\Sigma_K$
of dimension respectively $n_1$ and $n_2$.  The same is then true of the semisimplification of the $\{\rho^i_\lambda\}_\lambda$.
Moreover, one can construct a new SCS (resp. WCS) by any of the operations below.

\begin{enumerate}[(A)]
\item \textit{(Direct sum)}. The direct sum system $\{\rho^1_\lambda\oplus\rho^2_\lambda\}_\lambda$ 
is an $n_1+n_2$-dimensional SCS (resp. WCS) of $K$ defined over $E$.\smallskip

\item \textit{(Tensor product)}. The tensor product system $\{\rho^1_\lambda\otimes\rho^2_\lambda\}_\lambda$ 
is an $n_1n_2$-dimensional SCS (resp. WCS) of $K$ defined over $E$.\smallskip

\item \textit{(Dual)}. The dual system $\{\rho_\lambda^{i\vee}\}_\lambda$ is an $n_i$-dimensional SCS (resp. WCS)
of $K$ defined over $E$.\smallskip

\item \textit{(Restriction to subgroup)}. The restricted system $\{\mathrm{Res}^K_{K''}\rho_\lambda\}_\lambda$
is an $n$-dimensional SCS (resp. WCS) of $K''$ defined over $E$.\smallskip

\item \textit{(Induction from subgroup)}. The induced system $\{\mathrm{Ind}^{K'}_K\rho_\lambda\}_\lambda$
is an $n[K:K']$-dimensional SCS (resp. WCS)\footnote{This is due to the fact that 
the L-factors satisfy $L_{\mathfrak p'}(\mathrm{Ind}^{K'}_K\rho_\lambda,s)=
\prod_{\mathfrak{p}|\mathfrak{p'}}L_{\mathfrak{p}}(\rho_\lambda,s)$ 
for almost all prime ideals $\mathfrak p'$ of $K'$.} of $K'$ defined over $E$.\smallskip

\item \textit{(Coefficient extension)}. The coefficient extension $E''/E$ system:
$$\{\rho_\lambda\}_\lambda\otimes_E E'':=\{\rho_\lambda\otimes_{E_\lambda} E''_{\lambda''}:\hspace{.05in}\lambda''|\lambda\}_{\lambda''}$$
is an $n$-dimensional SCS of $K$ defined over $E''$.\smallskip

\item \textit{(Coefficient descent)}. If $\{\rho_{\lambda}\}_\lambda$ is a WCS and is \emph{regular}, i.e., 
the $\tau$-Hodge-Tate weights in $\mathsection$\ref{wcs}(iii) are distinct for all $\tau:K\to\overline{E}$,
then there exists a SCS 
$$\{\psi_{\lambda''}:\Gal_K\to \GL_n(E''_{\lambda''})\}_{\lambda''}$$
of $K$ defined over a finite extension $E''$ of $E$
such that 
$$\{\psi_{\lambda''}\circ(\GL_n(E''_{\lambda''})\to \GL_n(\overline{E''_{\lambda''}}))\}_{\lambda''}$$
is the coefficient extension system $\{\rho_\lambda\}_\lambda\otimes_E E''$
\cite[Lemma 5.3.1(3)]{BLGGT14}. Therefore, a regular WCS
can be turned into a semisimple SCS. \smallskip

\item \textit{(Restriction of scalars)}. By restriction of scalars $\mathrm{Res}_{E/E'}\GL_{n,E}$, 
the system  
$$\{\mathrm{Res}_{E/E'}\rho_{\lambda}:=\bigoplus_{\lambda|\lambda'}\rho_{\lambda}
:\Gal_K\to\prod_{\lambda|\lambda'}\GL_n(E_\lambda)\subset\GL_{n[E:E']}(E'_{\lambda'})\}_{\lambda'}$$
 is an $n[E:E']$-dimensional SCS of $K$ defined over $E'$.
\end{enumerate}

\subsection{Algebraic monodromy groups and $\lambda$-independence}\label{amg}
In this section, denote by $\{\rho_\lambda\}_\lambda$ a SCS (resp. WCS) defined above.
The Galois image $\rho_\lambda(\Gal_K)$, denoted by $\Gamma_\lambda$, is called the \emph{monodromy group} of $\rho_\lambda$.
It is a compact $\ell$-adic Lie group.
The \emph{algebraic monodromy group} of $\rho_\lambda$,
denoted by $\bG_\lambda$, is the algebraic subgroup of 
$\GL_{n,E_\lambda}$ (resp. $\GL_{n,\overline E_\lambda}$) 
defined as the Zariski closure of $\Gamma_\lambda$ 
in $\GL_n$. We review some $\lambda$-independence results on 
the family of subgroups $\{\bG_\lambda\subset\GL_n\}_\lambda$.

\begin{thm}(Serre) \label{Serre1}\cite{Se81}
Let $\{\rho_\lambda\}_\lambda$ be a SCS or WCS.
\begin{enumerate}[(i)]
\item The component group $\bG_\lambda/\bG_\lambda^\circ$ is independent of $\lambda$.
\item The formal character of $\bG_\lambda\subset\GL_n$ is independent of $\lambda$.
\end{enumerate}
\end{thm}

\begin{remark}\label{remSerre}
It follows from Theorem \ref{Serre1}(i) and $\mathsection$\ref{op}(D) that 
if $L$ is a finite extension of $K$
such that the Zariski closure of $\rho_\lambda(\Gal_L)$ in $\GL_n$
is connected (i.e., equal to $\bG_\lambda^\circ$) for one $\lambda$ then this is also true 
for any other $\lambda$. Denote by $K^{\conn}/K$ the minimal of such extensions.
We obtain $\Gal(K^{\conn}/K)\cong \bG_\lambda/\bG_\lambda^\circ$ for all $\lambda$.
We say that $\{\rho_\lambda\}_\lambda$ is connected
if $\bG_\lambda$ is connected for one $\lambda$ (hence for all $\lambda$).
\end{remark}

This theorem of Serre is proved by studying the algebraic variety of characteristic 
polynomials of $\bG_\lambda$ for every $\lambda$, see also \cite[$\mathsection4$]{LP92}.
When $\{\rho_\lambda\}_\lambda$ is pure of weight $w$, Serre also used
the \emph{method of Frobenius torus} to prove the theorem and construct an $E$-torus $\bT$
such that the base change $\bT_{E_\lambda}$ is a maximal torus of $\bG_\lambda$ for all $\lambda$ not above some rational prime, 
see \cite[Theorem 1.1]{LP97} 
and see \cite[Theorem 2.6]{Hu18} for a more general case. 

When $\{\rho_\lambda\}_\lambda$ is a semisimple SCS and $\rho_\lambda$ is 
abelian for one $\lambda$, Serre developed the theory of 
abelian $\ell$-adic representations based on class field theory \cite{Se98}
to prove that $\bG_\lambda\subset\GL_{n, E_\lambda}$ is 
independent of $\lambda$ in the sense that there exists 
an abelian subgroup $\bG\subset\GL_{n,E}$ such that 
\begin{equation}\label{indep1}
(\bG\subset\GL_{n,E})_{E_\lambda}\cong (\bG_\lambda\subset\GL_{n,E_\lambda})\hspace{.2in}\forall\lambda.
\end{equation}
The Mumford-Tate conjecture (see \cite{Se94}) implies that \eqref{indep1} holds (for some reductive subgroup $G\subset\GL_{n,E}$) if 
$\{\rho_\lambda\}_\lambda$ is connected and arising from the $\ell$-adic cohomology groups
of smooth projective varieties defined over $K$ (see $\mathsection$\ref{ael}). 
For general rational semisimple SCS $\{\rho_\lambda\}_\lambda$, 
Larsen-Pink \cite{LP92} obtained various $\ell$-independence results on $\bG_\lambda$ 
by utilizing the compatibility condition. In particular, one has the following result.

\begin{prop}\label{LP1}\cite[Proposition 8.7]{LP92}
Let $\{\rho_\ell\}_\ell$ be a rational semisimple SCS.
For all sufficiently large $\ell$, the identity component $\bG_\ell^\circ$
splits after an unramified extension of $\Q_\ell$.
\end{prop}

By exploiting some ideas in \cite{Se98} 
and the compatibility condition to compare $\bG_\lambda$ and $\bG_{\lambda'}$
for any pair $\lambda,\lambda'\in\Sigma_E$,
Theorem \ref{Serre1}(ii) is generalized to the following.

\begin{thm}\cite[Theorem 3.19, Remark 3.22]{Hu13}\label{Hui1}
Let $\{\rho_\lambda\}_\lambda$ be a SCS or WCS.
The formal bi-character of $\bG_\lambda\subset\GL_{n}$ is independent of $\lambda$.
\end{thm}

\begin{remark}
Theorem \ref{Hui1} has then been used to study the conjectural \eqref{indep1} under various assumptions
on the algebraic monodromy groups $\bG_\lambda$, for more details see \cite[Theorem 3.21]{Hu13},
\cite[Theorem 1.2]{Hu18}, \cite[Theorem 1.6]{Hu22}, and the related \cite[Theorems 1.5, 1.6]{Hu20}.
\end{remark}

\subsection{Nori groups}\label{N}
Let $G\subset\GL_n(\F_\ell)$ be a subgroup with $\ell\geq n$. 
Denote by $G[\ell]$ the set of order $\ell$ (or equivalently, non-trivial unipotent) elements of $G$
and by $G^+\subset\GL_n(\F_\ell)$ the (normal) subgroup generated by $G[\ell]$. 
The order of the quotient $G/G^+$ is prime to $\ell$\footnote{The subgroup $G^+$ contains all $\ell$-Sylow subgroups $P$ 
of $G$ because $\ell\geq n$ forcing that every non-trivial element of $P\subset\GL_n(\F_\ell)$ 
has order $\ell$.}.
Nori's theory \cite{No87}
produces a connected semisimple by unipotent subgroup $\uS\subset\GL_{n,\F_\ell}$, 
called the \emph{Nori group} of $G\subset\GL_n(\F_\ell)$, that 
approximates $G^+$ whenever $\ell$ is large enough in terms of $n$.
We briefly review the construction of $\uS$.
Define $\mathrm{exp}$ and $\mathrm{log}$ by
$$\mathrm{exp}(x):=\sum_{i=0}^{\ell-1}\frac{x^i}{i!}\hspace{.2in}\mathrm{and}\hspace{.2in}
\mathrm{log}(x):=-\sum_{i=1}^{\ell-1}\frac{(1-x)^i}{i}.$$
The Nori group $\uS$ of $G$ is the subgroup of $\mathrm{GL}_{n,\F_\ell}$
 generated by the one-parameter subgroups
\begin{equation}\label{one}
t\mapsto x^t:=\mathrm{exp}(t\cdot\mathrm{log}(x))
\end{equation}
for all $x\in \finiteG[\ell]$. 
A subgroup of $\GL_{n,\F_\ell}$ is said to be \textit{exponentially generated}
if it is generated by the one-parameter subgroups $x^t$ for some set of unipotent elements $x\in\GL_n(\F_\ell)$.
We state the properties of $\uS$ we need.

\begin{thm}\label{Nori}\cite{No87}
Let $\uS\subset\GL_{n,\F_\ell}$ be the Nori group of $G\subset\GL_n(\F_\ell)$. There is a constant $C(n)$
depending only on $n$ such that the following holds whenever $\ell\geq C(n)$:
\begin{enumerate}[(i)]
\item $G^+=\uS(\F_\ell)^+$;
\item $\uS(\F_\ell)/\uS(\F_\ell)^+$ is abelian of order $\leq 2^{n-1}$;
\item if $G$ is semisimple on $\F_\ell^n$, then $\uS$ is a semisimple group with semisimple action on $\F_\ell^n$.
\end{enumerate}
\end{thm}

\subsection{Tame inertia weights}\label{tiw}
Let $F$ be a finite field extension of $\Q_\ell$ and $\overline F$ be an algebraic closure of $F$.
Let $\Gal_F:=\Gal(\overline F/F)$ be the absolute Galois group and $I$ be the inertia subgroup.
Denote by $I^w$ the wild inertia subgroup of $I$ and by $I^t:=I/I^w$ the tame inertia quotient.
Pick a uniformizer $\pi$ of $F$. For $d\in\N$, the tame inertia $I^t$ acts on 
the roots of $x^{\ell^d-1}-\pi=0$ inducing a group epimorphism to the subgroup of $(\ell^d-1)$th roots of unity of $\overline F^*$:
\begin{align*}
\begin{split}
\bar\theta_d: I^t&\longrightarrow \mu_{\ell^d-1}\\
g&\mapsto g(\pi^{\frac{1}{\ell^d-1}})/\pi^{\frac{1}{\ell^d-1}}
\end{split}
\end{align*}
which is independent of the choices of the uniformizer $\pi$ and 
the root $\pi^{\frac{1}{\ell^d-1}}$. Moreover,  
\begin{equation}\label{norm}
\bar\theta_d=(\bar\theta_{dd'})^{\frac{\ell^{dd'}-1}{\ell^d-1}}=(\bar\theta_{dd'})^{1+\ell^{d}+\ell^{2d}+\cdots+\ell^{(d'-1)d}}
\end{equation}
gives a canonical isomorphism $I^t\stackrel{}{\rightarrow} \varprojlim_{d}\mu_{\ell^d-1}$ \cite[Proposition 2]{Se72}. 

By identifying the residue field of $\overline F$ as $\overline\F_\ell$ 
and $\mu_{\ell^d-1}$ as its image $\F_{\ell^d}^*$ in $\overline\F_\ell$,
we obtain a canonical isomorphism $I^t\cong \varprojlim_{d} \F_{\ell^d}^*$  
where the inverse limit is given by norm maps $\F_{\ell^{dd'}}\to\F_{\ell^d}$ by \eqref{norm}
and we write $\bar\theta_d: I^t\to \F_{\ell^d}^*\subset\overline\F_\ell^*$ as an $\overline\F_\ell$-character.
The \emph{fundamental characters of level} $d$ of the tame inertia $I^t$ \cite[$\mathsection 1.7$]{Se72} are defined as
\begin{equation}\label{fundchar}
\bar\theta_d^{\ell^j},~j=0,1,...,d-1.
\end{equation}
If $e$ is the ramification index of the extension $F/\Q_\ell$, the $e$th power of the 
mod $\ell$ cyclotomic character $\bar\epsilon_\ell:I^t\to \Aut(\mu_\ell)=\F_\ell^*$ is
exactly $\bar\theta_1$, the fundamental character of level $1$ \cite[Proposition 8]{Se72}. 

Any character $\bar\chi: I^t\to \F_{\ell^d}^*$ 
can be written uniquely as a product of the fundamental characters of level $d$ in \emph{$\ell$-restricted form}:
$$\bar\chi=(\bar\theta_d)^{m_0}\cdot (\bar\theta_d^{\ell})^{m_1}\cdots  (\bar\theta_d^{\ell^{d-1}})^{m_{d-1}}$$
such that the \emph{exponents} $m_j$ belong to $[0,\ell-1]$ and are not all equal to $\ell-1$. 
Let $\bar\phi:I\to \GL_d(\F_\ell)$ be an irreducible representation of the inertia group.
Since $\bar\phi$ factors through $I^t$ and the commutant of $\bar\phi$ in $\End(\F_\ell^d)$ is a subfield isomorphic to $\F_{\ell^d}$,
the representation is abelian and we get a character $\bar\chi_\phi: I^t\to \F_{\ell^d}^*$ $(\subset \GL_d(\F_\ell))$.
The multiset of \emph{tame inertia weights} of $\bar\phi$ is defined as the multiset of the exponents of $\bar\chi_\phi$ as 
a product of the fundamental characters of level $d$ in $\ell$-restricted form.
One sees easily that $\bar\phi: I^t\to \GL_d(\overline\F_\ell)$ is isomorphic to the direct sum of the $\Gal(\overline\F_\ell/\F_\ell)$-conjugates
of $\bar\chi_\phi: I^t\to \F_{\ell^d}^*\subset \overline\F_\ell^*$.
If $\bar\phi: \Gal_F\to \GL_n(\F_\ell)$ is an $n$-dimensional representation of the local field $F$, 
the multiset of tame inertia weights of $\bar\phi$ is
defined as the multiset 
of $n$ numbers given by the (multiset) union of the tame inertia weights of the Jordan-H\"older factors of $\bar\phi|_I$.
Note that the tame inertia weights are unchanged if $\bar\phi$ is restricted to $\Gal_{F'}$ where $F'/F$ is unramified.
When $\F_q$ is a degree $d$ field extension of $\F_\ell$ and $\bar\phi: \Gal_F\to \GL_n(\F_q)$ is a representation,
the multiset of tame inertia weights of $\bar\phi$ (in this article) is defined as 
that of $\bar\phi$ when viewed as an ($dn$-dimensional) $\F_\ell$-representation.

\subsection{Big Galois images and algebraic envelopes}\label{ael}
Let $X$ be a smooth proper variety defined over a number field $K$.
Fixing $w\in\Z_{\geq 0}$, the absolute Galois group $\Gal_K$ acts on the $\ell$-adic 
cohomology group $V_\ell:=H^w(X_{\overline K},\Q_\ell)$ for each $\ell$.
Deligne \cite{De74} proved that the family  
\begin{equation}\label{sys0}
\{\rho_\ell:\Gal_K\to\GL(V_\ell)\cong\GL_n(\Q_\ell)\}_\ell
\end{equation}
of Galois representations is a rational SCS that is pure of weight $w$.

Denote by $\bar\Gamma_\ell\subset\GL_n(\F_\ell)$
the image of $\bar\rho_\ell^{\ss}:\Gal_K\to \GL_n(\F_\ell)$, the semisimplified
reduction of $\rho_\ell$. Since there is a natural surjection $\Gamma_\ell\to\bar\Gamma_\ell$,
the largeness of the compact subgroup $\Gamma_\ell$ in $\bG_\ell(\Q_\ell)$ 
implies the largeness of the finite image $\bar\Gamma_\ell$.
Conjecturally (see \cite{Se94,La95}), the representation $\rho_\ell$
is semisimple (or equivalently, $\bG_\ell$ is reductive) for all $\ell$ 
and the monodromy group $\Gamma_\ell$ is comparable to a 
hyperspecial maximal compact subgroup $\Omega_\ell$ of 
$\bG_\ell^\circ(\Q_\ell)$ uniformly for almost all $\ell$
\footnote{For $\ell\gg0$, the connected $\bG_\ell^\circ/\Q_\ell$ 
is an unramified reductive group and the indices $[\Gamma_\ell:\Gamma_\ell\cap\Omega_\ell]$ 
and $[\Omega_\ell:\Gamma_\ell\cap\Omega_\ell]$ are bounded by a constant independent of $\ell$.}. 
Hence, for almost all $\ell$ the reduction $\bar\Gamma_\ell$ should be comparable to $\uG_\ell(\F_\ell)$
uniformly, where $\uG_\ell/\F_\ell$ is some connected reductive group such that 
the absolute root data of $\uG_\ell$ and $\bG_\ell^\circ$ are isomorphic\footnote{
For $\ell\gg0$, the reductive groups $\uG_\ell$ and $\bG_\ell^\circ$ should respectively be
the base change of the Mumford-Tate group $\bG/\Q$ of the Hodge structure 
$H^w(X(\C),\Q)$ to $\F_\ell$ and $\Q_\ell$.}.
A prototypical result in this direction is the following well-known theorem of Serre.

\begin{thm}\cite{Se72}\label{Serre2}
If $X/K$ is an elliptic curve without complex multiplication over $\overline K$
 and $w=1$, then $\bG_\ell=\GL_2$
for all $\ell$.  Furthermore, for almost all $\ell$,
$\Gamma_\ell\cong\GL_2(\Z_\ell)$ and therefore $\bar\Gamma_\ell\cong\GL_2(\F_\ell)$.
\end{thm}

On the other hand, for any smooth proper variety $X/K$ and $w\in\Z_{\geq0}$ we constructed for each $\ell\gg0$ 
a connected reductive subgroup $\uG_\ell\subset\GL_{n,\F_\ell}$
with nice properties called the \emph{algebraic envelope}
of $\rho_\ell$ to study the largeness of $\Gamma_\ell$ in $\bG_\ell(\Q_\ell)$, 
see \cite{HL16} when $\bG_\ell$ is of type A
and \cite{HL20} when $X$ is an abelian or hyperk\"ahler variety.
The properties on the algebraic envelopes $\uG_\ell$ are summarized as follows.

\begin{thm}\label{Hui2}
Let $\{\rho_\ell\}_\ell$ be the rational SCS \eqref{sys0}.
There exist a finite Galois extension $L/K$ (with $K^{\conn}\subset L$) and for each $\ell\gg0$, 
a connected reductive subgroup $\uG_\ell\subset\GL_{n,\F_\ell}$ 
with properties below.
\begin{enumerate}[(i)]
\item The derived group $\uG_\ell^{\der}$ is the Nori group of $\bar\rho_\ell^{\ss}(\Gal_K)\subset\GL_n(\F_\ell)$.
\item The image $\bar\rho_\ell^{\ss}(\Gal_L)$ is a subgroup of $\uG_\ell(\F_\ell)$ 
with index bounded by a constant independent of $\ell$. 
\item The action of $\uG_\ell$ on the ambient space is semisimple.
\item The formal characters of $\uG_\ell\subset\GL_{n,\F_\ell}$ for all $\ell\gg0$ are bounded.
\item The formal bi-characters of $\bG_\ell$ and $\uG_\ell$ are the same and independent of $\ell$.
\item The commutants of $\bar\rho_\ell^{\ss}(\Gal_L)$ and $\uG_\ell$ (resp. 
$[\bar\rho_\ell^{\ss}(\Gal_L),\bar\rho_\ell^{\ss}(\Gal_L)]$ and $\uG_\ell^{\der}$) in $\End(\F_\ell^n)$ are equal.
\item If $\{\rho_\ell\}_\ell$ is connected, one can take $L=K$.
\end{enumerate}
The group $\uG_\ell$ is called the algebraic envelope of $\rho_\ell$ 
and is uniquely determined by properties (ii)--(iv) when $\ell$ 
is sufficiently large.
\end{thm}

\begin{remark}\label{pin}
\begin{enumerate}[(1)]
\item Construction of $\uG_\ell\subset\GL_{n,\F_\ell}$ satisfying properties (i)--(iv) 
was done in \cite[Theorem 2.0.5]{Hu15}, whose proof relies only on Nori group ($\mathsection\ref{N}$), 
some boundedness results on the local representations $\bar\rho_\ell^{\ss}|_{\Gal_{K_v}}$ for $v\in\Sigma_K$ 
(e.g., the tame inertia weights ($\mathsection\ref{tiw}$) when $v|\ell$), and the compatibility condition of $\{\rho_\ell\}_\ell$. 

\item Property (vi) follows from \cite[Proposition 3.1(iv)]{HL20} (resp. \cite[Proposition 3.1(iv)]{HL20} and property (ii)).
Property (vii) is \cite[Theorem 4.5]{HL20}. 
The uniqueness of algebraic envelope $\uG_\ell$ for $\ell\gg0$ is \cite[Proposition 4.4]{HL20}.

\item By the uniqueness of algebraic envelope, $\bG_\ell$ and $\uG_\ell$ are respectively the projection $p_1$ of the algebraic monodromy group $\bG_\ell'$
and algebraic envelope $\bar\bG_\ell'$ in the proof of \cite[Theorem 3.2.1]{Hu15}
for $\ell\gg0$. Following the notation in the proof of \cite[Theorem 3.2.1]{Hu15},
the formal characters of 
$\bG_\ell$ and $\uG_\ell$ are respectively
$$\bT_\ell''\subset p_1(\bT_\ell')\subset\GL_{n,\overline\Q_\ell} 
\hspace{.2in}\text{and}\hspace{.2in}\bar\bT_\ell''\subset p_1(\bar\bT_\ell')\subset\GL_{n,\overline\F_\ell}$$
and they are equal. Since the former formal-bi-character is independent of $\ell$ 
by Theorem \ref{Hui1},
property (v) follows.
\end{enumerate}
\end{remark}

\section{Big monodromy of subrepresentations of SCS}\label{s3}

\subsection{Algebraic envelopes of SCS}\label{aeSCS}
Let $\{\rho_\ell:\Gal_K\to\GL_n(\Q_\ell)\}_\ell$ be a semisimple rational SCS.
Then the semisimple mod $\ell$ system is 
$\{\bar\rho_\ell^{\ss}:\Gal_K\to\GL_n(\F_\ell)\}_\ell$.
In this section, we pin down the exact local conditions\footnote{They are exactly conditions (a) and (b) of Theorem \ref{ae}.} on 
$\bar\rho_\ell^{\ss}$
required to construct algebraic envelopes $\uG_\ell\subset\GL_{n,\F_\ell}$
with nice properties like Theorem \ref{Hui2}.

\subsubsection{Criteria and properties of algebraic envelope $\uG_\ell$} 

\begin{thm}\label{ae}
Let $\{\rho_\ell\}_\ell$ be a semisimple rational SCS of $K$.
Suppose there exist integers $N_1,N_2\geq 0$ and a finite extension $K'/K$
such that the following conditions hold.
\begin{enumerate}[(a)]
\item (Bounded tame inertia weights): for all $\ell\gg0$ and each finite place $v$ of $K$ above $\ell$, 
the tame inertia weights of the local representation $(\bar\rho_\ell^{\ss}\otimes\bar\epsilon_\ell^{N_1})|_{\Gal_{K_v}}$ belong to $[0,N_2]$.
\item (Potential semistability): for all $\ell\gg0$ and each finite place $w$ of $K'$ not above $\ell$,
the semisimplification of the local representation $\bar\rho_\ell^{\ss}|_{\Gal_{K_{w}'}}$ is unramified.
\end{enumerate}
Then there exist a finite Galois extension $L/K$ (with $K^{\conn}\subset L$) and for each $\ell\gg0$, 
a connected reductive subgroup $\uG_\ell\subset\GL_{n,\F_\ell}$ 
satisfying properties (i)--(vi) of Theorem~\ref{Hui2}.
The group $\uG_\ell$ is called the algebraic envelope of $\rho_\ell$ 
and is uniquely determined by properties (ii)--(iv) when $\ell$ 
is sufficiently large.
\end{thm}

\begin{proof}
If $N_1=0$, the construction of $\uG_\ell$ for $\ell\gg0$ with properties (i)--(iv) is 
exactly the proof of \cite[Theorem 2.0.5]{Hu15},
and the property (v) follows exactly from the proofs of \cite[Theorems 3.1.1, 3.2.1]{Hu15}. 
Although in \cite{Hu15} the SCS $\{\rho_\ell\}_\ell$
is the dual of the $i$th \'etale cohomology of a smooth projective variety $X/K$, what we needed in the construction 
are just the compatibility condition of $\{\rho_\ell\}_\ell$, the conditions \ref{ae}(a),(b), and some purely group theoretic results.
Note that condition \ref{ae}(a) is used throughout \cite[$\s 2,\s 3$]{Hu15} while condition 
\ref{ae}(b) is used only in the proof of \cite[Theorem 2.4.2]{Hu15}.

If $N_1>0$, consider the $(n+1)$-dimensional SCS 
$$\{\phi_\ell:=(\rho_\ell\otimes\epsilon_\ell^{N_1})\oplus \epsilon_\ell^{N_1}\}$$
with algebraic monodromy groups $\bH_\ell\subset\GL_{n+1,\Q_\ell}$.
For $\ell\gg0$, the semisimplified reductions
$\bar\phi_\ell^{\ss}=(\bar\rho_\ell^{\ss}\otimes\bar\epsilon_\ell^{N_1})\oplus \bar\epsilon_\ell^{N_1}=\F_\ell^{n+1}$ satisfy (b) 
and when restricted to $\Gal_{K_v}$ for $v$ above $\ell$, have tame inertia weights belonging to $[0, \max\{N_1,N_2\}]$.
By the $N_1=0$ case, there exist a finite Galois extension $L/K$ and for $\ell\gg0$, 
a connected reductive group $\uH_\ell\subset\GL_{n+1,\F_\ell}$ satisfying properties (i)--(v).

By properties (ii) and (iv) and \cite[Proposition 3.1(iv)]{HL20}, the commutants 
of $\uH_\ell$ and its subgroup $\bar\phi_\ell^{\ss}(\Gal_L)$ in $\End(\F_\ell^{n+1})$ 
are equal for $\ell\gg0$.
By Proposition \ref{lem1}(i), for $\ell\gg0$
the action of $\uH_\ell$ on the ambient space also respects the decomposition 
$$(\bar\rho_\ell^{\ss}\otimes\bar\epsilon_\ell^{N_1})\oplus \bar\epsilon_\ell^{N_1}=\F_\ell^n\oplus\F_\ell,$$
which implies $\uH_\ell\subset\GL_{n,\F_\ell}\times\GL_{1,\F_\ell}$.

Now, consider the $(n+1)$-dimensional SCS
$$\{\psi_\ell:=\rho_\ell\oplus \epsilon_\ell^{N_1}\}$$
and the semisimplified reduction $\bar\psi_\ell^{\ss}=\bar\rho_\ell^{\ss}\oplus \bar\epsilon_\ell^{N_1}$.
Let $\iota:\GL_n\times\GL_1\to \GL_n\times\GL_1$ be the automorphism sending $(A,a)$ to $(a^{-1}A,a)$.
Since $\uH_\ell\subset\GL_{n,\F_\ell}\times\GL_{1,\F_\ell}$ for $\ell\gg0$, 
the equalities $\psi_\ell=\iota\circ\phi_\ell$ and  $\bar\psi_\ell^{\ss}=\iota\circ\bar\phi_\ell^{\ss}$ 
of representations extend to the level of $\bH_\ell$ and $\uH_\ell$ for $\ell\gg0$. 
Since the family $\{\uH_\ell\subset\GL_{n,\F_\ell}\times\GL_{1,\F_\ell}\}_{\ell\gg0}$ 
has bounded formal characters by (iv), the family 
$\{\iota(\uH_\ell)\subset\GL_{n,\F_\ell}\times\GL_{1,\F_\ell}\}_{\ell\gg0}$ 
also has bounded formal characters.
Hence, it follows that 
$$\{\bar\psi_\ell^{\ss}(\Gal_L)\subset\iota(\uH_\ell)\subset\GL_{n,\F_\ell}\times\GL_{1,\F_\ell}\}_{\ell\gg0}$$ 
satisfies (i)--(iv).

Next, let $\pi:\GL_n\times\GL_1\to \GL_n$ be the projection to the first factor.
For $\ell\gg0$, the equalities $\rho_\ell=\pi\circ\psi_\ell$ and $\bar\rho_\ell^{\ss}=\pi\circ\bar\psi_\ell^{\ss}$
of representations extend to the level of $\iota(\bH_\ell)$ and $\iota(\uH_\ell)$.
Define $\uG_\ell:=\pi\circ\iota(\uH_\ell)$ for $\ell\gg0$. Again, 
the family $\{\bar\rho_\ell^{\ss}(\Gal_L)\subset \uG_\ell\subset\GL_{n,\F_\ell}\}_{\ell\gg0}$
satisfies (i)--(iv). These properties together with the compatibility of $\{\rho_\ell\}_\ell$
imply (v), see the proof of \cite[Theorem 3.1.1]{Hu15}. 
Finally, observe that $\bG_\ell=\pi\circ\iota(\bH_\ell)$ for all $\ell$
and the derived groups of $\bG_\ell^{\circ},\iota(\bH_\ell^\circ),\bH_\ell^\circ$ 
(resp. $\uG_\ell,\iota(\uH_\ell),\uH_\ell$)
are canonically the same.
Therefore, (v) for $\bG_\ell$ and $\uG_\ell$ for $\ell\gg0$ follows from (v) for $\bH_\ell$ and $\uH_\ell$ for $\ell\gg0$.
Again by properties (ii),(iv) (resp. (v)) and \cite[Proposition 3.1(iv)]{HL20}, we obtain (vi). 

The uniqueness of algebraic envelope $\uG_\ell$ for $\ell\gg0$ follows from properties (ii)--(iv) and \cite[Proposition 4.4]{HL20}
\end{proof}

\begin{remark}\label{rem1}
\begin{enumerate}[(1)]
\item The condition $\ref{ae}$(a) allows the flexibility for the tame inertia weights to concentrate near $0$ or $\ell-1$.
For example, the system $\{\psi_\ell\otimes\psi_\ell^\vee\}_\ell$ satisfies $\ref{ae}$(a) when the tame inertia weights 
of the local representations (at places above $\ell$) 
of $\bar\psi_\ell^{\ss}$ belong to $[0,N_2/2]$ for all $\ell\gg0$.
\item We will see later in Proposition \ref{connect} that if 
$\{\rho_\ell\}_\ell$ is connected and satisfies some weight conditions, 
then one can take $L=K$ in Theorem \ref{ae} like Theorem \ref{Hui2}(vii).
\end{enumerate}
\end{remark}

\subsubsection{Total $\ell$-rank} 
Let $\Gamma\subset\GL_n(\Q_\ell)$ be a compact $\ell$-adic Lie subgroup with $\ell\geq 5$.
It fixes some $\Z_\ell$-lattice and we may assume $\Gamma\subset\GL_n(\Z_\ell)$ for simplicity.
The \emph{total $\ell$-rank} $\rk_\ell\Gamma$ of $\Gamma$ is defined (see \cite[$\mathsection2$]{HL16}, \cite[$\mathsection2.3$]{HL20})
as the \emph{total $\ell$-rank} $\rk_\ell\bar\Gamma$ of the mod $\ell$ reduction $\bar\Gamma\subset\GL_n(\F_\ell)$.
Roughly speaking, $\rk_\ell\bar\Gamma$ is the sum of ``ranks'' of the finite simple groups of Lie type in characteristic $\ell$ 
in the multiset of Jordan-H\"older factors of the finite group $\bar\Gamma$.
For example,  the total $\ell$-ranks of $\GL_n(\Z_\ell)$ and $\SL_n(\Z_\ell)$ are both equal to $n-1$.
One sees that the definition is independent of the choice of the lattice.

\begin{cor}\label{rank}
Let $\{\rho_\ell\}_\ell$ be the semisimple rational SCS in Theorem \ref{ae}. 
For $\ell\gg0$, the total $\ell$-rank of $\Gamma_\ell\cap \bG_\ell^{\der}(\Q_\ell)$ 
is equal to the rank of $\bG_\ell^{\der}$.
\end{cor}

\begin{proof}
Let $L$ be the finite extension of $K$ in Theorem \ref{ae}.
Since $L$ contains $K^{\conn}$, we have $\rho_\ell(\Gal_L)\subset\bG_\ell^\circ(\Q_\ell)$ 
for all $\ell$.
Suppose $\ell\geq \max\{5,n,[L:K]\}$. Then for $\ell\gg0$, we have
\begin{align*}
\begin{split}
\rk_\ell(\Gamma_\ell\cap \bG_\ell^{\der}(\Q_\ell))&=\rk_\ell(\rho_\ell(\Gal_L)\cap \bG_\ell^{\der}(\Q_\ell))
=\rk_\ell(\rho_\ell(\Gal_L))\\
=\rk_\ell(\bar\rho_\ell^{\ss}(\Gal_L))&=\rk_\ell(\bar\rho_\ell^{\ss}(\Gal_L)^+)=\rk_\ell(\bar\rho_\ell^{\ss}(\Gal_K)^+)\\
=\rk_\ell(\uG^{\der}_\ell(\F_\ell)^+)&=\rk_\ell(\uG^{\der}_\ell(\F_\ell))=\rank\uG_\ell^{\der}=\rank\bG_\ell^{\der},
\end{split}
\end{align*}
where the first two rows follow from the definitions of the total $\ell$-rank (and $|\bar\rho_\ell^{\ss}(\Gal_L)/\bar\rho_\ell^{\ss}(\Gal_L)^+|$ 
is prime to $\ell$)
and for the third row, the first equality is by property (i) in Theorem \ref{ae} and Theorem \ref{Nori}(i),
the second equality is by Theorem \ref{Nori}(ii), the third equality is  by \cite[Proposition 4(ii)]{HL16},
and the last one is by property (v) in Theorem \ref{ae}.
\end{proof}

\subsection{Auxiliary groups}\label{aux}
\subsubsection{Reductive group schemes $\cG_\ell$}
Let $\{\rho_\ell\}_\ell$ be the semisimple rational SCS in Theorem \ref{ae}.
The notation in $\mathsection\ref{aeSCS}$ 
remains in force.

\begin{prop}\label{smooth}
Let $F$ be a finite extension of $\Q_\ell$ with ring of integers $\cO$ and residue field $\F$. 
Let $\bG$ be a connected  reductive subgroup of $\GL_{n,F}$.
Let $\cG/\cO$ be a smooth closed subgroup scheme of $\GL_{n,\cO}$ with generic fiber $\bG$
such that the special fiber $\cG_{\F}$ has the same rank as $\bG$.
If $\ell$ is sufficiently large in terms of $n$, and
$\cG(\cO)$ contains a hyperspecial maximal compact subgroup $\Omega$ of $\bG(F)$,
then $\cG$ is a connected reductive group scheme over $\cO$.
\end{prop}

\begin{proof} It suffices to prove that the identity component of $\cG_\F$ is reductive by \cite[Proposition 3.1.12]{Co14}.
As $\Omega$ is hyperspecial, it is of the form $\cH(\cO)$, 
where $\cH$ is a connected reductive group scheme over $\cO$ with generic fiber $\bG$.
Since $\cG(\cO)$ is compact in $\cG(F)$ and contains the maximal compact $\Omega$, it follows that $\cG(\cO)=\Omega$.
The Jordan-H\"older factors of $\cG(\F)$ and $\cH(\F)$ are the same up to factors of $\Z/\ell\Z$ as 
the kernel of the surjection $\Omega\to \cG(\F)$ (resp. $\Omega\to \cH(\F)$) is pro-$\ell$. 
In particular, the Jordan-H\"older factors of $\cG_{\F}^\circ(\F)$ other than $\Z/\ell\Z$
are contained in those of $\cH(\F)$.
Let $\uU$ be the unipotent radical of $\cG_{\F}^\circ$ and $\cG_{\F}^{\circ,\red}$ be the reductive quotient $\cG_{\F}^\circ/\uU$.
Since $H^1(\F,\uU)$ is trivial, we obtain an exact sequence of groups
$$1\to \uU(\F)\to \cG_{\F}^\circ(\F)\to \cG_{\F}^{\circ,\red}(\F)\to 1.$$
As $\uU(\F)$ is an $\ell$-group, it follows that 
the Jordan-H\"older factors of $\cG_{\F}^{\circ,\red}(\F)$ other than $\Z/\ell\Z$
are contained in those of $\cH(\F)$.

For a connected almost simple algebraic group $\uJ$ over $\F$ of rank bounded by $n$,
the product of orders of all cyclic Jordan-H\"older factors of $\uJ(\F)$
is bounded by a constant $C'(n)$ depending only on $n$. 
Let $\uZ_1$ (resp. $\uZ_2$) be the connected center of $\cG_{\F}^{\circ,\red}$ (resp. $\cH_{\F}$).
Let $N_1$ (resp. $N_2$) be the product of orders of all cyclic Jordan-H\"older factors 
of $\cG_{\F}^{\circ,\red}(\F)$ (resp. $\cH(\F)$)
not isomorphic to $\Z/\ell\Z$. 
Since the sequence of groups
$$1\to \uZ_2(\F)\to \cH_{\F}(\F) \to (\cH_{\F}/\uZ_2)(\F)$$
is exact and $\cH_{\F}/\uZ_2$ is connected almost simple, it follows from above that 
\begin{equation*}
|\uZ_1(\F)|\leq N_1\leq N_2\leq C'(n)|\uZ_2(\F)|.
\end{equation*}
Hence,  if $\ell$ is sufficiently large in terms of $n$, we obtain
\begin{equation}\label{centerineq}
\dim\uZ_1\leq \dim \uZ_2.
\end{equation}

For $\ell \ge 5$ and $\uJ$ a connected almost simple algebraic group over $\F$, the group $\uJ(\F)$ 
has a unique non-abelian Jordan-H\"older factor which determines $\uJ$ up to isogeny.
Thus, it follows from above that the semisimple part of $\cG_{\F}^{\circ,\red}$ is isogenous to a normal factor of 
the adjoint quotient of $\cH_{\F}$.
By \eqref{centerineq} and 
$$\rank \cG_{\F}^{\circ,\red}=\rank\cG_{\F}^\circ=\rank \bG =\rank \cH_{\F},$$ 
we conclude that the connected reductive groups
$\cG_{\F}^{\circ,\red}$ and $\cH_{\F}$ have the same adjoint quotient as well as the 
dimension of center.
Since the two special fibers $\cG_{\F}$ and $\cH_{\F}$ have the same dimension, 
we conclude that $\cG_{\F}^\circ$
has trivial unipotent radical, i.e., $\cG_{\F}^\circ$ is reductive.
\end{proof}

\begin{prop}\label{scheme}
For all sufficiently large $\ell$, there exist a totally ramified extension $F_\ell/\Q_\ell$ 
with ring of integers $\mathcal{O}_\ell$,
an $\mathcal O_{\ell}$-lattice $\mathcal{L}_\ell\subset F_\ell^n$, and a closed reductive subscheme 
$\cG_\ell\subset\GL_{\mathcal{L}_\ell}\cong\GL_{n,\mathcal O_{\ell}}$
such that $\Gamma_\ell\cap \bG_\ell^\circ(\Q_\ell)\subset\cG_\ell(\mathcal O_{\ell})$
and the generic fiber $\cG_{\ell,F_\ell}=\bG_{\ell,F_\ell}^\circ$.
\end{prop}

\begin{proof}
For simplicity, assume $K=K^{\conn}$ and $\bG_\ell$ is connected reductive for all $\ell$.
Denote the connected center of $\bG_\ell$ by $\bZ_\ell$.
By Corollary \ref{rank} and \cite[Theorem 2.11(iii)]{HL20}, 
$\bG_\ell^{\der}$ is unramified over 
 a degree $12$ totally ramified extension $F_\ell^t/\Q_\ell$ for $\ell\gg0$.
Let $x_0$ be the centroid of a facet of the \emph{Bruhat-Tits building} $\mathcal B(\bG_\ell,F_\ell^t)$ (see \cite{Ti79},\cite{BT72})
stabilized by $\Gamma_\ell$. For field extension $K_\ell/F_\ell^t$, 
there is a canonical injection of buildings $\mathcal B(\bG_\ell,F_\ell^t)\to \mathcal B(\bG_\ell,K_\ell)$ 
(resp. for $\bG_\ell^{\der}\times F_\ell^t$, $\bZ_\ell\times F_\ell^t$) \cite[Theorem 2.1.1(ii)]{Lan00}
and there are canonical bijections 
\begin{equation}\label{BT}
\mathcal B(\bG_\ell^{\der},K_\ell)\times\mathcal B(\bZ_\ell,K_\ell)\cong\mathcal B(\bG_\ell^{\der}\times\bZ_\ell,K_\ell)
\cong \mathcal B(\bG_\ell,K_\ell)
\end{equation}
that are compatible with field extensions by \cite[Propositions 2.16, 2.17]{Lan00} (see also \cite[$\mathsection4.2.18$]{BT84}).

Let $(x_0',x_0'')\in \mathcal B(\bG_\ell^{\der},F_\ell^t)\times\mathcal B(\bZ_\ell,F_\ell^t)$
be the point corresponding to $x_0$ via \eqref{BT}. By construction, $x_0'$ is the centroid 
of a facet. Since $\bG_\ell^{\der}\times F_\ell^t$ is semisimple and unramified for $\ell\gg0$,
the point $x_0'$ becomes hyperspecial in $\mathcal B(\bG_\ell^{\der},F_\ell)$
for some totally ramified extension $F_\ell/F_\ell^t$ for $\ell\gg0$.
Since the torus $\bZ_\ell$ is unramified over $\Q_\ell$ for $\ell\gg0$ (Proposition \ref{LP1}),
we conclude that $x_0$ becomes hyperspecial 
in $\mathcal B(\bG_\ell,F_\ell)$ for $\ell\gg0$ (\cite[Lemma 2.4]{La95}).
Thus, the stabilizer of $x_0$ in $\bG_\ell(F_\ell)$ is a hyperspecial maximal 
compact subgroup $\Omega_\ell$
containing $\Gamma_\ell$ (by construction) for $\ell\gg0$. 
 Let $\mathcal{L}_\ell\subset F_\ell^n$ be an $\mathcal O_{\ell}$-lattice
stabilized by $\Omega_\ell$. Let $\cG_\ell$ denote the Zariski closure of $\Omega_\ell$ in $\GL_{n,\cO_\ell}$ endowed with its
reduced scheme structure.  Then $\cG_\ell$ is an affine smooth group scheme with constant rank over $\cO_\ell$
and generic fiber $(\bG_\ell)_{F_\ell}$ for $\ell\gg0$ by 
\cite[Theorem 9.1 and $\mathsection 9.2.1$]{CHT17}\footnote{Although \cite[$\mathsection9$]{CHT17} is proved in the setting 
$\cO_\ell=\Z_\ell$ for all $\ell$, the arguments there work equally well 
for the system $\{\rho_\ell:\Gal_K\to\GL_n(\cO_\ell)\}_{\ell\gg0}$, see also \cite[Proposition 1.3]{LP95}.}.  
By Proposition~\ref{smooth},
if $\ell$ is sufficiently large, $\cG_\ell$ is connected and reductive.
Finally, note that $F_\ell$ is also a totally ramified extension of $\Q_\ell$.\end{proof}

\begin{cor}\label{fb}
For all sufficiently large $\ell$, the special fiber $\cG_{\ell,\F_\ell}$ and $\bG_\ell$
have the same formal bi-character, which is independent of $\ell$.
\end{cor}

\begin{proof}
Consider the reductive group scheme $\cG_\ell$ for $\ell\gg0$ (Proposition \ref{scheme}).
The base change $\cG_\ell':=\cG_\ell\times \mathcal{O}_{\ell}'$ has a maximal torus $\cT_\ell'$ 
for some \'etale extension of $\mathcal{O}_{\ell}'$ of $\mathcal{O}_{\ell}$ \cite[Corollary 3.2.7]{Co14}.
The reductive group scheme $\cG_\ell'$ has a derived group $\cG_\ell'^{\der}$ \cite[Theorem 5.3.1]{Co14}
such that the special (resp. generic) fiber of 
\begin{equation}\label{fbO}
(\cT_\ell'\cap\cG_\ell'^{\der})\subset\cT_\ell'\subset\GL_{n,\mathcal{O}_{\ell}'}
\end{equation}
gives the formal bi-character of $\cG_{\ell,\F_\ell}$ (resp. $\bG_\ell$) \cite[Proposition 5.3.4]{Co14}. As \eqref{fbO} can be descended to a chain of split subtori of $\GL_{n,\Z}$,
the special fiber of $\cG_{\ell,\F_\ell}$ and $\bG_\ell$
have the same formal bi-character.
The $\ell$-independence assertion follows  from Theorem \ref{Hui1}.
\end{proof}

\subsubsection{Intermediate group $\uG_\ell'$}
Consider the mod $\ell$ representation 
$$\widetilde\rho_\ell:\Gal_{K^{\conn}}\to \cG_\ell(\mathcal O_{\ell})
\subset\GL(\mathcal{L}_\ell)\cong\GL_n(\mathcal O_{\ell})\to\GL_n(\F_\ell)$$
for $\ell\gg0$, where $\cG_\ell$ is the reductive group scheme constructed in Proposition \ref{scheme}.
The semisimplification of $\widetilde\rho_\ell$ is isomorphic to $\bar\rho_\ell^{\ss}|_{\Gal_{K^{\conn}}}$ and the 
image $\widetilde\rho_\ell(\Gal_L)$ (as $K^{\conn}\subset L$) lies in the group of  $\F_\ell$-points of the special fiber $\cG_{\ell,\F_\ell}$, which is identified as a 
connected reductive subgroup of $\GL_{n,\F_\ell}$ by Proposition~\ref{scheme}.

\begin{prop}\label{inter}
For all sufficiently large $\ell$, there exists a connected subgroup $\uG_\ell'\subset\cG_{\ell,\F_\ell}$ 
satisfying the following properties.
\begin{enumerate}[(i)] 
\item The image $\widetilde\rho_\ell(\Gal_L)$ is a subgroup of $\uG_\ell'(\F_\ell)$. 
\item The two groups $\uG_\ell'$ and $\uG_\ell$ have the same rank and the same semisimple rank.
\item The semisimplification of the representation $\uG_\ell'\subset\GL_{n,\F_\ell}$ is the representation of the algebraic envelope $\uG_\ell\subset\GL_{n,\F_\ell}$.
\end{enumerate}
\end{prop}

\begin{proof}

For $\ell\gg0$, let $\uS_\ell'$ be the Nori group of $\widetilde\rho_\ell(\Gal_L)\subset\GL_n(\F_\ell)$ 
and let  $\uS_\ell''$ be the Nori group of $\cG_\ell(\F_\ell)\subset\GL_n(\F_\ell)$.
By the $\ell$-independence of 
the formal bi-character of $\cG_{\ell,\F_\ell}$ for $\ell\gg0$ (Corollary \ref{fb})
and \cite[Proposition 3.1]{HL20}, 
the Nori group $\uS_\ell''$ is the derived group of $\cG_{\ell,\F_\ell}$ for $\ell\gg0$
and thus we conclude $\uS_\ell'\subset\cG_{\ell,\F_\ell}$.
Define $\uG_\ell'$ as the subgroup of $\cG_{\ell,\F_\ell}$ generated by $\uS_\ell'$ and the connected center $\uZ_\ell$ of $\cG_{\ell,\F_\ell}$
for $\ell\gg0$. The subgroup $\uG_\ell'$ is equal to the product $\uS_\ell'\uZ_\ell$ 
and has unipotent radical equal to that of $\uS_\ell'$.
Since the semisimplification of $\widetilde\rho_\ell|_{\Gal_L}$ is $\bar\rho_\ell^{\ss}|_{\Gal_L}$,
the semisimplification of the representation $\uS_\ell'\subset\GL_{n,\F_\ell}$ is 
$\uG_\ell^{\der}\subset\GL_{n,\F_\ell}$ for $\ell\gg0$ by the construction of Nori groups ($\mathsection$\ref{N}) and Theorem \ref{ae}(i).
Hence, we obtain (ii) by the construction of $\uG_\ell'$, Theorem \ref{ae}(v), 
and the fact that $\cG_\ell$ is reductive with generic fiber $\bG^\circ_{\ell, F_\ell}$ (Proposition \ref{scheme}). 

Since $\widetilde\rho_\ell(\Gal_L)$ normalizes $\uS_\ell'$ (see $\mathsection$\ref{N}) and of course $\uZ_\ell$,
the product  $\uP_\ell:=\widetilde\rho_\ell(\Gal_L)\uG_\ell'$ is a subgroup of $\cG_{\ell,\F_\ell}$ 
with identity component $\uP_\ell^\circ=\uG_\ell'$ for $\ell\gg0$. Let $\uT_\ell$ be a maximal torus of $\uG_\ell'$.
For each component of $\uP_\ell$, choose an $\overline\F_\ell$-point $g$ that normalizes $\uT_\ell$.
Since $\uT_\ell$ is also a maximal torus of  $\cG_{\ell,\F_\ell}$ by (ii), 
we obtain a function $f$ from the component group $\uP_\ell/\uP_\ell^\circ$ to the Weyl group of $\cG_{\ell,\F_\ell}$ which 
has size bounded by some constant $C'(n)$ depending only on $n$. The function $f$ is injective since 
the centralizer of $\uT_\ell$ in $\cG_{\ell,\F_\ell}$ is $\uT_\ell$ which belongs to $\uG_\ell'$.
Hence, we conclude that 
\begin{equation}\label{lowbdd}
|\uP_\ell/\uP_\ell^\circ|\leq C'(n)\hspace{.2in} \text{for}\hspace{.1in} \ell\gg0.
\end{equation}

Let $\uP_\ell^{\red}\subset\GL_{n,\F_\ell}$ be the semisimplification of  
the representation $\uP_\ell\subset\GL_{n,\F_\ell}$ with 
\begin{equation}\label{redp}
r:\uP_\ell\to \uP_\ell^{\red}
\end{equation}
is the morphism. For $\ell> C'(n)$,
the kernel of \eqref{redp}
 is the unipotent radical of $\uS_\ell'$, which is connected.
Thus, $\uP_\ell$ and $\uP_\ell^{\red}$ have the same number of connected components.
We would like to show that this number is $1$ for $\ell\gg0$.
The image $G$ of $\Gal_L$ in $\uP_\ell^{\red}(\F_\ell)$ acts semisimply on the ambient space because
the Nori group of $G$ is by construction the semisimplification 
of $\uS_\ell'\subset\GL_{n,\F_\ell}$  and
$G^+$ is a normal subgroup of $G$ of prime to $\ell$ index.
Hence, for $\ell\gg0$ the representation $\Gal_L\to \uP_\ell^{\red}(\F_\ell)\subset\GL_n(\F_\ell)$
is just $\bar\rho_\ell^{\ss}|_{\Gal_L}$ and the algebraic envelope $\uG_\ell$ enters the picture:
\begin{equation}\label{pic}
\bar\rho_\ell^{\ss}(\Gal_L)\subset \uP_\ell^{\red} \cap \uG_\ell \subset\GL_{n,\F_\ell}.
\end{equation}
By  Theorem \ref{ae}(iv) and Corollary \ref{fb}, 
$\{\uP_\ell^{\red,\circ}\}_{\ell\gg0}$ and
$\{\uG_\ell\}_{\ell\gg0}$ have bounded formal characters.
By \eqref{lowbdd}, \eqref{pic},  and Theorem \ref{ae}(ii), the ratio
$$\frac{|\uG_\ell(\F_\ell)\cap \uP_\ell^{\red,\circ}(\F_\ell)|}{|\uG_\ell(\F_\ell)|}$$
 has a positive lower bound as $\ell\to \infty$.
By \cite[Proposition 2.23]{HL20}, it follows that $\uG_\ell\subset \uP_\ell^{\red,\circ}$ for $\ell\gg0$. 
Since $\uP_\ell:=\widetilde\rho_\ell(\Gal_L)\uG_\ell'$, 
it follows that 
$$\uP_\ell^{\red}=r(\widetilde\rho_\ell(\Gal_L))\cdot r(\uG_\ell')=\bar\rho_\ell^{\ss}(\Gal_L)\cdot r(\uG_\ell')\subset \uG_\ell \cdot r(\uG_\ell')\subset \uP_\ell^{\red,\circ}$$
because $\uG_\ell'=\uS_\ell'\uZ_\ell$ is connected, here $r$ is the surjection \eqref{redp}.
Thus, $\uP_\ell^{\red}$ is connected, and the same is true for $\uP_\ell$, which implies (i).
As $\uP_\ell^{\red}\supset \uG_\ell$ is just the semisimplification of $\uP_\ell = \uG_\ell'$,
and the (connected) groups 
$\uG_\ell$ and $\uP_\ell^{\red}$ have the same dimension (see Theorem \ref{ae} and Proposition \ref{scheme}),
they coincide, giving (iii).
\end{proof}

\subsubsection{Irreducible decomposition}\label{irred}
Suppose we have an irreducible decomposition of the semisimple representation 
$\bG_\ell^\circ\to \GL_{n,\overline\Q_\ell}$ of the identity component
 $\bG_\ell^\circ$ over $\overline\Q_\ell$ for $\ell\gg0$:
\begin{equation}\label{decom1}
\overline\Q_\ell^n=\bigoplus_{j\in J} U_{\ell,j}.
\end{equation}
Since $\bG_\ell^\circ$ splits over an unramified extension of $\Q_\ell$ by Proposition \ref{LP1}, 
the decomposition \eqref{decom1} is also defined over a finite unramified extension $F_\ell'$ of $F_\ell$ (defined in Proposition~\ref{scheme}):
\begin{equation}\label{decom2}
(F_\ell')^n=\bigoplus_{j\in J} U_{\ell,j}'
\end{equation}
such that each $U_{\ell,j}'$ is an absolutely irreducible representation of $\bG_\ell^\circ$ over $F_\ell'$.
Let $\mathcal{O}_{\ell}'$ be the ring of integers of $F_\ell'$.
Define the following hyperspecial maximal bounded subgroup of $\bG_\ell^{\circ}( F_\ell')$:
\begin{equation}\label{biggroup}
\Omega_{\ell}':=\cG_\ell(\mathcal{O}_{\ell}'),
\end{equation} 
where $\cG_\ell/\cO_\ell$ is the reductive group scheme in Proposition \ref{scheme}.
For each $j$, choose an $\mathcal{O}_{\ell}'$-lattice $\mathcal{L}_{\ell,j}'$ of 
$U_{\ell,j}'$ stable under $\Omega_{\ell}'$. The direct sum 
$$\mathcal{L}_{\ell}':=\bigoplus_{j\in J} \mathcal{L}_{\ell,j}'$$
is a lattice of $(F_\ell')^n$ and we define
$\cH_{\ell}'$ to be the Zariski closure of $\Omega_{\ell}'$ in $\GL_{\mathcal{L}_{\ell}'}$ with reduced structure. 

\begin{prop}\label{prep1}
For all sufficiently large $\ell$, the $\mathcal{O}_{\ell}'$-subscheme
\begin{equation}\label{chain1}
\cH_{\ell}'\subset\prod_{j\in J}\GL_{\mathcal{L}_{\ell,j}'}\subset\GL_{\mathcal{L}_{\ell}'}
\end{equation}
is split reductive and there is an $\overline\F_\ell$-isomorphism $\alpha:\cH_{\ell,\overline\F_\ell}'\to \cG_{\ell,\overline\F_\ell}$ of reductive groups 
such that the two semisimple representations below 
\begin{align*}
\begin{split}
\cH_{\ell}'(\overline\F_\ell)&\stackrel{\alpha}{\longrightarrow} \cG_\ell(\overline\F_\ell)
\to\GL(\mathcal{L}_\ell\otimes_{\cO_\ell}\overline\F_\ell)\cong\GL_n(\overline\F_\ell)\\
\cH_{\ell}'(\overline\F_\ell)&\to \prod_{j\in J}\GL(\mathcal{L}_{\ell,j}'\otimes_{\cO_\ell'}\overline\F_\ell)
\subset\GL(\mathcal{L}_{\ell}'\otimes_{\cO_\ell'}\overline\F_\ell)\cong\GL_n(\overline\F_\ell)
\end{split}
\end{align*}
are equivalent, where $\cG_\ell\to\GL_{\mathcal{L}_\ell}$ is the reductive scheme in Proposition \ref{scheme}.
\end{prop}

\begin{proof}
Since $\cH_{\ell}'$ is the Zariski closure (with reduced structure) in $\GL_{\mathcal{L}_{\ell}'}$ 
of the image of 
$$\rho_\ell:\Gal_{K^{\conn}}\to\cG_\ell(\mathcal{O}_{\ell})\subset\Omega_{\ell}'
=\cH_{\ell}'(\mathcal{O}_{\ell}')\subset\GL(\mathcal{L}_{\ell}')\cong\GL_n(\mathcal{O}_{\ell}')$$
and $\{\rho_\ell\}_\ell$ is a SCS, the $\mathcal{O}_{\ell}'$-group scheme $\cH_{\ell}'$ for $\ell\gg0$
is smooth with constant rank by \cite[Theorem 9.1 and $\mathsection 9.2.1$]{CHT17}
and hence reductive by Proposition \ref{smooth}. 
Since $\bG_\ell^\circ$ splits over $F_\ell'$, the group scheme $\cH'_\ell$ splits over $\cO_\ell'$.
Let $\F_q$ be the residue field of $F_\ell'$.
Suppose $\ell\geq n$ is large enough so that 
the special fibers $\cG_{\ell,\F_q}$ and $\cH'_{\ell,\F_q}$ are reductive
and their actions on $\mathcal{L}_\ell\otimes_{\cO_\ell}\overline\F_q$
and $\mathcal{L}'_\ell\otimes_{\cO'_\ell}\overline\F_q$ are semisimple \cite{Ja97}.
Then the $\F_q$-rational points $\cG_{\ell}(\F_q)$ and $\cH'_{\ell}(\F_q)$
are also semisimple on the ambient spaces if $\ell$ is large in terms of $n$.
It follows from the Brauer-Nesbitt Theorem that the two semisimple representations
\begin{align*}
\begin{split}
\Omega_\ell'&\twoheadrightarrow \cG_\ell(\F_q)\to\GL(\mathcal{L}_\ell\otimes_{\cO_\ell}\F_q)\cong\GL_n(\F_q)\\
\Omega_\ell'&\twoheadrightarrow \cH_{\ell}'(\F_q)\to \GL(\mathcal{L}_{\ell}'\otimes_{\cO_\ell'}\F_q)\cong\GL_n(\F_q)
\end{split}
\end{align*}
are equivalent. For $\ell\gg0$, 
the formal character of the fibers of $\cG_\ell$ and $\cH_\ell'$
are all equal to some torus $T\subset\GL_n$, where $T$ is the characteristic torus of
some \emph{good place}\footnote{The set of good places is of positive Dirichlet density \cite[Proposition 7.2]{LP92}.} of $K^{\conn}$ by the first paragraph of the proof of \cite[Proposition 1.3]{LP95}. To finish the proof, it suffices to show that 
if $\uG_1$ and $\uG_2$ are connected reductive subgroups of $\GL_{n,\F_q}$ with formal character $T$
such that $\uG_1(\F_q)=\uG_2(\F_q)$, then $\uG_1=\uG_2$ when $\ell$ is sufficiently large 
in terms of $T$ and $n$. This follows from \cite[Proposition 2.23]{HL20} with 
$\F_\ell$ replaced by $\F_q$\footnote{The argument there also works for finite field $\F_q$.}.
\end{proof}

\subsection{Algebraic envelopes of subrepresentations of SCS}\label{main}
\subsubsection{Setting}\label{setting}
Let $\{\rho_\ell:\Gal_K\to \GL_n(\Q_\ell)\}_\ell$ be the semisimple rational SCS in Theorem \ref{ae}.
Given a (not necessarily irreducible) decomposition of the action of $\Gal_K$ or $\bG_\ell$ over $\overline\Q_\ell$:
\begin{equation}\label{decom3}
\rho_\ell\otimes\overline\Q_\ell=\overline\Q_\ell^n=\bigoplus_{i\in I} W_{\ell,i}.
\end{equation}
Again by reduction (modulo some $\overline\Z_\ell$-lattice of $W_{\ell,i}$ for each $i$) 
and semisimplification, we obtain a direct sum decomposition 
\begin{equation}\label{decom4}
\bar\rho_\ell^{\ss}\otimes\overline\F_\ell=\overline\F_\ell^n=\bigoplus_{i\in I} \overline{W}_{\ell,i}^{\ss}.
\end{equation}
of $\bar\rho_\ell^{\ss}(\Gal_K)$ and $\bar\rho_\ell^{\ss}(\Gal_L)$ over $\overline\F_\ell$,
where $L/K$ is the Galois extension in Theorem \ref{ae}. 
The decomposition \eqref{decom4} is also a decomposition of $\uG_\ell$ on $\overline\F_\ell^n$ by Theorem \ref{ae}(vi)
and Proposition \ref{lem1}(ii) for $\ell\gg0$.
Define the following notation.

\begin{itemize}
\item $n_i$: the $\overline\Q_\ell$-dimension of $W_{\ell,i}$.
\item $\bG_{W_{\ell,i}}$: the image of $\bG_{\ell,\overline\Q_\ell}\to \GL_{W_{\ell,i}}\cong\GL_{n_i,\overline\Q_\ell}$, which is also
the algebraic monodromy group of the Galois representation $W_{\ell,i}$.
\item $\uG_{W_{\ell,i}}$: the image of $\uG_{\ell,\overline\F_\ell}\to \GL_{\overline{W}_{\ell,i}^{\ss}}\cong\GL_{n_i,\overline\F_\ell}$,
is called the \emph{algebraic envelope} of the Galois representation $W_{\ell,i}$.
\end{itemize}

\subsubsection{Formal bi-characters of subrepresentations of SCS}
Recall that 
$F_\ell$ is the totally ramified extension of $\Q_\ell$  in Proposition \ref{scheme} with ring of integers $\cO_\ell$ and
$F_\ell'$ is the finite unramified extension of $F_\ell$ in $\mathsection\ref{irred}$ with ring of integers $\cO_\ell'$
and residue field $\F_q$.

\begin{prop}\label{mt}
Let $\{\rho_\ell\}_\ell$ be the semisimple rational SCS in Theorem \ref{ae}.
For all sufficiently large $\ell$, there exists 
a chain of split reductive group schemes defined over $\cO_\ell'$:
\begin{equation}\label{chain2}
\cH_{\ell}'\subset \prod_{j\in J}\GL_{m_j,\cO_\ell'}\subset\GL_{n,\cO_\ell'}
\end{equation}
such that for any component $W_{\ell,i}$ in any decomposition \eqref{decom3},
there exist a subset $J(i)$ of $J$ with $n_i=\sum_{j\in J(i)} m_j$ and a chain of 
connected subgroups\footnote{The notation $J$
depends on $\ell$ and $J(i)$, $m_j$, and $n_i$ all depend on $W_{\ell,i}$.} 
\begin{equation}\label{chain3}
\uG_{W_{\ell,i}}'\subset\uH_{W_{\ell,i}}'\subset\prod_{j\in J(i)}\GL_{m_j,\overline\F_\ell}\subset\GL_{n_i,\overline\F_\ell}
\end{equation}
satisfying the following.
\begin{enumerate}[(i)]
\item The representation given by the image of the generic fiber of $\cH_\ell'$ in \eqref{chain2}
$$\cH_{\ell,\overline\Q_\ell}'\hto\prod_{j\in J}\GL_{m_j,\overline\Q_\ell}\twoheadrightarrow 
\prod_{j\in J(i)}\GL_{m_j,\overline\Q_\ell}\subset\GL_{n_i,\overline\Q_\ell}$$
is isomorphic to $\bG_{W_{\ell,i}}^\circ\hto\GL_{n_i,\overline\Q_\ell}\subset\GL_{n_i,\overline\Q_\ell}$.
\item The group $\uH_{W_{\ell,i}}'$ is the image of the special fiber of $\cH_{\ell}'$ in \eqref{chain2}:
$$\cH_{\ell,\overline\F_\ell}'\hto\prod_{j\in J}\GL_{m_j,\overline\F_\ell}\twoheadrightarrow 
\prod_{j\in J(i)}\GL_{m_j,\overline\F_\ell}\subset\GL_{n_i,\overline\F_\ell}.$$
\item The groups $\uG_{W_{\ell,i}}'$ and $\uH_{W_{\ell,i}}'$ have the same rank and semisimple rank
and $\uG_{W_{\ell,i}}$ is isomorphic to the reductive quotient of $\uG_{W_{\ell,i}}'$. 
\item The semisimplification of $\uG_{W_{\ell,i}}'\hto\GL_{n_i,\overline\F_\ell}$ 
is isomorphic to the representation $\uG_{W_{\ell,i}}\hto\GL_{n_i,\overline\F_\ell}$.
\item The reductive subgroups $\bG_{W_{\ell,i}}$, $\uH_{W_{\ell,i}}'$, and $\uG_{W_{\ell,i}}$ 
have the same formal bi-character. Moreover, this formal bi-character has finitely many 
possibilities among all decompositions \eqref{decom3} and all $\ell$.
\end{enumerate} 
\end{prop}

\begin{proof}
Suppose $\ell$ is sufficiently large.
The chain \eqref{chain2} is defined as \eqref{chain1} in Proposition \ref{prep1}.
Since the representation $\bG_\ell^\circ\to\GL_{W_{\ell,i}}$ can be descended
to $F_\ell'$ by Proposition \ref{LP1},
there is a subset $J(i)\subset J$ (from \eqref{decom2}) such that 
\begin{equation}\label{onlyisom}
W_{\ell,i}\cong(\bigoplus_{j\in J(i)}U_{\ell,j}')\otimes_{F_\ell'}\overline\Q_\ell
\end{equation}
and we obtain assertion (i). Note that the two subspaces in \eqref{onlyisom} may not coincide in $\overline{\Q}_\ell^n$.

Since the special fiber of \eqref{chain2} is isomorphic to 
the representation $\cG_\ell(\overline\F_\ell)\hto\GL_n(\overline\F_\ell)$
by Proposition \ref{prep1}(ii),
the chain \eqref{chain3} is defined as the image of the chain
\begin{equation}\label{chainimage}
(\uG_{\ell,\overline\F_\ell}'\subset\cG_{\ell,\overline\F_\ell}\stackrel{\alpha^{-1}}{\cong}\cH_{\ell,\overline\F_\ell}'
\subset\prod_{j\in J}\GL_{m_j,\overline\F_\ell})\longrightarrow \prod_{j\in J(i)}\GL_{m_j,\overline\F_\ell}
\end{equation}
and we have assertion (ii), where $\uG_\ell'$ is constructed in Proposition \ref{inter}. 

Since $\uG_{W_{\ell,i}}'\subset \uH_{W_{\ell,i}}'$ 
is the image of $\uG_{\ell,\overline\F_\ell}'\subset\cG_{\ell,\overline\F_\ell}$
and the algebraic envelope $\uG_\ell$ and $\cG_{\ell,\overline\F_\ell}$
have the same ranks and semisimple ranks,
assertion (iii) follows from Proposition \ref{inter}(ii),(iii). 
Assertion (iv) follows from Proposition \ref{inter}(iii) and 
the definition of $\uG_{W_{\ell,i}}'$ in \eqref{chainimage}.

Let $\cT_\ell'$ be a maximal split torus of $\cH_\ell'$. 
Then $\cT_\ell'\cap \cH_\ell'^{\der}$ is a maximal split torus
of the derived group $\cH_{\ell}'^{\der}$ \cite[Proposition 5.3.4]{Co14}.
The formal bi-character of $\bG_{W_{\ell,i}}$ (resp. $\uH_{W_{\ell,i}}'$)
is the generic (resp. special) fiber of the image of 
$$(\cT_{\ell}'\cap \cH_{\ell}'^{\der}\subset\cT_{\ell}')\longrightarrow \prod_{j\in J(i)}\GL_{m_j,\cO_\ell'} $$
by \cite[Lemma 6.1.4]{Gi19}. Thus, the formal bi-character of $\bG_{W_{\ell,i}}$ and $\uH'_{W_{\ell,i}}$
are the same. Hence, the formal bi-character of  $\bG_{W_{\ell,i}}$ and $\uG_{W_{\ell,i}}$ 
are the same by (iii) and (iv). The last part of (v) follows 
from Theorem \ref{Hui1} and that the sub-collections
of weights of the formal bi-character has finitely many possibilities. 
\end{proof}

\subsubsection{Type A and irreducibility}
Since $\cH_{\ell}'$ is a connected split reductive group scheme over $\cO_\ell'$, 
 the following split reductive group schemes are well-defined \cite[$\mathsection5.3$, Theorem 6.1.17]{Co14}:
\begin{itemize}
\item $\cH_{\ell}'^{\sc}$: the universal cover of the derived group $\cH_{\ell}'^{\der}$ of $\cH_{\ell}'$.
\item $\pi_\ell: \cH_{\ell}'^{\sc}\to \cH_{\ell}'^{\der}$: the central isogeny over $\cO_\ell'$.
\item $\cH_{\ell}'^{\sc}=\prod_{r\in R}\cS_{\ell,r}'$ where each 
$\cS_{\ell,r}'$ is semisimple over $\cO_\ell'$ with irreducible root datum and $R$ is an index set.
\item $\cT_{\ell,r}'$: a split maximal torus of $\cS_{\ell,r}'$ for $r\in R$.
\end{itemize}

\begin{cor}\label{mc}
Suppose $\ell$ is sufficiently large and $\bG_{W_{\ell,i}}$ is of type A. 
Then the following holds.
\begin{enumerate}[(i)]
\item The groups $\uH_{W_{\ell,i}}'$, $\uG_{W_{\ell,i}}'$, $\uG_{W_{\ell,i}}$ coincide and are also of type A.
\item If $\bG_{W_{\ell,i}}^\circ\to \GL_{W_{\ell,i}}$ is irreducible, then 
$\uG_{W_{\ell,i}}$ and thus $\Gal_K$ (resp. $\Gal_{K''}$, where $K''$ is abelian over $K$) are irreducible on $\overline{W}_{\ell,i}^{\ss}$.
\end{enumerate}
\end{cor}

\begin{proof}
(i) Suppose $\ell$ is large enough so that $\uG_{W_{\ell,i}}$ is semisimple on the ambient space \cite{Ja97}.
Consider the representation 
$$f_{\ell,i}:\cH_{\ell}'^{\sc}\stackrel{\pi_\ell}{\rightarrow}\cH_{\ell}'^{\der}\hto\cH_{\ell}'
\hto\prod_{j\in J}\GL_{m_j,\cO_\ell'}\twoheadrightarrow\prod_{j\in J(i)}\GL_{m_j,\cO_\ell'}$$
with $J(i)$ defined in Proposition \ref{mt}.
The image $f_{\ell,i}(\cH_{\ell}'^{\sc}(\overline\Q_\ell))$ is the derived group $\bG_{W_{\ell,i}}^{\der}(\overline\Q_\ell)$.
If $\cS_{\ell,r}'(\overline\Q_\ell)$ is not of type A, then $f_{\ell,i}(\cS_{\ell,r}'(\overline{\mathcal{O}}_{\ell}))$ 
is trivial by the assumption where $\overline{\mathcal{O}}_{\ell}$ is the ring of integers of $\overline\Q_\ell$. 
Hence, $f_{\ell,i}(\cS_{\ell,r}'(\overline\F_\ell))$ is trivial 
by $\cS_{\ell,r}'(\overline{\mathcal{O}}_{\ell})\twoheadrightarrow\cS_{\ell,r}'(\overline\F_\ell)$,
which implies that $\uH_{W_{\ell,i}}'$ is of type $A$. By Proposition \ref{mt}(iii)
and \cite[Lemma 2]{HL16}, we have $\uG_{W_{\ell,i}}'=\uH'_{W_{\ell,i}}$ is reductive.
By Proposition \ref{mt}(iv), we obtain $\uG_{W_{\ell,i}}'=\uG_{W_{\ell,i}}$. 

(ii) The split torus $\prod_{r\in R}\cT_{\ell,r}'$ is a split maximal torus of $\cH_{\ell}'^{\sc}=\prod_{r\in R}\cS_{\ell,r}'$.
The generic (resp. special) fiber of $f_{\ell,i}$
is a semisimple representation determined by the weights of 
$f_{\ell,i}|_{\prod_{r\in R}\cT_{\ell,r}'}$.
By Proposition \ref{mt}(v) and the fact that
 $\pi_\ell$ can be descended to a morphism of Chevalley schemes over $\Z$ 
\cite[Theorems 6.1.16, 6.1.17]{Co14} 
with only finitely many possibilities (for fixed $n$),
the weights of 
$f_{\ell,i}|_{\prod_{r\in R}\cT_{\ell,r}'}$
have finitely many possibilities.
Hence for $\ell\gg0$, the highest weight theory and Steinberg tensor product theorem \cite[Theorem 13.1]{St68}
imply that $\cH_{\ell}'^{\sc}(\overline\Q_\ell)$ is irreducible on $W_{\ell,i}$
iff $\cH_{\ell}'^{\sc}(\overline\F_\ell)$ is irreducible on $\overline W_{\ell,i}^{\ss}$.
Since $f_{\ell,i}(\cH_{\ell}'^{\sc}(\overline\Q_\ell))=\bG_{W_{\ell,i}}^{\der}(\overline\Q_\ell)$
and $f_{\ell,i}(\cH_{\ell}'^{\sc}(\overline\F_\ell))=\uH_{W_{\ell,i}}'^{\der}(\overline\F_\ell)$,
the irreducibility of $\bG_\ell^\circ$ on $W_{\ell,i}$ (assumption) 
implies that $\uH_{W_{\ell,i}}'$ is irreducible on $\overline{W}_{\ell,i}^{\ss}$.
Therefore, $\uG_{W_{\ell,i}}=\uH_{W_{\ell,i}}'$ (by Corollary \ref{mc}(i)) is irreducible on $\overline{W}_{\ell,i}^{\ss}$ 
and the restriction to
$\bar\rho_\ell^{\ss}(\Gal_L)$ (resp. $\bar\rho_\ell^{\ss}(\Gal_{L'})$, where $L'$ is abelian over $L$) is also irreducible by Theorem \ref{ae}(vi) and Proposition \ref{lem1}(i). Thus, $\Gal_K$ and $\Gal_{K''}$ are also irreducible on $\overline{W}_{\ell,i}^{\ss}$.
\end{proof}

\subsection{Big monodromy results for $E$-rational SCS}\label{ESCS}

Theorem \ref{thmA} and Corollary \ref{corA} will be proven in this section.
Before that, we consider the setting of $E$-rational SCS. Let 
\begin{equation}\label{condSCS}
\{\rho_\lambda:\Gal_K\to\GL_n(E_\lambda)\}_{\lambda}
\end{equation}
be a semisimple Serre compatible system together with integers $N_1,N_2\geq 0$ and a finite extension $K'/K$
such that the following conditions hold.
\begin{enumerate}[(a)]
\item (Bounded tame inertia weights): for almost all $\lambda$ 
and each finite place $v$ of $K$ above $\ell$, 
the tame inertia weights of the local representation 
$(\bar\rho_{\lambda}^{\ss}\otimes\bar\epsilon_\ell^{N_1})|_{\Gal_{K_v}}$ belong to $[0,N_2]$.
\item (Potential semistability): for almost all $\lambda$ and each finite place $w$ of $K'$ not above $\ell$,
the semisimplification of the local representation $\bar\rho_{\lambda}^{\ss}|_{\Gal_{K_{w}'}}$ is unramified.
\end{enumerate}

\begin{remark}\label{cond}
Suppose $\{\rho^1_\lambda\}_\lambda$ and $\{\rho_\lambda^2\}_\lambda$ are two semisimple SCS of $K$ defined over $E$
satisfying the conditions (a) and (b) above (for some integers $N_1,N_2\geq 0$ and finite extension $K'/K$). 
Then the new SCS constructed under the operations in $\mathsection$\ref{op}
also satisfies the conditions above accordingly\footnote{Possibly with different integers and finite extension.}.
\end{remark}

Let $d$ be the degree $[E:\Q]$.
By restriction of scalars, 
we obtain an $dn$-dimensional rational semisimple SCS
\begin{equation}\label{res11}
\{\rho_{\ell}:=\bigoplus_{\lambda|\ell}\rho_{\lambda}:\Gal_K\to(\mathrm{Res}_{E/\Q}\GL_n)(\Q_\ell)\subset\GL_{dn}(\Q_\ell)\}_\ell.
\end{equation}
From the construction of \eqref{res11} and the definition of tame inertia weights ($\mathsection$\ref{tiw}),
it follows that $\{\rho_\ell\}_\ell$ satisfies Theorem \ref{ae}(a),(b)  
(with $N_1,N_2\geq0$, and $K'/K$)
 if and only if $\{\rho_\lambda\}_\lambda$
satisfies $\mathsection$\ref{ESCS}(a),(b) above  
(with $N_1,N_2\geq0$, and $K'/K$).
Whenever $\lambda$ divides $\ell$, it follows that $\rho_{\lambda}\otimes\overline\Q_\ell$
is a subrepresentation of $\rho_{\ell}\otimes\overline\Q_\ell$ \eqref{res11}.
Then Remark \ref{cond}, Theorem \ref{ae}, Proposition \ref{mt}, and Corollary \ref{mc} (applied to \eqref{res11}) imply that  
for almost all $\lambda$ and subrepresentation $W_\lambda$ of $\rho_\lambda\otimes\overline\Q_\ell$, 
there exist a finite Galois extension $L/K$ (in Theorem \ref{ae})
and four linear algebraic groups 
\begin{itemize}
\item $\uH_{W_{\lambda}}'$, 
\item $\uG_{W_{\lambda}}'$, 
\item $\uG_{W_{\lambda}}$ the algebraic envelope of $W_\lambda$, 
\item $\bG_{W_\lambda}$ the algebraic monodromy group of $W_\lambda$,
\end{itemize}
satisfying the assertions in Theorem \ref{ae}, Proposition \ref{mt}, and Corollary \ref{mc}, where
the first three are defined over $\overline\F_\ell$ and the last one is over $\overline\Q_\ell$. Therefore, we obtain the following.

\begin{thm}\label{general}
Let $\{\rho_\lambda\}_{\lambda}$ be a semisimple SCS of $K$ satisfying conditions (a) and (b) in $\mathsection$\ref{ESCS}
for some integers $N_1,N_2\geq 0$ and finite extension $K'/K$.
Then there exists a finite Galois extension $L/K$ such that for almost all $\lambda$
and for each subrepresentation $\sigma_\lambda:\Gal_K\to \GL(W_\lambda)$ of $\rho_\lambda\otimes\overline\Q_\ell$,
the following assertions hold.
\begin{enumerate}[(i)]
\item The inclusion $\bar\sigma_\lambda^{\ss}(\Gal_L)\subset \uG_{W_{\lambda}}$ holds.
\item The commutants of $\bar\sigma_\lambda^{\ss}(\Gal_L)$ and $\uG_{W_{\lambda}}$ 
(resp. $[\bar\sigma_\lambda^{\ss}(\Gal_L),\bar\sigma_\lambda^{\ss}(\Gal_L)]$ 
and $\uG_{W_{\lambda}}^{\der}$) in $\End(\overline W_\lambda^{\ss})$ are equal.
In particular, $\bar\sigma_\lambda^{\ss}(\Gal_L)$ (resp. $[\bar\sigma_\lambda^{\ss}(\Gal_L),\bar\sigma_\lambda^{\ss}(\Gal_L)]$) is irreducible on $\overline W_{\lambda}^{\ss}$ if and only if 
$\uG_{W_{\lambda}}$ (resp. $\uG_{W_{\lambda}}^{\der}$) is irreducible on $\overline W_{\lambda}^{\ss}$.
\item The reductive groups $\uG_{W_{\lambda}}$ and $\bG_{W_\lambda}$ have the same formal bi-character. Moreover, this formal bi-character has
finitely many possibilities depending on $\{\rho_\lambda\}_{\lambda}$.
\item If $\bG_{W_\lambda}$ is of type A, then $\uH_{W_{\lambda}}'$, $\uG_{W_{\lambda}}'$,  $\uG_{W_{\lambda}}$ coincide and are also of type A.
\item If $\bG_{W_\lambda}$ is of type A and $\bG_{W_\lambda}^\circ\to\GL_{W_\lambda}$ is irreducible, then 
$\uG_{W_{\lambda}}$ and thus $\Gal_K$ (resp. $\Gal_{K^{ab}}$, where $K^{ab}/K$ is the maximal abelian extension) 
are irreducible on $\overline{W}_{\lambda}^{\ss}$.
\end{enumerate}
\end{thm}

\begin{proof}
Only the ``in particular'' part in (ii) needs justification: it follows from 
the equality on commutants and Proposition \ref{lem1}(i).
\end{proof}

\subsubsection{\textbf{Proof of Theorem \ref{thmA}}}\label{pthmA}
Consider the semisimple $E$-rational compatible system  
\begin{equation}\label{use}
\{\rho_\lambda:\pi_1(X)\to\GL_n(E_\lambda)\}_\lambda
\end{equation}
 in Theorem \ref{thmA}.
By restriction of scalars, 
we obtain an $dn$-dimensional semisimple rational compatible system
\begin{equation*}\label{res12}
\{\rho_{\ell}:=\bigoplus_{\lambda|\ell}\rho_{\lambda}:\pi_1(X)\to(\mathrm{Res}_{E/\Q}\GL_n)(\Q_\ell)\subset\GL_{dn}(\Q_\ell)\}_\ell
\end{equation*}
such that $\rho_{\lambda}$ is a subrepresentation of $\rho_{\ell}\otimes\overline\Q_\ell$ 
whenever $\lambda$ divides $\ell$, where $d$ is the degree $[E:\Q]$.
For a pair of distinct primes $\ell_1\neq \ell_2$, the image of $\rho_{\ell_1}\oplus\rho_{\ell_2}$
contains an open subgroup isomorphic to the product of a pro-$\ell_1$ group and a pro-$\ell_2$ group
by Goursat's lemma. Hence, condition (ii) of \cite[$\mathsection1$]{Se81} is satisfied which implies
the existence of a closed point $x\in X$ such that 
$$(\rho_{\ell_1}\oplus\rho_{\ell_2})(\pi_1(X))=(\rho_{\ell_1,x}\oplus\rho_{\ell_2,x})(\Gal_{K(x)}),$$
where $\rho_{\ell,x}:=\rho_\ell\circ sp_x$ and $sp_x$ is the specialization map in \eqref{diag}.
This implies that for any two finite places $\lambda_1$ and $\lambda_2$ of $E$,
we can find some closed point $x\in X$ such that 
$$\rho_{\lambda_1}(\pi_1(X))=\rho_{\lambda_1,x}(\Gal_{K(x)})\hspace{.2in}\text{and}\hspace{.2in}
\rho_{\lambda_2}(\pi_1(X))=\rho_{\lambda_2,x}(\Gal_{K(x)}),$$
where $\rho_{\lambda_i,x}:=\rho_{\lambda_i}\circ sp_x$ for $i=1,2$.
Therefore, we obtain the following $\lambda$-independence 
assertions for the algebraic monodromy group 
$\bG_\lambda$ of $\rho_\lambda$ since they hold for SCS $\{\rho_{\lambda,x}\}_\lambda$.
\begin{enumerate}[(i)]
\item The component group $\bG_\lambda/\bG_\lambda^\circ$ is independent of $\lambda$.
\item There exists a minimal Galois cover $X^{\conn}$ of $X$ such that $\rho_\lambda(\pi_1(X^{\conn}))\subset\bG_\lambda^\circ$ for all $\lambda$.
\item The formal bi-character of $\bG_\lambda\subset\GL_{n,E_\lambda}$ is independent of $\lambda$.
\end{enumerate}
Pick a closed point $x_0\in X$ such that $\rho_{\ell_1}(\pi_1(X))=\rho_{\ell_1,x_0}(\Gal_{K(x_0)})$ for some prime $\ell_1$ 
and consider the  $E$-rational SCS
\begin{equation}\label{sp0}
\{\rho_{\lambda,x_0}:\Gal_{K(x_0)}\to\GL_n(E_\lambda)\}_\lambda
\end{equation}
with algebraic monodromy groups $\bG_{\lambda,x_0}$ (not necessarily reductive). It follows that 
\begin{enumerate}[(iv)]
\item $\bG_{\lambda,x_0}\subset \bG_\lambda$ have the same formal bi-character.
\end{enumerate}
Let $W_\lambda$ be a type A irreducible subrepresentation of $\rho_\lambda\otimes\overline\Q_\ell$
with algebraic monodromy group $\bG_{W_\lambda}$. Denote the image of $\bG_{\lambda,x_0}$ in $\GL_{W_\lambda}$ by $\bG_{W_\lambda,x_0}$.\\ 

The degree $k:=deg(X^{\conn}/X)=|\bG_\lambda/\bG_\lambda^\circ|$
is independent of $\lambda$ by assertions (i) and (ii) above. 
In this proof, we call $k$ the \emph{degree} of \eqref{use}.
Consider $\ell>\max\{k,n\}$ from now on and 
we prove Theorem \ref{thmA} by induction on the degree $k$.
When $k=1$, assertion (iv) and \cite[Lemma 2]{HL16} imply that the inclusion
$$\bG_{W_\lambda,x_0}\subset\bG_{W_\lambda}\hspace{.1in}\text{ (connected and of type A)}$$ 
is an equality. Thus, $\bG_{W_\lambda,x_0}$ is connected and irreducible on $W_\lambda$.
By Theorem \ref{general}(v) on the semisimplification of \eqref{sp0},  it follows that
$W_\lambda$ is a residually irreducible representation of $\Gal_{K(x)}$ and $\pi_1(X)$
for almost all $\lambda$. Hence, Theorem \ref{thmA} is true when $k=1$.
Now suppose $k>1$. By Remark \ref{cond} and the induction hypothesis,
we assume that Theorem \ref{thmA} is true for the $E$-rational compatible system $\{\mathrm{Res}^{\pi_1(X)}_{\pi_1(X')}\rho_\lambda\}_\lambda$
and $\{\mathrm{Res}^{\pi_1(X)}_{\pi_1(X')} (\rho_\lambda\otimes\rho_\lambda^{\vee})\}_\lambda$ having degrees
bounded by $deg(X^{\conn}/X')<k$, 
here $X'$ runs through subcovers (of $X^{\conn}/X$) such that
$$X^{\conn}\to X'\stackrel{\ncong}{\rightarrow} X.$$
Such $X'$ has finitely many possibilities.

The quotient $\bG_{W_\lambda}/\bG_{W_\lambda}^\circ$ corresponds to  a subcover $X_{W_\lambda}$ (of $X^{\conn}/X$)
\begin{equation}\label{finitep}
X^{\conn}\to X_{W_\lambda}\to  X
\end{equation}
such that $\Gal(X_{W_\lambda}/X)\cong\bG_{W_\lambda}/\bG_{W_\lambda}^\circ$.
The argument above for $k=1$ enables us to just 
consider the non-Lie-irreducible case.
Restricting the irreducible $\pi_1(X)$-representation
$W_\lambda$ to the normal subgroup $\pi_1(X_{W_\lambda})$ yields the following irreducible decomposition
\begin{equation}\label{nordecom}
\bigoplus_{i=1}^e(\overbrace{U_i\oplus\cdots\oplus U_i}^{f\hspace{.05in}\text{terms}})
\end{equation}
for some positive integers $e,f$ such that $U_i\ncong U_j$ if $i\neq j$ and $\dim U_i$ is independent of $i$
by Clifford \cite[Theorem 1]{Cl37} (or \cite[Proposition 24]{Se77a}). There are two cases to analyze (a): $e>1$ and (b): $e=1$ and $f>1$.

\textit{Case (a).} The stabilizer of the subspace $W_1:=\oplus^f U_1$ (in \eqref{nordecom}) 
in $\pi_1(X)$ corresponds to a subcover $X_1$ (of $X_{W_\lambda}/X$) 
$$X_{W_\lambda}\to X_1\to X$$
with $deg(X_1/X)=e>1$ and $W_1$ is a  $\pi_1(X_1)$-representation.
Frobenius reciprocity asserts that
$$\Hom_{\pi_1(X)}(W_\lambda, \mathrm{Ind}^{\pi_1(X)}_{\pi_1(X_1)} W_1)=\Hom_{\pi_1(X_1)}(\mathrm{Res}^{\pi_1(X)}_{\pi_1(X_1)}W_\lambda, W_1)\neq 0.$$
Together with the facts that $W_\lambda$ is irreducible and $\dim W_\lambda =\dim \mathrm{Ind}^{\pi_1(X)}_{\pi_1(X_1)} W_1$,
we obtain
\begin{equation}\label{irredind}
W_\lambda\cong \mathrm{Ind}^{\pi_1(X)}_{\pi_1(X_1)} W_1.
\end{equation}
It suffices to show that for almost all $\lambda$ and $W_\lambda$ in case (a), 
the semisimple $\mathrm{Ind}^{\pi_1(X)}_{\pi_1(X_1)} \overline{W}_1^{\ss}$ (as $k<\ell$)
is irreducible. This is because Frobenius reciprocity will then imply that
$\overline{W}_\lambda^{\ss}\cong\mathrm{Ind}^{\pi_1(X)}_{\pi_1(X_1)} \overline{W}_1^{\ss}$ is also irreducible.
By Mackey's irreducibility criterion (see \cite[$\mathsection7.4$]{Se77a}), that is equivalent to the two conditions.
\begin{enumerate}[(I)]
\item The $\pi_1(X_1)$-representation $\overline{W}_1^{\ss}$ is irreducible.
\item For $g\in \pi_1(X)\backslash \pi_1(X_1)$, the two semisimple\footnote{Note that $\overline{W}_1^{\ss}$ 
restricting to the normal subgroup $\pi_1(X_{W_\lambda})$ is semisimple and $[g \pi_1(X_1) g^{-1}:\pi_1(X_{W_\lambda})]<\ell$.}
 representations below have no common factor:
\begin{align*}
\begin{split}
&(g \pi_1(X_1) g^{-1})\cap \pi_1(X_1)\longrightarrow\GL(\overline{W}_1^{\ss}), \\
&(g \pi_1(X_1) g^{-1})\cap \pi_1(X_1)\longrightarrow  \GL(g\overline{W}_1^{\ss}),
\end{split}
\end{align*}
\end{enumerate}
which is equivalent to 
\begin{enumerate}
\item[(II')] the semisimple representation below has trivial invariants:
\begin{align}\label{Mackey}
\begin{split}
(g \pi_1(X_1) g^{-1})\cap \pi_1(X_1)\longrightarrow\GL(g\overline{W}_1^{\ss}\otimes\overline{W}_1^{{\ss},\vee}).
\end{split}
\end{align}
\end{enumerate}
The group $(g \pi_1(X_1) g^{-1})\cap \pi_1(X_1)$ corresponds to a subcover $X_1$ (of $X_{W_\lambda}/X_1$)
$$X_{W_\lambda}\to X_1'\to X_1.$$
Since $\overline{W}_1^{\ss}$ and $g\overline{W}_1^{\ss}\otimes\overline{W}_1^{{\ss},\vee}$ 
are respectively the semisimplified reduction of the semisimple representations 
$W_1$ and $gW_1\otimes W_1^{\vee}$ of $\pi_1(X_1')$, applying Mackey's irreducibility criterion to the irreducible \eqref{irredind} 
implies that (I) and (II') will follow from the claims below.
\begin{enumerate}
\item[(C1)] For almost all $\lambda$, any type A irreducible subrepresentation of the (semisimple) $\pi_1(X_1')$-representation $\rho_\lambda\otimes\overline\Q_\ell$ or $(\rho_\lambda\otimes\rho_\lambda^{\vee})\otimes\overline\Q_\ell$ is residually irreducible.
\item[(C2)] For almost all $\lambda$, any type A subrepresentation $R_\lambda$ of the (semisimple) 
$\pi_1(X_1')$-representation $(\rho_\lambda\otimes\rho_\lambda^{\vee})\otimes\overline\Q_\ell$
satisfies the following on invariant dimensions:
$$\dim R_\lambda^{\pi_1(X_1')}=\dim\overline R_\lambda^{\ss,\pi_1(X_1')}.$$
\end{enumerate}
Claim (C1) follows directly from the induction hypothesis 
taking $X'=X_1'$.
Claim (C2) follows from (C1) and the lemma below taking $X=X_1'$.
\begin{lemma}\label{easy}
For almost all $\lambda$, if $R_\lambda$ is a non-trivial 
one-dimensional $\pi_1(X)$-subrepresentation of $(\rho_\lambda\otimes\rho_\lambda^{\vee})\otimes\overline\Q_\ell$
then its semisimplified reduction is also non-trivial.
\end{lemma}

\begin{proof}
If the monodromy in $\GL(R_\lambda)$ is finite and non-trivial, then its size is bounded by $k<\ell$.
Since the kernel of the reduction map is pro-$\ell$, the monodromy in $\GL(\overline R_\lambda^{\ss})$
is non-trivial. 

If the monodromy in $\GL(R_\lambda)$ is infinite, then the algebraic monodromy 
$\bG_{R_\lambda}=\GL_1$. Thus, the algebraic monodromy of $\Gal_{K(x_0)}$ on $R_\lambda$
is also $\GL_1$ by assertion (iv). By Theorem \ref{general}(iii) on the semisimplification of \eqref{sp0}, for $\ell\gg0$
the algebraic envelope of $\Gal_{K(x_0)}$ on $R_\lambda$ is also $\GL_1$. 
For almost all $\lambda$, let $\uH_{\lambda}$ be the algebraic envelopes of the semisimple SCS
\begin{equation}\label{temp}
\{Y_{\lambda}:=\rho_{\lambda,x_0}^{\ss}\otimes\rho_{\lambda,x_0}^{\ss,\vee}\}_\lambda.
\end{equation}
By Theorem \ref{general}(i) and (ii), there exists $L'/K(x_0)$ such that the Galois image of $\Gal_{L'}$ in $\GL_{\overline{Y}_{\lambda}^{\ss}}$ 
factors through $\uH_\lambda$ and they have the same commutant in $\End(\overline{Y}_{\lambda}^{\ss})$ for $\ell\gg0$.
Since $Y_\lambda\otimes\overline\Q_\ell$ contains $R_\lambda$ and the trivial $\overline\Q_\ell$ as subrepresentations
and the image of $\uH_\lambda$ in $\GL_{\overline R_\lambda^{\ss}}$ is $\GL_1$ for $\ell\gg0$,
it follows that $\Gal_{L'}$ is non-trivial on $\overline R_\lambda^{\ss}$ by Proposition \ref{lem1}(i).
\end{proof}

\textit{Case (b).} In this case, $e=1$ and $f>1$ in the decomposition \eqref{nordecom}
and recall $\ell>\max\{k,n\}$.
By \cite[Theorem 3]{Cl37}, $W_\lambda$ is the tensor product of two irreducible 
projective representations\footnote{Given a group $G$ and a vector space $V$, a 
map $G\to \GL(V)$ is a projective representation
if the composition $G\to \GL(V)\to\PGL(V)$ is a group homomorphism.} of $\pi_1(X)$:
\begin{equation*}
 W_\lambda\cong U\otimes D
\end{equation*}
such that restricting $U$ to $\pi_1(X_{W_\lambda})$ is $U_1$ in \eqref{nordecom} and $D$ is actually a projective representation
of the quotient group $\Gal(X_{W_\lambda}/X)$ of dimension $f$.
Let $\Delta_U$ (resp. $\Delta_D$) be the image of $\pi_1(X)$ in $\GL(U)$ (resp. $\GL(D))$
and $\mu_f$ be the group of $f$th roots of unity.
Since $\Gal(X_{W_\lambda}/X)$ is finite of size bounded by $k<\ell$, we may adjust by multiplying 
suitable scalars so that  $\Delta_D$  is contained in $\SL(D)$ 
and the \emph{Schur multiplier}\footnote{Given a projective representation $\phi:G\to \GL_n(F)$,
the Schur multiplier is the function $c:G\times G\to F^*$ so that $\phi(gh)=c(g,h)\phi(h)\phi(g)$.} of $D$ (and hence $U$) has values in 
$$\mu_f=\text{Ker}(\SL(D)\to \PSL(D)).$$ 
Since restricting $U$ and $D$ 
to the normal subgroup $\pi_1(X_{W_\lambda})$ are true representations,
the product sets $\Delta_U\cdot\mu_f$ and $\Delta_D\cdot\mu_f$
are compact subgroups of $\GL(U)$ and $\GL(D)$ respectively.
By finding a $\overline\Z_\ell$-lattice of $U$ (resp. $D$) stabilized by $\Delta_U\cdot\mu_f$ 
(resp. $\Delta_D\cdot\mu_f$) and reduction, the residual representation $\overline W_\lambda$
is a tensor product of two projective representations of $\pi_1(X)$:
\begin{equation}\label{tensorproj}
 \overline W_\lambda\cong \overline U\otimes \overline D.
\end{equation}
It suffices to show that $\overline U$ and $\overline D$ are irreducible.
For almost all $\lambda$, the restriction of $\overline U$ to $\pi_1(X_{W_\lambda})$ is $\overline U_1$ which 
is irreducible by the facts that
\begin{itemize}
\item $X_{W_\lambda}$ in \eqref{finitep} has finitely many possibilities and 
\item the degree $k=1$ case for $\{\mathrm{Res}^{\pi_1(X)}_{\pi_1(X_{W_\lambda})}\rho_\lambda\}_\lambda$ 
because the image of $\pi_1(X_{W_\lambda})$ in $\GL(W_\lambda)$ is contained in $\bG_{W_\lambda}^\circ$. 
\end{itemize}
Since $\Delta_D$ has size
prime to $\ell$, it acts irreducibly on $\overline D$ \cite[Proposition 43(ii)]{Se77a}.
\qed

\subsubsection{\textbf{Proof of Corollary \ref{corA}}} 
Since $\dim W_\lambda\leq 3$ is of type A (including the possibility
that $\bG_{W_\lambda}^\circ$ is a torus), 
we are done by Theorem \ref{thmA}. \qed

\subsection{Connectedness}
Let $\{\rho_\ell\}_\ell$ be the rational 
semisimple SCS in Theorem \ref{ae}. We give conditions
so that a generalization of Theorem \ref{Hui2}(vii)
holds for $\{\rho_\ell\}_\ell$. 

\subsubsection{Compatible subsystem}
We use
the notation in $\mathsection\ref{csgr}$.
Denote the ring of integers of $\overline\Q$ by $\overline\Z$.
For every $\ell$, choose an embedding $\overline\Q\hto\overline\Q_\ell$.
Let $\mathcal{P}\subset\Sigma_\Q$ be a non-empty subset.
Suppose for every $\ell\in\mathcal{P}$, we have a subrepresentation 
$W_\ell$ of $\rho_\ell\otimes\overline\Q_\ell$.
The subsystem 
\begin{equation}\label{subsys}
\{\rho_{W_\ell}:\Gal_K\to \GL(W_\ell)\}_{\ell\in\mathcal{P}}
\end{equation}
is said to be \emph{compatible} if 
for all primes $\ell, \ell'\in\mathcal{P}$ and $v\in \Sigma_K\backslash (S\cup S_\ell\cup S_{\ell'})$, 
the equality $P_{v,W_\ell}(T)=P_{v,W_\ell'}(T)\in\overline\Q[T]$ holds.

\subsubsection{Frobenius tori}\label{hypo}
The algebraic monodromy group $\bG_{W_\ell}$ of $\rho_{W_\ell}$
is reductive. We state some facts for the compatible subsystem \eqref{subsys}.

\begin{enumerate}[(i)]
\item The component group $\bG_{W_\ell}/\bG_{W_\ell}^\circ$ and the rank of the $\bG_{W_\ell}$ 
are both independent of $\ell\in\mathcal{P}$\footnote{The proof goes similarly as Theorem \ref{Serre1}}.
\item Let $v\in \Sigma_K\backslash (S\cup S_\ell)$ and $\bar v$ a finite place of $\overline K$ extending $v$.
The Frobenius image $\rho_{W_\ell}(Frob_{\bar v})$ is well-defined.
\item (Serre) The \emph{Frobenius torus} at $\bar v$, denoted by $\bT_{\bar v,W_\ell}$,
 is defined to be the identity component of the smallest algebraic group containing the semisimple part of $\rho_{W_\ell}(Frob_{\bar v})$.
It is a subtorus of $\bG_{W_\ell}$. If $\bar v'$ is another finite place of $\overline K$ extending $v$, then $\bT_{\bar v,W_\ell}$
and $\bT_{\bar v',W_\ell}$ are conjugate in $\bG_{W_\ell}$.
\item The hypothesis\footnote{``MFT'' stands for ``maximal Frobenius tori''.} below is independent of $\ell\in\mathcal{P}$.
\begin{enumerate}
\item[(MFT):] The group $\bG_{W_\ell}$ is connected and there is a Dirichlet density one subset 
$\mathscr{S}_K$ of  $\Sigma_K$ such that if $v\in \mathscr{S}_K$,
then $W_\ell$ is unramified at $v$ and the Frobenius torus $\bT_{\bar v,W_\ell}$ 
is a maximal torus of $\bG_{W_\ell}$.
\end{enumerate}
\item (Serre) Suppose $\bG_{W_\ell}$ is connected. Hypothesis (MFT) holds when the roots of $P_{v,W_\ell}(T)\in\overline\Q[T]$
for all $v$ outside $S\cup S_\ell$ satisfy \cite[Theorem 2.6(a)--(c)]{Hu18}, e.g.,
when $W_\ell$ is (the semisimplification of) a subquotient of 
$$\bigoplus_{i=1}^m H^{w_i}(X_{i,\overline K},\overline\Q_\ell(n_i)),$$
where $X_i/K$ is a smooth projective variety for all $1\leq i\leq m$.
\end{enumerate}

\subsubsection{Galois image in $\uG_{W_\ell}$}
Since $W_\ell$ is a component of a decomposition of $\rho_\ell\otimes\overline\Q_\ell$ for all $\ell\in\mathcal{P}$,
we have $\uG_{W_\ell}\subset \GL(\overline W_\ell^{\ss})$ the algebraic envelope of $\rho_{W_\ell}$ for almost all $\ell$ in $\mathcal{P}$.
The following statement is a mild generalization of \cite[Theorem 4.5]{HL20} assuming 
Hypothesis $(MFT)$ in $\mathsection$\ref{hypo}(iv);
the difference is that the algebraic monodromy $\bG_\ell$ can be non-connected.

\begin{prop}\label{connect}
Suppose Hypothesis $(MFT)$ in $\mathsection$\ref{hypo}(iv) holds. 
For almost all $\ell\in\mathcal{P}$,
the image of $\bar\rho_{W_\ell}^{\ss}:\Gal_K\to \GL(\overline W_\ell^{\ss})$
is contained in $\uG_{W_\ell}$.
\end{prop}

\begin{proof}
Since $\bar\rho_\ell^{\ss}(\Gal_K)$ normalizes $\bar\rho_\ell^{\ss}(\Gal_L)$,
which is a subgroup of $\uG_\ell(\F_\ell)$ of uniformly bounded index independent of $\ell\gg0$ 
and the formal character of $\uG_\ell$ is independent of $\ell\gg0$ (Theorem \ref{ae}(ii),(v)),
the image $\bar\rho_\ell^{\ss}(\Gal_K)$ normalizes the algebraic envelope $\uG_\ell$ 
for $\ell\gg0$ by \cite[Proposition 2.23]{HL20}.
Thus, for $\ell\gg0$ the product $\bar\rho_\ell^{\ss}(\Gal_K)\cdot\uG_\ell$ is reductive with 
the identity component $\uG_\ell$.
Projecting into $\GL_{\overline W_\ell^{\ss}}$, for $\ell\gg0$ the product 
$$\bar\rho_{W_\ell}^{\ss}(\Gal_K)\cdot\uG_{W_\ell}\subset\GL_{\overline W_\ell^{\ss}}$$ 
is a reductive subgroup with identity component $\uG_{W_\ell}$ and number of components is bounded by $[L:K]$.
It suffices to show that this product is connected when $\ell$ is large.
The proof goes exactly the same as the proof of \cite[Theorem 4.5]{HL20} using Hypothesis (MFT)
and the only difference comes from the fields of definition of Frobenius tori:
in  \cite[Theorem 4.5]{HL20} 
the Frobenius tori are defined over $\Q_\ell$ (with an underlying $\Q$-rational torus),
whereas here the Frobenius tori are defined over $\overline\Q_\ell$
(with an underlying $\overline{\Q}$-rational torus).
\end{proof}

\section{Irreducibility of automorphic Galois representations}\label{s4}
\subsection{Algebraic envelopes of automorphic SCS}\label{ic}
Let $F$ be a totally real or CM number field and $F^+$ the maximal totally real subfield of $F$.   
Attached to a regular algebraic, polarized, cuspidal 
automorphic representation $(\pi,\chi)$ of $\GL_n(\A_F)$ is a pair of strictly compatible systems ($\mathsection$\ref{wcs})
defined over a CM field $E$\footnote{By \cite[Lemmas 1.2, 1.4]{PT15} and \cite[Lemma 5.3.1(3)]{BLGGT14}, 
we obtain the Serre compatible systems in \eqref{twosys} defined over a CM field $E$ through coefficient descent ($\mathsection$\ref{op}(G)).}
\begin{align}\label{twosys}
\begin{split}
\{\rho_{\pi,\lambda}: \Gal_F\to \GL_n(E_\lambda)\}_\lambda\hspace{.1in}\mathrm{and}\hspace{.1in}
\{\rho_{\chi,\lambda}: \Gal_{F^+}\to E_\lambda^*\}_\lambda
\end{split}
\end{align}
such that the conditions (i)--(iv) below are satisfied (\cite[$\mathsection2$]{BLGGT14}).

\begin{enumerate}[(i)]
\item $\rho_{\pi,\lambda}^c\cong \rho_{\pi,\lambda}^\vee\otimes \epsilon_\ell^{1-n}\rho_{\chi,\lambda}$ as $\Gal_F$-representations, 
where $\rho_{\pi,\lambda}^c$ denotes $\rho_{\pi,\lambda}$ composed with the conjugation of a complex conjugation $c\in\Gal_{F^+}$ and $\epsilon_\ell$ is the 
$\ell$-adic cyclotomic character.
\item Let $\iota:E_\lambda\hookrightarrow \C$. 
If $v\nmid\ell$, the semisimplified Weil-Deligne representation  is independent of $\lambda$ and satisfies:
$$\iota \text{WD}(\rho_{\pi,\lambda}|_{\Gal_{F_v}})^{F-ss}\cong(\rho_v:W_{F_v}\to \GL_n(\C),N)
\cong \text{rec}_{F_v}(\pi_v\otimes|\text{det}|_v^{(1-n)/2}),$$
where $\rho_v$ is a continuous representation of the Weil group  $W_{F_v}$ (of $F_v$), $N\in \End(\C^n)$ is nilpotent,
and $\text{rec}_{F_v}$ denotes the local Langlands correspondence.
\item When $F$ is totally real, $\rho_{\pi,\lambda}$  factors through $\GSp_n$ for all $\lambda$ or $\GO_n$ for all $\lambda$.
\item The representation $\rho_{\pi,\lambda}$ has pure weight $w$ (independent of $\lambda$) and distinct $\tau$-Hodge-Tate weights
for all $\tau:F\to\overline E$.
\end{enumerate}

\begin{thm}\label{autoae} 
The  semisimple SCS $\{\rho_{\pi,\lambda}\}_\lambda$ in \eqref{twosys}
satisfies the conditions (a) and (b) of $\mathsection$\ref{ESCS} for some integers $N_1,N_2\geq 0$ and finite  extension $F'/F$.
\end{thm}

\begin{proof}
The SCS $\{\rho_{\pi,\lambda}\}_\lambda$ is a strictly compatible system 
(i.e., $\mathsection$\ref{wcs}(i)--(iv) are satisfied).
The (finite) multiset $HT\subset\Z$ of union of $\tau$-Hodge Tate weights of $\rho_\lambda$ 
for all embeddings $\tau:F\to\overline E$ is independent of $\lambda$ 
($\mathsection$\ref{wcs}(iii)). Write $HT=\{b_1,b_2,...,b_m\}$ and 
suppose $b_1\leq b_2\leq\cdots\leq b_m$.
Denote by $v$ a place of $F$ above $\ell$ (the residue characteristic of $\lambda$).
Suppose $\ell$ is sufficiently large so that $F_v/\Q_\ell$ is unramified.
 Since $\rho_{\pi,\lambda}$
is crystalline at $v$ for almost all $\lambda$ 
($\mathsection$\ref{wcs}(ii)),
the $\ell$-adic representation $\rho_{\pi,\ell}:=\bigoplus_{\lambda|\ell}\rho_{\pi,\lambda}$ given by restriction of scalars (see \eqref{res11}) is crystalline at $v$ for $\ell\gg0$.
Pick integers $N_1,N_2\geq 0$ such that $0\leq b_1+ N_1\leq b_m+N_1\leq N_2$.
For $\ell\gg N_2+1$, the Fontaine-Laffaille theory  implies that 
the tame inertia weights of 
\begin{equation}
(\bar\rho_{\pi,\ell}^{\ss}\otimes \bar\epsilon_\ell^{N_1})|_{\Gal_{F_v}}
\end{equation}
and the Hodge-Tate weights (belonging to $[0,\ell-1)$) of the crystalline representation
\begin{equation}\label{localrepn}
(\rho_{\pi,\ell}\otimes \epsilon_\ell^{N_1})|_{\Gal_{F_v}}
\end{equation}
coincide (\cite[Theorem 5.3]{FL82}, see also \cite[Theorem 1.0.1]{Ba20}).
Since each Hodge-Tate weight of \eqref{localrepn} belongs to
$$HT+N_1:=\{b_1+N_1,b_2+N_1,...,b_m+N_1\}\subset [0,N_2],$$
Theorem \ref{ae}(a) (resp. $\mathsection$\ref{ESCS}(a)) holds for the SCS $\{\rho_{\pi,\ell}\}_\ell$ (resp. $\{\rho_{\pi,\lambda}\}_\lambda$) with 
integers $N_1,N_2\geq 0$.

Suppose the SCS $\{\rho_{\pi,\lambda}\}_\lambda$ is unramified outside $S$.
Consider large enough $\ell$ that is not divisible by any $v\in S$. 
For each $v\in S$, there is an open subgroup $U_v$ of the inertia group $I_{F_v}$ of $\Gal_{F_v}$ 
such that $\rho_{\pi,\lambda}(U_v)=\exp(t_\ell(U_v)N)$ is unipotent for all $\lambda$ whose characteristic $\ell$ 
is not divisible by $v$ by $\mathsection$\ref{ic}(ii), where
$$t_\ell: I_{F_v}\twoheadrightarrow I_{F_v}^t\twoheadrightarrow\Z_\ell$$
and the second surjection is the maximal pro-$\ell$ quotient of the tame inertia group. 
By shrinking $U_v$, we may assume $U_v$ is the inertia group of 
a finite Galois extension $F_v'$ of $F_v$.
Since $\overline F$ is dense in $\overline F_v$, we may 
pick $\alpha_v\in\overline F$ such that $F_v'\subset F_v(\alpha_v)$
and $F(\alpha_v)/F$ is Galois.
Thus, the representation $\rho_{\pi,\lambda}$ 
is unipotent on the inertia group $I_{F_v(\alpha_v)}$.
Define the finite Galois extension $F':=F(\alpha_v: ~v\in S)$ of $F$.
If $w$ is a finite place of $F'$ above $v$, then the completion $F'_w$
contains $F_v(\alpha_v)$ and $\rho_{\pi,\lambda}$
is unipotent on $I_{F'_w}$. Therefore,  
the semisimplification of  
$\bar\rho_{\pi,\lambda}^{\ss}|_{\Gal_{F'_w}}$ is unramified whenever $w\nmid\ell$
and $\mathsection$\ref{ESCS}(b) holds.
\end{proof}

By Theorem \ref{autoae}, the assertions in Theorem \ref{general}, Theorem \ref{thmA}, and Corollary \ref{corA}
hold for the automorphic SCS $\{\rho_{\pi,\lambda}\}_\lambda$.
Denote the ambient space and the algebraic monodromy group of $\rho_{\pi,\lambda}$ respectively
by $V_\lambda$ and $\bG_\lambda$.
For almost all $\lambda$ we write 
the algebraic envelope of $\rho_{\pi,\lambda}$ ($\mathsection\ref{ESCS}$) as 
\begin{equation}\label{aelambda}
\uG_\lambda:=\uG_{V_\lambda}\subset\GL_{\overline V_\lambda^{\ss}}.
\end{equation}
The remaining sections are devoted to the proof of Theorem \ref{thmB}.

\subsection{Potential automorphy for subrepresentations of SCS}
We present the essential results in \cite{BLGGT14}, \cite{CG13}, and \cite{PT15} that we need later on
and improve a result in \cite{CG13} on dihedral subrepresentation (Proposition \ref{esd2}).

\subsubsection{Polarized and odd $\ell$-adic representations}\label{selfdual}
We follow the terminology in \cite[$\mathsection2.1$]{BLGGT14}.
Let $F$ be a totally real or CM field and $F^+$ its maximal totally real subfield. 
For any infinite place $v$ of $F^+$, 
denote by $c_v\in\Gal_{F^+}$ the corresponding complex conjugation. 
A pair of $\ell$-adic
representations $(r,\mu)$ 
\begin{equation}\label{pair}
r:\Gal_F\to\GL_n(\overline\Q_\ell)\hspace{.1in}\text{and}\hspace{.1in}\mu:\Gal_{F^+}\to\overline\Q_\ell^*
\end{equation} 
is said to be \emph{polarized} if for some (and therefore for any) infinite place $v$ of $F^+$,
there exists $\epsilon_v\in\{\pm 1\}$ and a non-degenerate pairing $\left\langle~,~\right\rangle_v$
on $\overline \Q_\ell^n$ such that for all $x,y\in \overline \Q_\ell^n$ and $g\in\Gal_F$:
\begin{enumerate}[(i)]
\item $\left\langle x,y\right\rangle_v= \epsilon_v\left\langle y,x\right\rangle_v$;
\item $\left\langle r(g)x, r(c_v g c_v)y\right\rangle_v= \mu(g)\left\langle x,y\right\rangle_v$;
\item $\epsilon_v=-\mu(c_v)$ when $F$ is imaginary.
\end{enumerate}
When $F$ is totally real, $(r,\mu)$ is polarized if and only if
$r$ factors through either $\GSp_n(\overline\Q_\ell)$
with multiplier $\mu$ and $\mu(c_v)=-\epsilon_v$, or $\GO_n(\overline\Q_\ell)$ with 
multiplier $\mu$ and $\mu(c_v)=\epsilon_v$.
The pair $(r,\mu)$ is said to be \emph{totally odd}
or just \emph{odd} if $\epsilon_v=1$ for every infinite place $v$ of $F^+$.  
(Note that $\epsilon_v$ has nothing to do with $\epsilon_\ell$, the cyclotomic character.)

An $\ell$-adic representation $r$ of $F$ is said to be polarized (and odd)
if there exists an $\ell$-adic character $\mu$ of $F^+$ such that $(r,\mu)$ is polarized (and odd).
The automorphic $\ell$-adic representation $\rho_{\pi,\lambda}$ in \eqref{twosys} 
(or more precisely, the pair $(\rho_{\pi,\lambda},\epsilon_\ell^{1-n}\rho_{\chi,\lambda})$)  is polarized.
According to \cite[p. 208--209]{PT15}, there is a positive Dirichlet density set $\mathcal{L}$
of rational primes $\ell$ such that for every $\lambda$ dividing some prime in $\mathcal{L}$, the irreducible constituents
$W_\lambda$ of $\rho_{\pi,\lambda}\otimes\overline\Q_\ell$  are all polarized and odd. 
We state state some results when $F$ is totally real.

\begin{thm}\cite[Corollary 1.3]{BC11}\label{BC}
Let $F$ be a totally real field and $(\rho_{\pi,\lambda},\rho_{\chi,\lambda})$ the pair of (automorphic) $\ell$-adic 
Galois representations in \eqref{twosys}.
If $W_\lambda$ is an irreducible factor of $\rho_{\pi,\lambda}\otimes\overline\Q_\ell$ with 
$W_\lambda\cong W_\lambda^{\vee}\otimes \epsilon_\ell^{1-n}\rho_{\chi,\lambda}$,
then $(W_\lambda,\epsilon_\ell^{1-n}\rho_{\chi,\lambda})$ is 
polarized and odd.
\end{thm}

By condition (i) in $\mathsection\ref{ic}$, we have the following observation (\cite[Corollary 2.4]{CG13}).

\begin{cor}\label{BCcor}
Let $F$ be a totally real field and $W_\lambda$ an irreducible factor of $\rho_{\pi,\lambda}\otimes\overline\Q_\ell$. 
If $\dim W_\lambda>n/2$, then $(W_\lambda,\epsilon_\ell^{1-n}\rho_{\chi,\lambda})$ is 
polarized and odd.
\end{cor}

\begin{prop}\cite[Lemma 2.2]{CG13}\label{esd1}
If $F$ is totally real and $\mu:\Gal_F\to\overline\Q_\ell^*$ is a character, then $(\mu,\mu^2)$ is polarized and odd.
\end{prop}

\begin{prop}\cite[Lemma 2.1]{CG13}\label{esd3}
If $F$ is totally real and $(r,\mu)$ is polarized with $n$ odd, then $(r,\mu)$ is odd.
\end{prop}

We treat the $\dim W_\lambda\leq 2$ case next.
An irreducible two-dimensional representation $\phi:G\to \GL(V)$ of 
a group $G$ is called \emph{dihedral}
if it is induced from a one-dimensional representation of an index two subgroup of $G$.
In this case, the projective image of $G$ in $\PGL(V)$ is a dihedral group. 

\begin{prop}\label{esd2}
Let $F$ be a totally real field and $\{\rho_{\pi,\lambda}\}_\lambda$ the (automorphic) semisimple SCS in \eqref{twosys}.
For almost all $\lambda$, if $W_\lambda$ 
is a two-dimensional irreducible factor of $\rho_{\pi,\lambda}\otimes\overline\Q_\ell$, 
then $(W_\lambda, \det W_\lambda)$ is polarized and odd.
\end{prop}

\begin{proof}
By Theorem \ref{autoae} and Corollary \ref{corA} 
(on $W_\lambda\subset \rho_{\pi,\lambda}\otimes\overline\Q_\ell$), the semisimplified reduction 
$\overline W_\lambda:=\overline W_\lambda^{\ss}$ is irreducible for almost all $\lambda$.
For almost all $\lambda$, 
consider the algebraic monodromy group $\bG_{W_\lambda}$ and 
 algebraic envelope $\uG_{W_\lambda}$ of $W_\lambda$.
By \cite[Theorem 1]{Cl37}, the three cases we have are 
(a) $\SL_2\subset\bG_{W_\lambda}$,
(b) $\bG_{W_\lambda}^\circ$ is scalar, or
(c) $W_\lambda$ is dihedral.
Since $W_\lambda$ has distinct Hodge-Tate weights at any $v$ above $\ell$ (with respect to any $F_v\to\overline\Q_\ell$), case (b) is ruled out.
If case (a) holds,  for $\ell\gg0$ 
 the restriction $\overline W_\lambda|_{\Gal_{F(\zeta_\ell)}}$
is irreducible by Theorem \ref{general}(v)
and thus $\mathrm{Sym}^2 \overline W_\lambda|_{\Gal_{F(\zeta_\ell)}}$
is irreducible as well. Otherwise, the image of $\Gal_{F(\zeta_\ell)}$ in $\GL(\overline W_\lambda)$
is contained in some $\GO_2$ which is abelian.
Hence, the pair $(W_\lambda, \det W_\lambda)$ is polarized and odd 
for almost all $\lambda$ by \cite[Proposition 2.5]{CG13} (see also the proof of \cite[Theorem 3.2]{CG13}).
If case (c) holds, $(W_\lambda, \det W_\lambda)$ is polarized and odd 
as in the proof of \cite[Proposition 2.7]{CG13}.
\end{proof}

\begin{remark}
Proposition \ref{esd2} is an improvement of
\cite[Proposition 2.7]{CG13} from all $\lambda$ lying above $\ell\in\mathcal{L}$
of Dirichlet density one to almost all $\lambda$.
\end{remark}

Suppose $F$ is a CM field and $\{\rho_{\pi,\lambda}\}_\lambda$ is the (automorphic) semisimple SCS in \eqref{twosys}.
We obtain the following $(2n+1)$-dimensional SCS of $F^+$:
\begin{equation}\label{2n+1}
\{\Phi_\lambda:=(\mathrm{Ind}^{F^+}_F \rho_{\pi,\lambda})\oplus  \rho_{\chi,\lambda}\}_\lambda
\end{equation}
Denote by $F_{1,\pi}/F^+$ the Galois extension corresponding to the identity component of 
the algebraic monodromy group of $\Phi_{\lambda}$. Note that $F\subset F_{1,\pi}$.
By summarizing the results of \cite[$\mathsection\mathsection$6,8,9]{Xi19}, we have the following.

\begin{prop}(Xia)\label{X0}
Let $F$ be a CM field and $\{\rho_{\pi,\lambda}\}_\lambda$ the (automorphic) semisimple SCS in \eqref{twosys}. 
Suppose $F$ is the maximal CM subextension of $F_{1,\pi}/F^+$ and $n\leq 6$.
For almost all $\lambda$, every irreducible factor $W_\lambda$ of $\rho_{\pi,\lambda}\otimes\overline\Q_\ell$
is polarized and odd.
\end{prop}

\begin{proof}
By Proposition 18 and Lemma 5 of \cite{Xi19}, 
for almost all $\lambda$
and any irreducible factor $W_\lambda'$ of $\rho_{\pi,\lambda}|_{\Gal_{F_{1,\pi}}}$
 we have
\begin{equation}\label{cdual'}
W_\lambda'^{c_v}\cong W_\lambda'^\vee\otimes \delta_v
\end{equation}
for any infinite place $v$ of $F^+$, where $\delta_v$ is some character of $\Gal_{F_{1,\pi}}$.
This assertion together with \cite[Proposition 3]{Xi19}
imply that for almost all $\lambda$ and any irreducible factor $W_\lambda$ of 
$\rho_{\pi,\lambda}$, 
we have
\begin{equation}\label{cdual}
W_\lambda^{c_v}\cong W_\lambda^\vee\otimes \epsilon_\ell^{1-n}\rho_{\chi,\lambda}
\end{equation}
for any infinite place $v$ of $F^+$. We remark that \cite[Proposition 3]{Xi19}
assumes that  $F$ is the maximal CM subextension of $F_{1,\pi}/F^+$.

For any infinite place $v$ of $F^+$, let $\left\langle~,~\right\rangle_v$ be the non-degenerate pairing on $\overline\Q_\ell^n$
as $(\rho_{\pi,\lambda},\epsilon_\ell^{1-n}\rho_{\chi,\lambda})$ is polarized and odd.
To prove the proposition, it suffices to show that the restriction of each $\left\langle~,~\right\rangle_v$ 
to $W_\lambda$ is still non-degenerate such that conditions \ref{selfdual}(i),(ii),(iii) hold.
If $\left\langle~,~\right\rangle_v$ is not non-degenerate on $W_\lambda$,
then $\left\langle W_\lambda,W_\lambda\right\rangle_v\equiv 0$ by the irreducibility of $W_\lambda$ and \ref{selfdual}(ii).
The maximal subspace $M_\lambda$ such that $\left\langle W_\lambda,M_\lambda\right\rangle_v\equiv 0$
is a representation of $\Gal_F$ containing $W_\lambda$. By semisimplicity, we obtain a decomposition of representations
$$\rho_{\pi,\lambda}\otimes\overline\Q_\ell\cong M_\lambda\oplus U_\lambda$$
 such that the pairing 
$$\left\langle~,~\right\rangle_v: W_\lambda\times U_\lambda\to\overline\Q_\ell$$
is non-degenerate. Thus, we obtain 
$U_\lambda^{c_v}\cong W_\lambda^\vee\otimes \epsilon_\ell^{1-n}\rho_{\chi,\lambda}$
by \ref{selfdual}(ii) and then $U_\lambda\cong W_\lambda$ by \eqref{cdual}.
This contradicts the regularity of $\rho_{\pi,\lambda}$ 
since $W_\lambda$ and $U_\lambda$  are distinct in $\overline\Q_\ell^n$.
\end{proof}

\subsubsection{Irreducibility vs potential automorphy}\label{irredauto}
Let $\rho_{\pi,\lambda}$ be the $\ell$-adic Galois representation of a totally real (resp. CM) field $F$ attached to the 
regular algebraic, polarized, cuspidal automorphic 
representation of $\GL_n(\A_F)$ in \eqref{twosys} with irreducible decomposition
\begin{equation}\label{ird}
\rho_{\pi,\lambda}\otimes\overline\Q_\ell=W_1\oplus W_2\oplus\cdots\oplus W_k.
\end{equation}
The link between absolute irreducibility of $\rho_{\pi,\lambda}$ and potential automorphy
is as follows, see \cite[Proposition 3.1]{CG13} and \cite[Proposition 5.4.6, Theorem 5.5.2]{BLGGT14}.

\begin{prop}\label{link}
Let $F$ be a totally real (resp. CM) field and $\rho_{\pi,\lambda}$ be the $\ell$-adic representation
of $F$ in \eqref{twosys}.
Suppose there is a totally real (resp. CM) Galois extension $F'/F$ such that
$W_i|_{\Gal_{F'}}$  is irreducible and automorphic for all $1\leq i\leq k$ in \eqref{ird}. 
Then $k=1$, so $\rho_{\pi,\lambda}$
is absolutely irreducible.
\end{prop} 

Theorem \ref{potential} below gives some criteria for potential automorphy, whose proof essentially follows
the one of \cite[Theorem 3.2]{CG13} using the conditions \ref{potential}(a),(b),
the properties of a WCS in $\mathsection\ref{wcs}$, \cite[Theorem 4.5.1]{BLGGT14} and \cite[Lemma 1.5]{BLGHT11}.

\begin{thm}\label{potential}
Let $F$ be a totally real (resp. CM) field and $\rho_{\pi,\lambda}$ be the $\ell$-adic representation
of $F$ in \eqref{twosys}. Suppose $\ell$ is sufficiently large depending on the system $\{\rho_{\pi,\lambda}\}_\lambda$
and  $W_{i_1},...,W_{i_h}$ are some irreducible factors in \eqref{ird} such that for each $1\leq j\leq h$:
\begin{enumerate}[(a)]
\item $\overline W_{i_j}^{\ss}|_{\Gal_{F(\zeta_\ell)}}$ is irreducible;
\item $W_{i_j}$ is polarized and odd.
\end{enumerate}
Then there is a totally real (resp. CM) Galois extension $F'/F$ such that $W_{i_j}|_{\Gal_{F'}}$ 
is irreducible and automorphic for all $1\leq j\leq h$.
\end{thm}

Theorem \ref{infirred} below asserts the existence of infinitely many $\lambda$ 
such that $\rho_{\pi,\lambda}\otimes\overline\Q_\ell$ is irreducible.
We call the finite places $\lambda$ lying above some prime in $\mathcal{L}$ (in the theorem) the \emph{Patrikis-Taylor primes}.

\begin{thm}\cite[Theorem D, Theorem 1.7]{PT15}\label{infirred}
Let $F$ be a totally real (resp. CM) field and $\rho_{\pi,\lambda}$ be the $\ell$-adic representation
of $F$ in \eqref{twosys}. Then there exists a positive Dirichlet density set $\mathcal{L}$ of rational primes $\ell$
such that $\rho_{\pi,\lambda}\otimes\overline\Q_\ell$ is irreducible whenever $\lambda$ is above some $\ell\in\mathcal{L}$.
\end{thm}

\subsubsection{A reduction step of Xia}\label{Xia}
We first state a reduction result in \cite{Xi19}.

\begin{prop}\cite[Proposition 2]{Xi19}\label{X1}\footnote{The result in \cite{Xi19} is stated for all $\lambda$ lying above 
a Dirichlet density $1$ set of rational primes $\ell$ but the proof also works here. Moreover, the condition ``polarized'' 
should be added to \cite[Proposition 2]{Xi19} since the proof applies \cite[Theorem 2.1]{PT15} which requires the representations
to be polarized.}
Let $n_0\geq 1$ be an integer. Suppose for all CM fields $F$, all integers $n\leq n_0$,
and all regular algebraic, polarized, cuspidal automorphic representations $\pi$
of $\GL_n(\A_F)$ such that $F$ is the maximal CM subextension of $F_{1,\pi}/F^+$\footnote{The field $F_{1,\pi}$
is the minimal extension of $F^+$ for \eqref{2n+1} to be connected.}, 
we have $\rho_{\pi,\lambda}\otimes\overline\Q_\ell$ is irreducible for almost all $\lambda$.
Then for all CM fields $F$, all integers $n\leq n_0$, and for all 
 regular algebraic, polarized, cuspidal automorphic representation $\pi$
of $\GL_n(\A_F)$, $\rho_{\pi,\lambda}\otimes\overline\Q_\ell$ is irreducible for almost all $\lambda$.
\end{prop}

Proposition \ref{X1} is due to the following assertion (using the notation in \cite{Xi19}), Mackey's irreducibility criterion, and 
the regularity of $\rho_{\pi,\lambda}$. 

\begin{prop}\cite[Proof of Proposition 2]{Xi19}\label{X2}
Let $F$ be a CM field, $\{\rho_{\pi,\lambda}\}_\lambda$ the SCS of $F$ defined over $E$ in \eqref{twosys},
and $F_2$ the maximal CM subextension of $F_{1,\pi}/F^+$. After enlarging the CM field $E$ if necessary,
there exist a family of Galois representations $\{r_{1,\lambda}\}_\lambda$ of a subextension $F_4$ of $F_2/F$ 
and a regular algebraic polarized cuspidal automorphic representation $\pi_1$ of $\GL_m(\A_{F_3})$
where $F_3$ is a finite CM extension of $F_2$ such that 
$$\{\mathrm{Ind}^F_{F_4} r_{1,\lambda}\}_\lambda\cong\{\rho_{\pi,\lambda}\}_\lambda\hspace{.1in}\text{and}\hspace{.1in}\{\mathrm{Res}^{F_4}_{F_3}r_{1,\lambda}\}_\lambda\cong\{\rho_{\pi_1,\lambda}\}_\lambda,$$
and $F_3$ is the maximal CM subextension of $F_{1,\pi_1}/F^+$.
\end{prop}

In fact, a totally real version of Proposition \ref{X2} can be established 
by following the proof of \cite[Proposition 2]{Xi19}.

\begin{prop}\label{X2real}
Let $F^+$ be a totally real field, $\{\rho_{\pi,\lambda}\}_\lambda$ the SCS of $F^+$ defined over 
$E$ in \eqref{twosys}, $F$ a CM field containing $F^+$ as maximal totally real subfield,
and $F_2$ the maximal CM subextension of $F_{1,\pi}/F^+$\footnote{The field $F_{1,\pi}$ here denotes the minimal extension of $F^+$
such that $(\mathrm{Ind}^{F^+}_F\mathrm{Res}^{F^+}_F \rho_{\pi,\lambda})\oplus \rho_{\chi,\lambda}$ is connected, analogous to \eqref{2n+1}.}. After enlarging the CM field $E$ if necessary,
there exist a family of Galois representations $\{r_{1,\lambda}\}_\lambda$ of a subextension $F_4$ of $F_2/F^+$ 
and a regular algebraic polarized cuspidal automorphic representation $\pi_1$ of $\GL_m(\A_{F_3})$
where $F_3$ is a finite CM extension of $F_2$ such that 
$$\{\mathrm{Ind}^{F^+}_{F_4} r_{1,\lambda}\}_\lambda\cong\{\rho_{\pi,\lambda}\}_\lambda\hspace{.1in}\text{and}\hspace{.1in}\{\mathrm{Res}^{F_4}_{F_3}r_{1,\lambda}\}_\lambda\cong\{\rho_{\pi_1,\lambda}\}_\lambda,$$
and $F_3$ is the maximal CM subextension of $F_{1,\pi_1}/F^+$.
\end{prop}

A Galois representation $V$ of $F$ is said to be Lie-irreducible if 
$\mathrm{Res}^F_{F'} V$ is  irreducible for all finite extensions $F'/F$.
The advantage of taking the reduction step is as follows.

\begin{prop}\cite[Corollary 1]{Xi19}\label{X3}
Let $F$ be a CM field and $\{\rho_{\pi,\lambda}\}_\lambda$ the SCS of $F$ in \eqref{twosys}.
If $F$ is the maximal CM subextension of $F_{1,\pi}/F^+$, then there exist infinitely many
Patrikis-Taylor primes $\lambda$ such that 
$\rho_{\pi,\lambda}\otimes\overline\Q_\ell$
is Lie-irreducible.
\end{prop}

\subsection{Preparations for Theorem \ref{thmB}(i)}\label{table}
Let $\{\rho_{\pi,\lambda}:\Gal_F\to\GL(V_\lambda)\}_\lambda$ 
be the $n$-dimensional semisimple SCS in \eqref{twosys} of a totally real or CM field $F$
with $n\leq 6$ and algebraic monodromy groups $\bG_\lambda$.
By Proposition \ref{X2} (resp. \ref{X2real}), Mackey's irreducibility criterion, 
and regularity of $\{\rho_{\pi,\lambda}\}_\lambda$,
the absolutely irreducibility of $\rho_{\pi,\lambda}$ and $r_{1,\lambda}$
are equivalent, and moreover, follows from the 
absolute irreducibility of $\rho_{\pi_1,\lambda}$.
Hence, we may replace $\pi$ with $\pi_1$ and assume that $F=F_3$ is a CM field.
By the line of Proposition  \ref{X2} (resp. \ref{X2real}) on $F_3$,
and Propositions \ref{X0} and \ref{X3}, we further assume that
\begin{enumerate}[(i)]
\item  for almost all $\lambda$, every irreducible factor $W_\lambda$ of $\rho_{\pi,\lambda}\otimes\overline\Q_\ell$
is polarized and odd; 
\item the identity component $\bG_{\lambda_0}^\circ$ and also the connected semisimple group $\bG_{\lambda_0}^{\der}$  
are irreducible on $V_{\lambda_0}\otimes\overline\Q_{\ell_0}$ for some Patrikis-Taylor prime $\lambda_0$.
\end{enumerate}
When $2\leq n\leq 6$,
the isomorphism class of $(\bG_{\lambda_0}^{\der}, V_{\lambda_0}\otimes\overline\Q_{\ell_0})$ 
belongs to Table A below.
We use the notation $S^i$ for the symmetric $i$th power and
$\mathrm{std}$  for the  standard representation of an orthogonal or symplectic group including $\SL_2$.
Given a representation $\rho:\bG\to \GL(V)$, we denote the faithful representation 
$(\rho(\bG),V)$ by $(\bG,V)$ for simplicity. Thus, $(\SL_n,\mathrm{std})$ and 
$(\SL_n,\mathrm{std}^\vee)$ are identified as the same representation.

\begin{center}
\begin{tabular}{|l|l|c|c|} \hline
Cases  & $(\bG,V)$ & Self-duality & Weights in formal character\\ \hline
$(2A_1)$ & $(\SL_2,\mathrm{std})$& Yes 
& \begin{tikzpicture}
\draw[fill=black] (0.5,0) circle (2pt);
\draw[fill=black] (1,0) circle (2pt);
\end{tikzpicture} \\ \hline
$(3A_1)$ & $(\SL_2,S^2(\mathrm{std}))$& Yes& \begin{tikzpicture}
\draw[fill=black] (0.5,0) circle (2pt);
\draw[fill=black] (1,0) circle (2pt);
\draw[fill=black] (1.5,0) circle (2pt);
\end{tikzpicture} \\ \hline
$(3A_2)$  &  $(\SL_3,\mathrm{std})$ & No 
& \begin{tikzpicture}
\draw[fill=black] (-0.3,0) circle (2pt);
\draw[fill=black] (0.15,-0.26) circle (2pt);
\draw[fill=black] (0.15,0.26) circle (2pt);
\end{tikzpicture}  \\ \hline
$(4A_1)$ & $(\SL_2,S^3(\mathrm{std}))$&  Yes 
& \begin{tikzpicture}
\draw[fill=black] (0,0) circle (2pt);
\draw[fill=black] (0.5,0) circle (2pt);
\draw[fill=black] (1,0) circle (2pt);
\draw[fill=black] (1.5,0) circle (2pt);
\end{tikzpicture} \\ \hline
$(2A_1\otimes 2A_1)$ & $(\SO_4,\mathrm{std})$&  Yes
& \begin{tikzpicture}
\draw[fill=black] (0,0) circle (2pt);
\draw[fill=black] (0.25,0.25) circle (2pt);
\draw[fill=black] (0.5,0) circle (2pt);
\draw[fill=black] (0.25,-0.25) circle (2pt);
\end{tikzpicture} \\ \hline
$(4B_2)$ & $(\Sp_4,\mathrm{std})$& Yes  
& \begin{tikzpicture}
\draw[fill=black] (0,0) circle (2pt);
\draw[fill=black] (0.25,0.25) circle (2pt);
\draw[fill=black] (0.5,0) circle (2pt);
\draw[fill=black] (0.25,-0.25) circle (2pt);
\end{tikzpicture} \\ \hline
$(4A_3)$  &  $(\SL_4,\mathrm{std})$ & No & Vertices of a regular tetrahedron\\ \hline
$(5A_1)$ & $(\SL_2,S^4(\mathrm{std}))$&  Yes & \begin{tikzpicture}
\draw[fill=black] (0,0) circle (2pt);
\draw[fill=black] (0.5,0) circle (2pt);
\draw[fill=black] (1,0) circle (2pt);
\draw[fill=black] (1.5,0) circle (2pt);
\draw[fill=black] (2,0) circle (2pt);
\end{tikzpicture} \\ \hline
$(5B_2)$ & $(\SO_5,\mathrm{std})$&  Yes & \begin{tikzpicture}
\draw[fill=black] (0,0) circle (2pt);
\draw[fill=black] (0.25,0.25) circle (2pt);
\draw[fill=black] (0.25,0) circle (2pt);
\draw[fill=black] (0.5,0) circle (2pt);
\draw[fill=black] (0.25,-0.25) circle (2pt);
\end{tikzpicture}  \\ \hline
$(5A_4)$  &  $(\SL_5,\mathrm{std})$& No & Omitted \\ \hline
$(6A_1)$   & $(\SL_2,S^5(\mathrm{std}))$& Yes 
& \begin{tikzpicture}
\draw[fill=black] (0,0) circle (2pt);
\draw[fill=black] (0.5,0) circle (2pt);
\draw[fill=black] (1,0) circle (2pt);
\draw[fill=black] (1.5,0) circle (2pt);
\draw[fill=black] (2,0) circle (2pt);
\draw[fill=black] (2.5,0) circle (2pt);
\end{tikzpicture}  \\ \hline
$(2A_1\otimes 3A_1)$  & $(\SL_2\times\SL_2,\mathrm{std}\otimes S^2(\mathrm{std}))$& Yes
& \begin{tikzpicture}
\draw[fill=black] (0,0) circle (2pt);
\draw[fill=black] (0.3,0) circle (2pt);
\draw[fill=black] (0.6,0) circle (2pt);
\draw[fill=black] (0,-0.6) circle (2pt);
\draw[fill=black] (0.3,-0.6) circle (2pt);
\draw[fill=black] (0.6,-0.6) circle (2pt);
\end{tikzpicture}  \\ \hline
$(6A_3)$  &  $(\SO_6,\mathrm{std})$& Yes  & Vertices of a regular octahedron \\ \hline
$(6C_3)$  &  $(\Sp_6,\mathrm{std})$& Yes & Vertices of a regular octahedron \\ \hline
$(6A_2)$  &  $(\SL_3, S^2(\mathrm{std}))$&  No   
& \begin{tikzpicture}
\draw[fill=black] (0.3,0) circle (2pt);
\draw[fill=black] (-0.15,0.26) circle (2pt);
\draw[fill=black] (-0.15,-0.26) circle (2pt);
\draw[fill=black] (-0.6,0) circle (2pt);
\draw[fill=black] (0.3,0.52) circle (2pt);
\draw[fill=black] (0.3,-0.52) circle (2pt);
\end{tikzpicture}  \\ \hline
$(2A_1\otimes 3A_2)$  &  $(\SL_2\times\SL_3,\mathrm{std}\otimes\mathrm{std})$&  No & Vertices of a regular triangular prism \\ \hline
$(6A_5)$  &  $(\SL_6,\mathrm{std})$&  No & Omitted \\ \hline
\end{tabular}
\end{center}
\begin{center}
\footnotesize Table A. Isomorphism classes of connected semisimple subgroups $\bG\subset\GL_V$ that 
are irreducible on $V=\overline\Q_\ell^n$, where $2\leq n\leq 6$. 
\end{center}

\vspace{.1cm}

A brief explanation on the symbols representing the cases, 
take $(nA_m)$ for example, the integer $n$ is the degree of the irreducible representation  
and $A_m$ is the Lie type of the derived group, and $(2A_1\otimes 3A_1)$
is the tensor product representation of $(2A_1)$ and $(3A_1)$.  The other symbols
follow the same rules. 
Since the formal character of the semisimple group 
$\bG_\lambda^{\der}\to\GL_n$ is independent of $\lambda$ (Theorem \ref{Hui1}),
it is thus equal to the formal character of some case on the table.
In the course of establishing Theorem \ref{thmB}(i) in the coming subsections, we will obtain
the following results when $\rho_{\pi,\lambda_1}\otimes\overline\Q_{\ell_1}$ is Lie-irreducible for some $\lambda_1$.

\begin{prop}\label{lambdaindep}
Let $F$ be totally real, $n\leq 6$, and $\{\rho_{\pi,\lambda}\}_\lambda$ 
the semisimple SCS of $F$ in \eqref{twosys} with algebraic monodromy groups $\bG_\lambda$.
For every $\lambda$, fix an embedding $E_\lambda\hookrightarrow\C$.
If $\rho_{\pi,\lambda_1}\otimes\overline\Q_{\ell_1}$ is Lie-irreducible for some $\lambda_1$, then the following assertions hold.
\begin{enumerate}[(i)]
\item For almost all $\lambda$, $\rho_{\pi,\lambda}\otimes\overline\Q_{\ell}$ is Lie-irreducible.
\item For almost all $\lambda$, 
the faithful representations are isomorphic:
\begin{equation*}\label{isomC}
(\bG_{\lambda}^{\circ}\to\GL_{n,E_\lambda})_\C\cong (\bG_{\lambda_1}^{\circ}\to\GL_{n,E_{\lambda_1}})_\C.
\end{equation*}
\item  If $\bG_{\lambda_1}^{\der}$ is  of type A, then the isomorphism in (ii) holds for all $\lambda$.
\end{enumerate}
\end{prop}

\begin{prop}\label{lambdaindepCM}
Let $F$ be CM, $n\leq 6$, and $\{\rho_{\pi,\lambda}\}_\lambda$ 
the semisimple SCS of $F$ in \eqref{twosys} with algebraic monodromy groups $\bG_\lambda$.
For every $\lambda$, fix an embedding $E_\lambda\hookrightarrow\C$.
If $\rho_{\pi,\lambda_1}\otimes\overline\Q_{\ell_1}$ is Lie-irreducible for some $\lambda_1$, then the following assertions hold.
\begin{enumerate}[(i)]
\item For almost all $\lambda$, $\rho_{\pi,\lambda}\otimes\overline\Q_{\ell}$ is Lie-irreducible.
\item If $\bG_{\lambda_1}^{\der}\neq \SO_4,\Sp_4,\SO_6,\Sp_6$ on Table A, then for almost all $\lambda$
the faithful representations are isomorphic:
\begin{equation*}
(\bG_{\lambda}^{\circ}\to\GL_{n,E_\lambda})_\C\cong (\bG_{\lambda_1}^{\circ}\to\GL_{n,E_{\lambda_1}})_\C.
\end{equation*}
\item If $\bG_{\lambda_1}^{\der}$ is 
 one of the cases in $\mathsection$\ref{431}---$\mathsection$\ref{434},
then the isomorphism in (ii) holds for all $\lambda$.
\end{enumerate}
\end{prop}

\subsection{Proofs of Theorem \ref{thmB}(i) and Propositions \ref{lambdaindep}, \ref{lambdaindepCM}}
By $\mathsection$\ref{table}, it suffices to prove Theorem \ref{thmB}(i) under the situation
that $F$ is CM and the conditions $\mathsection$\ref{table}(i),(ii) hold.
Since the case $n=1$ is obviously true, we perform a case-by-case analysis 
 based on Table A in $\mathsection$\ref{431}---$\mathsection$\ref{438}.
Propositions \ref{lambdaindep}, \ref{lambdaindepCM} will be proved in $\mathsection$\ref{439}.

\subsubsection{Cases $(2A_1)$, $(3A_2)$, $(4A_3)$, $(5A_4)$, and $(6A_5)$}\label{431}
Theorem \ref{Hui1} implies that the rank of the semisimple subgroup $\bG_\lambda^{\der}\subset\GL_n$
is $n-1$ for all $\lambda$. It follows that $\bG_\lambda^{\der}=\SL_n$ and $V_\lambda\otimes\overline\Q_\ell$ 
is irreducible for all $\lambda$. \qed

\subsubsection{Cases $(4A_1)$ and $(6A_1)$}\label{432}
Since the formal character of $\bG_\lambda^{\der}$ is independent of $\lambda$ (Theorem \ref{Hui1})
and it does not contain zero weight (Table A),
the representation $(\bG_\lambda^{\der},V_\lambda\otimes\overline\Q_\ell)$
can only be $(\SL_2, S^{n-1}(\mathrm{std}))$.
It follows that $V_\lambda\otimes\overline\Q_\ell$ 
is irreducible for all $\lambda$ \qed 

\subsubsection{Cases $(3A_1)$ and $(5A_1)$}\label{433}
If $V_\lambda\otimes\overline\Q_\ell$ is reducible, 
we obtain an irreducible decomposition 
\begin{equation}\label{35}
V_\lambda\otimes\overline\Q_\ell=W_\lambda\oplus W'_\lambda
\end{equation}
such that $(\dim W_\lambda,\dim W'_\lambda)=(2,3)$ for case $(5A_1)$
(resp. $(2,1)$ for case $(3A_1)$)
by the fact that the formal character of $\bG_\lambda^{\der}$ (independent of $\lambda$)
has only one zero weight (Table A). The possibility that $(\dim W_\lambda,\dim W'_\lambda)=(4,1)$ 
is ruled out because it and $(5A_1)$ have different formal characters.
In the $(2,3)$-case,  
$(\bG_\lambda^{\der},V_\lambda\otimes\overline\Q_\ell)=(\SL_2,\mathrm{std}\oplus S^2(\mathrm{std}))$  and we will show that this is impossible.

We treat the case $(5A_1)$ only since $(3A_1)$ is similar (and easier).
Since $\bG_{\lambda_0}^{\der}=\PGL_2$ belongs to $\SO_5$ (e.g., via the principal homomorphism) up to conjugation in $\GL_5$,
we obtain an isomorphism $V_{\lambda_0}^\vee\cong V_{\lambda_0}\otimes\alpha_{\lambda_0}$
for some character $\alpha_0$. This isomorphism can be extended to an isomorphism of SCS (enlarging $E$ if necessary) 
\begin{equation}\label{SCSisom0}
\{V_{\lambda}^\vee\}_\lambda\cong \{V_{\lambda}\otimes\alpha_{\lambda}\}_\lambda
\end{equation}
since $\alpha_{\lambda_0}$ is of Hodge-Tate type.
Then we find a one-dimensional SCS $\{\beta_\lambda\}_\lambda$ such that
\begin{equation}\label{SCSisom20}
\{V_{\lambda}^c\}_\lambda\cong \{V_{\lambda}\otimes\beta_{\lambda}\}_\lambda
\end{equation}
by \eqref{SCSisom0} and $\mathsection$\ref{ic}(i). 
The algebraic monodromy $\bG_{W'_\lambda}$ in $\GL_{W'_\lambda}$ is contained in $\GO_3\cong \SO_3\times\GL_1$.
Let $\mu_\lambda$ be the Galois character corresponding to the $\GL_1$-component in $\GO_3$. 
Define $\widetilde{V}_\lambda:=V_\lambda\otimes\mu_\lambda^{-1}$
and $\widetilde{W}'_\lambda:=W_\lambda'\otimes\mu_\lambda^{-1}$. 
It follows by \eqref{SCSisom20} that
\begin{equation}\label{SCSisom30}
\widetilde{V}_{\lambda}^c\cong \widetilde{V}_{\lambda}\otimes \kappa_\lambda
\end{equation}
for some character $\kappa_\lambda$ and by dimension consideration that
\begin{equation}\label{SCSisom40}
\widetilde{W}_{\lambda}'^c\cong \widetilde{W}'_{\lambda}\otimes \kappa_\lambda.
\end{equation}
Since $\bG_{\widetilde{W}'_{\lambda}}=\SO_3\subset\SL_3$, the complex conjugation on Frobenii and Chebotarev's density theorem
imply that
\begin{enumerate}[(i)]
\item $\bG_{\widetilde{W}_{\lambda}'^c}\subset\SL_3$ and 
\item there is a density one subset $\mathscr{S}_F\subset\Sigma_F$ such that  
$\widetilde{W}_{\lambda}'^c|_{Frob_{\bar v}}$ has $1$ as an eigenvalue whenever $\bar v$ is above 
some element in $\mathscr{S}_F$.
\end{enumerate}
By (i) and \eqref{SCSisom40}, we obtain $\kappa_\lambda^3$ is trivial. If $\kappa_\lambda$ is non-trivial, then (ii) above
implies the existence of a density one-third subset $\mathscr{S}_F'\subset \mathscr{S}_F$
such that the characteristic polynomial of $Frob_{\bar v}$ on $\widetilde{W}'_{\lambda}$ is $1-T^3$
whenever $\bar v$ is above 
some element in $\mathscr{S}_F'$. This is absurd by Chebotarev's density theorem 
because the trace-zero set in the connected $\SO_3$ is 
a proper closed subset. Thus, $\kappa_\lambda$ is trivial.

Since $\mu_\lambda$ is of Hodge-Tate type, it comes from a character $\mu$ of $\GL_1(\A_F)$.
By \eqref{SCSisom30}, $\kappa_\lambda=1$, and multiplicity one \cite{JS81}, we get $(\pi\otimes\mu^{-1})^c\cong \pi\otimes\mu^{-1}$.
The cuspidal $\pi\otimes\mu^{-1}$ is $\text{BC}_{F/F^+}(\Pi)$,
the base change of a cuspidal automorphic representation $\Pi$ of $\GL_5(\A_{F^+})$ by \cite[Theorem 4.2(d)]{AC89}.
Since $\pi\otimes\mu^{-1}$ is regular algebraic, so is $\Pi$ (see \cite[Definition 1.6]{Cl16}). 
Then \cite{HLTT16} attaches an SCS $\{U_\lambda\}_\lambda$ of totally real $F^+$ (enlarging $E$ if necessary) to $\Pi$
such that 
\begin{equation}\label{SCSisom50}
\{\mathrm{Res}^{F^+}_F U_{\lambda}\}_\lambda\cong \{\widetilde{V}_{\lambda}:=V_\lambda\otimes\mu_\lambda^{-1}\}_\lambda
\end{equation}
by Chebotarev's density theorem,
in particular we get $\bG_{U_{\lambda_0}}^{\der}\cong\bG_{\widetilde{V}_{\lambda_0}}^{\der}=\PGL_2\subset \SO_5\subset\GL_5$.
Since the normalizer in $\GL_5$ of $\bG_{U_{\lambda_0}}^{\der}$ is contained in $\GO_5$, 
the image of the complex conjugation $c$ in $\GL(U_{\lambda_0})$
belongs to $\GO_5$ and thus the algebraic monodromy $\bG_{U_{\lambda_0}}$ satisfies
$$\bG_{U_{\lambda_0}}\subset\GO_5,$$ 
which implies $U_{\lambda_0}\cong U_{\lambda_0}^\vee\otimes \psi_{\lambda_0}$
for some Hodge-Tate type character $\psi_{\lambda_0}$ (coming from a character of $\GL_1(\A_{F^+}))$.
It follows by multiplicity one \cite{JS81} and $5$ is odd that the regular algebraic cuspidal $\Pi$ is polarized and odd,
and then by $\mathsection$\ref{ic}(iii) that $\bG_{U_{\lambda}}\subset\GO_5$ for all $\lambda$.
By the fact that $(\SL_2,\mathrm{std}\oplus S^2(\mathrm{std}))$ 
is not contained in $\GO_5$ and \eqref{SCSisom50},
we conclude that $(\bG_\lambda^{\der},V_\lambda\otimes\overline\Q_\ell)=(\SL_2,S^4(\mathrm{std}))$ (\eqref{35} is impossible) and is irreducible for all $\lambda$.\qed

\subsubsection{Cases $(6A_2)$ and $(2A_1\otimes 3A_2)$}\label{434}
The formal character of 
$(\bG_\lambda^{\der},V_\lambda\otimes\overline\Q_\ell)$ (independent of $\lambda$) has no zero weight and satisfies the fact that if $w$ is a weight
then $-w$ is not a weight (Table A). If $V_\lambda\otimes\overline\Q_\ell$ is reducible, we obtain an irreducible decomposition 
\begin{equation}\label{62*3}
V_\lambda\otimes\overline\Q_\ell=W_\lambda\oplus W'_\lambda
\end{equation}
such that $(\dim W_\lambda,\dim W'_\lambda)=(3,3)$. Then the images of $\bG_\lambda^{\der}$ in $\GL_{W_\lambda}$ and $\GL_{W'_\lambda}$
have to be $\SL_3$, which is impossible by the weights of $(6A_2)$ and $(2A_1\otimes 3A_2)$ on Table A.
It follows that $V_\lambda\otimes\overline\Q_\ell$ is irreducible for all $\lambda$.
Again from their weights on Table A, we see that
the irreducible $V_\lambda\otimes\overline\Q_\ell$ is not induced from a proper subgroup.
Hence, we obtain $(\bG_\lambda^{\der},V_\lambda\otimes\overline\Q_\ell)=(\SL_3, S^2(\mathrm{std}))$
 (resp. $=(\SL_2\times\SL_3,\mathrm{std}\otimes\mathrm{std})$) for case $(6A_2)$ (resp. $(2A_1\otimes 3A_2)$)
for all $\lambda$.\qed

\subsubsection{Case $(2A_1\otimes 3A_1)$}\label{435}
If $V_\lambda\otimes\overline\Q_\ell$ is reducible, we obtain a decomposition 
\begin{equation}\label{2*3}
V_\lambda\otimes\overline\Q_\ell=W_\lambda\oplus W'_\lambda
\end{equation}
such that $(\dim W_\lambda,\dim W'_\lambda)=(3,3)$ or $(4,2)$
because the formal character of $\bG_\lambda^{\der}$ (independent of $\lambda$)
has no zero weight (Table A). In the $(3,3)$-case, the only possibility is 
$(\bG_\lambda^{\der},V_\lambda\otimes\overline\Q_\ell)=(\SL_3, \mathrm{std}\oplus\mathrm{std}^\vee)$
since the formal characters of $(\SL_3, \mathrm{std}\oplus\mathrm{std})$ and $(2A_1\otimes 3A_1)$ are different.
In the formal character (Table A), there are three weights $w_1,w_2,w_3$ (top row) such that $w_1+w_2=2w_3$
but this is not true for the formal character of $(\SL_3, \mathrm{std}\oplus\mathrm{std}^\vee)$.
Thus, there remains the $(4,2)$-case. Since $\bG_\lambda^{\der}$ is of rank two, the only possibility\footnote{Note that $\Sp_4$ does not have any irreducible representation of dimension two.} is 
$(\bG_\lambda^{\der},V_\lambda\otimes\overline\Q_\ell)=(\SL_2\times\SL_2, (\mathrm{std}\otimes\mathrm{std})\oplus\mathrm{std})$.

Since $\SL_2\times\SL_2$ is of type A, 
 $\Gal_{F(\zeta_\ell)}$ is residually irreducible (Theorem \ref{general}(v)) on 
the polarized and odd ($\mathsection$\ref{table}(i)) representations $W_\lambda$ and $W'_\lambda$ if $\ell\gg0$.
Hence, Theorem \ref{potential} and Proposition \ref{link}
imply that the decomposition \eqref{2*3} is impossible if $\ell\gg0$. 
Again from the weights of $(2A_1\otimes 3A_1)$ on Table A, we see that
 the irreducible $V_\lambda\otimes\overline\Q_\ell$ is not induced from a proper subgroup.
For almost all $\lambda$, we conclude that the self-dual $(\bG_\lambda^{\der},V_\lambda\otimes\overline\Q_\ell)$ is irreducible
and isomorphic to $(\SL_2\times\SL_2, \mathrm{std}\otimes S^2(\mathrm{std}))$.\qed

\subsubsection{Case $(5B_2)$}\label{436}
If $V_\lambda\otimes\overline\Q_\ell$ is reducible, then we obtain an irreducible decomposition  
\begin{equation}\label{5}
V_\lambda\otimes\overline\Q_\ell=W_\lambda\oplus W'_\lambda\oplus W''_\lambda
\end{equation}
such that $(\dim W_\lambda,\dim W'_\lambda,\dim W''_\lambda)=(4,1,0)$ or $(3,2,0)$ or $(2,2,1)$ by the weights in Table A.
For the case $(4,1,0)$, $\bG_{W_\lambda}^{\der}$ is either $\SO_4$ (of type A) or $\Sp_4$.
We first show below that $\bG_{W_\lambda}^{\der}$ cannot be $\Sp_4$ (the argument is similar to the one in $(5A_1)$).

Since $\bG_{\lambda_0}^{\der}=\SO_5$, we obtain an isomorphism $V_{\lambda_0}^\vee\cong V_{\lambda_0}\otimes\alpha_{\lambda_0}$
for some character $\alpha_0$. This isomorphism can be extended to an isomorphism of SCS (enlarging $E$ if necessary) 
since $\alpha_{\lambda_0}$ is of Hodge-Tate type:
\begin{equation}\label{SCSisom}
\{V_{\lambda}^\vee\}_\lambda\cong \{V_{\lambda}\otimes\alpha_{\lambda}\}_\lambda.
\end{equation}
By \eqref{SCSisom} and $\mathsection$\ref{ic}(i), we find an one-dimensional SCS $\{\beta_\lambda\}_\lambda$ such that
\begin{equation}\label{SCSisom2}
\{V_{\lambda}^c\}_\lambda\cong \{V_{\lambda}\otimes\beta_{\lambda}\}_\lambda.
\end{equation}
Let $\mu_\lambda$ be the character on the one-dimensional representation $W'_\lambda$ 
in the $(4,1,0)$-case
and define $\widetilde{V}_\lambda:=V_\lambda\otimes\mu_\lambda^{-1}$. It follows that $W'_\lambda\otimes \mu_\lambda^{-1}$ 
and $(W'_\lambda\otimes \mu_\lambda^{-1})^c$ are trivial and 
by \eqref{SCSisom2} that
\begin{equation}\label{SCSisom3}
\widetilde{V}_{\lambda}^c\cong \widetilde{V}_{\lambda}.
\end{equation}
Since $\mu_\lambda$ is of Hodge-Tate type, it comes from a character $\mu$ of $\GL_1(\A_F)$.
By \eqref{SCSisom3} and multiplicity one \cite{JS81}, we get $(\pi\otimes\mu^{-1})^c\cong \pi\otimes\mu^{-1}$.
The cuspidal $\pi\otimes\mu^{-1}$ is $\text{BC}_{F/F^+}(\Pi)$,
the base change of a cuspidal automorphic representation $\Pi$ of $\GL_5(\A_{F^+})$ by \cite[Theorem 4.2(d)]{AC89}.
Since $\pi\otimes\mu^{-1}$ is regular algebraic, so is $\Pi$. 
Then \cite{HLTT16} attaches an SCS $\{U_\lambda\}_\lambda$ of totally real $F^+$ (extending $E$ if necessary) to $\Pi$
so that 
\begin{equation}\label{SCSisom4}
\{\mathrm{Res}^{F^+}_F U_{\lambda}\}_\lambda\cong \{\widetilde{V}_{\lambda}:=V_\lambda\otimes\mu_\lambda^{-1}\}_\lambda
\end{equation}
by Chebotarev's density theorem,
in particular we get $\bG_{U_{\lambda_0}}^{\der}\cong\bG_{\widetilde{V}_{\lambda_0}}^{\der}=\SO_5$.
Since the normalizer in $\GL_5$ of $\SO_5$ is $\GO_5$, the image of the complex conjugation $c$ in $\GL(U_{\lambda_0})$
belongs to $\GO_5$ and thus the algebraic monodromy $\bG_{U_{\lambda_0}}$ satisfies
$$\SO_5\subset\bG_{U_{\lambda_0}}\subset\GO_5,$$ 
which implies $U_{\lambda_0}\cong U_{\lambda_0}^\vee\otimes \psi_{\lambda_0}$
for some Hodge-Tate type character $\psi_{\lambda_0}$ (coming from a character of $\GL_1(\A_{F^+}))$.
It follows by multiplicity one \cite{JS81} and $5$ is odd that the regular algebraic cuspidal $\Pi$ is polarized and odd,
and then by $\mathsection$\ref{ic}(iii) that $\bG_{U_{\lambda}}\subset\GO_5$ for all $\lambda$.
By the fact that $\Sp_4\times\{1\}$ is not contained in $\GO_5$ and \eqref{SCSisom4},
we conclude that $\bG_{W_\lambda}^{\der}=\SO_4$ in the $(4,1,0)$ case.

Therefore, the irreducible factors in \eqref{5} are Lie-irreducible and of type A for all three cases.
Since $\Gal_{F(\zeta_\ell)}$ is residually irreducible (Theorem \ref{general}(v)) on 
the polarized and odd ($\mathsection$\ref{table}(i)) representations $W_\lambda$, $W'_\lambda$, $W''_\lambda$ if $\ell\gg0$,
Theorem \ref{potential} and Proposition \ref{link}
imply that the decomposition \eqref{5} is impossible if $\ell\gg0$. 
Since the semisimple group $\bG_\lambda^{\der}$ is of rank two and $5$ is prime, 
the irreducible $V_\lambda\otimes\overline\Q_\ell$ cannot be induced from a proper subgroup.
We conclude from Table A that $(\bG_\lambda^{\der},V_\lambda\otimes\overline\Q_\ell)=(\SO_5,\mathrm{std})$ for almost all $\lambda$.

\subsubsection{Cases  $(2A_1\otimes 2A_1)$ and $(4B_2)$}\label{437}
 From the formal character of $\bG_\lambda^{\der}$ in Table A,
if $\bG_\lambda$ is not irreducible on $V_\lambda\otimes\overline\Q_\ell$ (for infinitely many $\lambda$)
then for large $\ell$ the representation admits a decomposition 
\begin{equation}\label{44}
V_\lambda\otimes\overline\Q_\ell=W_\lambda\oplus W'_\lambda
\end{equation}
such that $W_\lambda$ and $W'_\lambda$ are irreducible of dimension $2$ 
and the semisimple groups satisfy $\bG_{W_\lambda}^{\der}\cong\SL_2\cong \bG_{W'_\lambda}^{\der}$.
By Theorem \ref{general}(v), $\Gal_{F(\zeta_\ell)}$ is residually irreducible on  
the polarized and odd ($\mathsection$\ref{table}(i)) representations $W_\lambda$ and $W'_\lambda$ if $\ell\gg0$ . 
Hence, Theorem \ref{potential} and Proposition \ref{link}
imply that the decomposition \eqref{44} is impossible if $\ell\gg0$. 
We obtain that $V_\lambda\otimes\overline\Q_\ell$ is irreducible for almost all $\lambda$.

By assumption $\mathsection4.3$(ii) and twisting $\{V_\lambda\}_\lambda$
with the SCS of some power of cyclotomic characters, we assume $\bG_{\lambda_0}=\GO_4$ or $\GSp_4$
and thus $\bG_\lambda$ is connected for all $\lambda$. Therefore, $V_\lambda\otimes\overline\Q_\ell$ is not induced
and $V_\lambda\otimes\overline\Q_\ell$ is Lie-irreducible for almost all $\lambda$.
By considering the formal character, we 
conclude that $\bG_{\lambda}$ is either $\GO_4$ or $\GSp_4$ for almost all $\lambda$ 
(although we cannot tell if $\bG_{\lambda}$ and $\bG_{\lambda_0}$ are of the same Lie type).
\qed

\subsubsection{Cases $(6A_3)$ and $(6C_3)$}\label{438}
From the formal character of $\bG_\lambda^{\der}$ in Table A,
if $\bG_\lambda$ is not irreducible on $V_\lambda\otimes\overline\Q_\ell$ (for infinitely many $\lambda$)
then for large $\ell$ the representation admits a decomposition 
\begin{equation}\label{66}
V_\lambda\otimes\overline\Q_\ell=W_\lambda\oplus W'_\lambda
\end{equation}
such that $W_\lambda$ is irreducible of dimension $2$ 
and the semisimple group $\bG_{W_\lambda}^{\der}\cong\SL_2$.
By Theorem \ref{general}(v) and $\mathsection$\ref{table}(i), 
$\Gal_{F(\zeta_\ell)}$ is residually irreducible on the polarized and odd $W_\lambda$ if $\ell\gg0$.
Thus, Theorem \ref{potential} implies that for some $\lambda_1$ above $\ell\gg0$
the restriction $W_{\lambda_1}|_{\Gal_{F'}}$ is automorphic
for some totally real field $F'$, which implies that
$W_{\lambda_1}|_{\Gal_{F'}}$ belongs a two-dimensional semisimple WCS 
$$\{\psi_\lambda:\Gal_{F'}\to \GL_2(\overline E_\lambda)\}_\lambda$$ 
of $F'$ defined over $E$ (by enlarging $E$ if necessary) with 
algebraic monodromy groups all containing $\SL_2$ (Theorem \ref{Hui1}). 
Consider the WCS $\{\rho_{\pi,\lambda}\oplus\psi_\lambda\}_\lambda$ of $F'$
with algebraic monodromy groups $\bH_\lambda$. 
It follows that $\bH_{\lambda_1}^{\der}\cong\bG_{\lambda_1}^{\der}$ is of rank three
and (by Goursat's lemma) $\bH_{\lambda_0}^{\der}=\SO_6\times\SL_2$ or $\Sp_6\times\SL_2$ is of rank four, 
contradicting Theorem \ref{Hui1}.
We obtain that $V_\lambda\otimes\overline\Q_\ell$ is irreducible for almost all $\lambda$.

By twisting $\{V_\lambda\}_\lambda$
with the SCS of some power of cyclotomic characters, we assume $\bG_{\lambda_0}=\GO_6$ or $\GSp_6$
and thus $\bG_\lambda$ is connected for all $\lambda$. Therefore, $V_\lambda\otimes\overline\Q_\ell$ is not induced
and  $V_\lambda\otimes\overline\Q_\ell$ is Lie-irreducible for almost all $\lambda$.
By considering the formal character, we conclude that $\bG_{\lambda}$ is either $\GO_6$ or $\GSp_6$ for almost all $\lambda$
(although we cannot tell if $\bG_{\lambda}$ and $\bG_{\lambda_0}$ are of the same Lie type).
\qed

\subsubsection{Proofs of Propositions \ref{lambdaindep}, \ref{lambdaindepCM}}\label{439}
When $\rho_{\pi,\lambda_1}\otimes\overline\Q_\ell$ is Lie-irreducible for some $\lambda_1$,
the identity component $\bG_\lambda^\circ$ (for all $\lambda$) is unchanged under 
the reduction in $\mathsection$\ref{table} (see Propositions \ref{X2},\ref{X2real}). Hence,
Proposition \ref{lambdaindepCM}  follows immediately 
from the determination of $\bG_\lambda^{\der}$ in $\mathsection$4.1.1--$\mathsection$4.4.8.
It remains to treat the totally real case, in which we can say more about $\bG_\lambda^{\der}$.
As Proposition \ref{lambdaindepCM} holds and the totally real case can again be reduced to the CM 
case by $\mathsection$\ref{table},
it suffices to prove the following two assertions.
\begin{itemize}
\item Proposition \ref{lambdaindep}(ii) when $\bG_{\lambda_1}^{\der}\in\{\SO_4,\Sp_4,\SO_6,\Sp_6\}$.
\item Proposition \ref{lambdaindep}(iii) for cases $(2A_1\otimes 3A_1)$, $(2A_1\otimes 2A_1)$, and $(6A_3)$.
\end{itemize}

The first assertion is obvious by the $\lambda$-independence of the pairing ($\mathsection$\ref{ic}(iii)) when $F$ is totally real and the fact\footnote{This is seen by reducing to the CM case and applying $\mathsection$4.4.7--$\mathsection$4.4.8.} that
$\bG_{\lambda}^{\der}\in\{\SO_4,\Sp_4\}$ when $n=4$  
(resp. $\bG_{\lambda}^{\der}\in\{\SO_6,\Sp_6\}$ when $n=6$) 
 for almost all $\lambda$.

For case $(2A_1\otimes 2A_1)$ (resp. $(6A_3)$) in the second assertion,
we have an inclusion  $\bG_\lambda^{\der}\subset\SO_4$ (resp. $\bG_\lambda^{\der}\subset\SO_6$) 
of semisimple groups with the same rank for each $\lambda$.
Since $\SO_4$ (resp. $\SO_6$) is of type A, the inclusion is an equality for all $\lambda$
(\cite[Lemma 2]{HL16}). For case $(2A_1\otimes 3A_1)$, it suffices (by $\mathsection$4.4.5) to 
show that the $(4,2)$-decomposition in \eqref{2*3} is impossible, i.e., 
\begin{equation}\label{impossible}
(\bG_\lambda^{\der},V_\lambda\otimes\overline\Q_\ell)=
(\SL_2\times\SL_2,(\mathrm{std}\otimes\mathrm{std})\oplus \mathrm{std})
\end{equation}
is impossible. Since $\bG_\lambda^{\der}$ 
is a subgroup of $\SO_6$ or $\Sp_6$ by $\mathsection$\ref{ic}(iii),
the equation \eqref{impossible} cannot hold. We are done.
\qed

\subsection{Proof of Theorem \ref{thmB}(ii)}
Suppose $F=\Q$ in this section. By Proposition \ref{X2real} and $n\leq 6$, 
the system $\{\rho_{\pi,\lambda}\}_\lambda$ are of type A 
when induced, i.e., $F_4\neq \Q$. Then the result follows from Theorem \ref{thmB}(i) and Theorem \ref{thmA}.
Otherwise if $F_4=\Q$, Propositions \ref{X2real} and \ref{X3} imply that
 $\rho_{\pi,\lambda_0}\otimes\overline\Q_{\ell_0}$ is Lie-irreducible for some $\lambda_0$.
By Proposition \ref{lambdaindep}(iii) and Theorem \ref{thmA}, we just need to consider the 
non-type A (self-dual) cases on Table A, which are $(4B_2)$, $(5B_2)$, and $(6C_3)$.

\subsubsection{Cases $(4B_2)$ and $(6C_3)$}
Since the argument is essentially the same for both cases, we treat $(4B_2)$ only.
Proposition \ref{lambdaindep}(ii) implies that $\bG_\lambda^{\der}=\Sp_4$ for almost all $\lambda$.
Suppose $\Gal_\Q$ is not irreducible on $\overline V_\lambda^{\ss}\otimes\overline\F_\ell$ for some large $\ell$.
Since the formal characters of the semisimple groups $\bG_\lambda^{\der}$
and $\uG_{\lambda}^{\der}$ are equal (Theorem \ref{general}(iii)), we obtain an irreducible decomposition 
\begin{equation}\label{4c6d'}
\overline V_\lambda^{\ss}\otimes\overline\F_\ell=\overline W\oplus \overline W'
\end{equation}
of $\uG_\lambda$ with $\dim \overline W =\dim \overline W'= 2$ and $\uG_{\lambda}^{\der}=\SL_2\times\SL_2$
acting naturally on \eqref{4c6d'}\footnote{As the formal character of 
$\uG_{\lambda}^{\der}$ and $\Sp_6$ are equal in case $(6C_3)$, for almost all $\lambda$
the only possibilities of
 $(\uG_{\lambda}^{\der},\overline V_\lambda^{\ss}\otimes\overline\F_\ell)$ are
$(\Sp_6,\mathrm{std})$, $(\Sp_4\times\SL_2,\mathrm{std}\oplus \mathrm{std})$, and 
$(\SL_2\times\SL_2\times\SL_2, \mathrm{std}\oplus \mathrm{std}\oplus\mathrm{std})$.}.
Since the image of $[\Gal_L,\Gal_L]$ (in $\uG_\lambda^{\der}(\overline{\F}_\ell)$) and $\uG_\lambda^{\der}(\overline{\F}_\ell)$ have the same commutant in 
$\End(\overline V_\lambda^{\ss}\otimes\overline\F_\ell)$ for $\ell\gg0$ 
by Theorem \ref{general}(ii), 
it follows that $[\Gal_L,\Gal_L]$ is irreducible on $\overline W$ (resp. $\overline W'$).
Thus, we obtain 
 a continuous two-dimensional irreducible Galois representation
\begin{equation}\label{tworepn}
\bar\phi_\lambda:\Gal_\Q\to \GL(\overline W)\cong\GL_2(\overline\F_\ell).
\end{equation}
such that $\bar\phi_\lambda|_{[\Gal_\Q,\Gal_\Q]}$ is still irreducible.
We would like to show that $\bar\phi_\lambda$ is odd and apply Serre's modularity conjecture \cite{Se87} proved in \cite{KW09a,KW09b}.

The condition $\mathsection\ref{ic}$(i) asserts that 
$V_\lambda\cong V_\lambda^\vee\otimes \epsilon_\ell^{-3}\rho_{\chi,\lambda}$.
The character $\mu_\lambda:=\epsilon_\ell^{-3}\rho_{\chi,\lambda}$ is odd by Theorem \ref{BC} and Theorem \ref{thmB}(i).
By semisimplified reduction, we obtain
\begin{equation}\label{odd1}
\overline V_\lambda^{\ss}\cong (\overline V_\lambda^{\ss})^\vee\otimes \overline \mu_\lambda.
\end{equation}
If we can show that 
\begin{equation}\label{odd2}
\overline W\cong \overline W^\vee\otimes \overline \mu_\lambda,
\end{equation}
then $\det(\overline W)=\overline\mu_\lambda$ is odd since $\overline W$ is two-dimensional.
If on the contrary \eqref{odd2} is false, then $\overline W'\cong \overline W^\vee\otimes \overline \mu_\lambda$
which implies $\uG_\lambda^{\der}=\SL_2$. This is absurd since $\uG_\lambda^{\der}=\SL_2\times\SL_2$.
We conclude that the Galois representation \eqref{tworepn} of $\Q$ is of \emph{Serre-type} (i.e., odd and irreducible). 

If $\overline V_\lambda^{\ss}\otimes\overline\F_\ell$ is not irreducible for infinitely many $\lambda$,
then we can find a sequence $\{\bar\phi_{\lambda}\}_\ell$ of two-dimensional irreducible odd Galois representations 
 of $\Q$ \eqref{tworepn} indexed by an infinite set of distinct primes $\ell$ such that $\lambda$ divides $\ell$. 
After twisting $\bar\phi_{\lambda}$ with the $m$th power of the cyclotomic character, 
the Galois representation $\bar\phi_{\lambda}\otimes \bar\epsilon_\ell^m$ of $\Q$
is still of Serre-type and thus comes from 
a cuspidal Hecke eigenform $f_\lambda$ of level $N_\lambda$ and weight $k_\lambda$ by Serre's conjecture, where $N_\lambda$ and $k_\lambda$ always denote Serre's predicted minimal prime-to-$\ell$ level and weight.
We would like to show that for some integer $m\in\Z$, 
there exists a subsequence of $\{\bar\phi_{\lambda}\otimes \bar\epsilon_\ell^m\}_\ell$
coming from the same eigenform $f$. By the finiteness of Hecke eigenforms of a fixed level and weight,
it suffices to show that for some $m\in\Z$, the weights $k_\lambda$ and levels $N_\lambda$ 
are bounded for an infinite subset of $\{f_\lambda\}_\ell$.
Hence by twisting $\{\rho_\lambda\}_\lambda$ with $\{\epsilon_\ell^{m'}\}_\ell$ for some $m'\in\N$, 
we assume the Hodge-Tate weights of the WCS $\{\rho_\lambda\}_\lambda$
belong to $[0,C]$ for some $C>0$.

Suppose $\ell$ is odd. Let $I_\ell$ be the inertia subgroup of $\Gal_{\Q_\ell}$.
The semisimplification of $\bar\phi_{\lambda}|_{I_\ell}$ factors through 
the tame inertia $I_\ell^t$, giving us two tame inertia characters $\bar\gamma_\lambda$ and $\bar\gamma_\lambda'$
stable under the action of Frobenius $x\mapsto x^\ell$. Following \cite[$\mathsection2$]{Da95}, we can distinguish two cases:\\

 \textit{Case 1}: $\bar\gamma_\lambda^\ell=\bar\gamma_\lambda'$ and 
$(\bar\gamma_\lambda')^\ell=\bar\gamma_\lambda$.
Then we can write 
$$\bar\gamma_\lambda=\bar\theta_2^{a_\lambda+\ell b_\lambda}$$ 
where $\bar\theta_2$ is the fundamental character of level $2$ (see $\mathsection$\ref{tiw}) and
$0\leq a_\lambda,b_\lambda\leq \ell-1$.\\

 \textit{Case 2}: $\bar\gamma_\lambda^\ell=\bar\gamma_\lambda$ and 
$(\bar\gamma_\lambda')^\ell=\bar\gamma_\lambda'$.
Then we can write 
\begin{equation}\label{inertiarp}\bar\phi_{\lambda}|_{I_\ell}=\bpx
\bar\epsilon_\ell^{a_\lambda} & \ast  \\
0 & \bar\epsilon_\ell^{b_\lambda}
\epx
\end{equation}
where $\bar\epsilon_\ell$ is the mod $\ell$ cyclotomic character. We normalize $a_\lambda$ and $b_\lambda$
so that 
\begin{itemize}
\item $0\leq a_\lambda\leq \ell-2$ if \eqref{inertiarp} is semisimple,  $1\leq a_\lambda\leq \ell-1$ otherwise.
\item  $0\leq b_\lambda\leq \ell-2$.
\end{itemize}

Suppose $\ell\gg C$ (upper bound of the (positive) Hodge-Tate weights) such that $\rho_{\pi,\lambda}|_{\Gal_{\Q_\ell}}$ 
is \emph{Fontaine-Laffaille} (i.e., crystalline and the Hodge-Tate weights belong to $[0,\ell-1)$).
By $\mathsection$\ref{wcs}(iii), the regularity of $\rho_{\pi,\lambda}$, and the Fontaine-Laffaille theory \cite[Theorem 5.3]{FL82} (see also \cite[Theorem 1.0.1]{Ba20}),
the exponents $a_\lambda$ and $b_\lambda$ are distinct integers belonging
 to $[0,C]$. 
Thus, $\{a_\lambda,b_\lambda\}$ is equal to $\{a,b\}$ for infinitely many $\lambda$
and we assume  $\{a_\lambda,b_\lambda\}=\{a,b\}$ without loss of generality.
Note that if we are in Case 2 and \eqref{inertiarp} is not semisimple, 
we always have $b_\lambda< a_\lambda$ by applying \cite[Remark 8.3.7, Exercise 8.4.3]{BrC09} to 
a Fontaine-Laffaille lift of $\bar\phi|_{\Gal_{\Q_\ell}}$ 
whose existence is guaranteed by \cite[Proposition 2.3.1]{GHLS17}.
Take $m$ such that $0=\mathrm{min}\{a+m,b+m\}$.
According to Serre's recipe (see \cite[$\mathsection2$]{Da95}), 
the weight $k_\lambda$ of the eigenform $f_\lambda$ 
corresponding to $\overline{\phi_{\lambda}\otimes \epsilon_\ell^m}$ is given by
\begin{equation}\label{weight}
1+(a+m)+(b+m)+(\ell-1)\mathrm{min}\{a+m,b+m\}+(\ell-1)\delta,
\end{equation}
where $\delta=0$ or $1$. The case $\delta=1$ arises when $a+m=b+m=0$ or 
when $\bar\phi_{\lambda}\otimes \bar\epsilon_\ell^m|_{\Gal_{\Q_\ell}}$ is \emph{tr\`es ramifi\'ee}.
Since $\rho_{\pi,\lambda}\otimes\epsilon_\ell^m |_{\Gal_{\Q_\ell}}$ is also Fontaine-Laffaille,
any of its reduction 
is  peu ramifi\'ee by \cite[Proposition 2.3.1]{GHLS17}.
As $\bar\phi_{\lambda}\otimes \bar\epsilon_\ell^m|_{\Gal_{\Q_\ell}}$
is a subquotient of a reduction of $\rho_{\pi,\lambda}\otimes\epsilon_\ell^m |_{\Gal_{\Q_\ell}}$
 that is peu ramifi\'ee,
the two-dimensional representation 
$\bar\phi_{\lambda}\otimes \bar\epsilon_\ell^m|_{\Gal_{\Q_\ell}}$ is also 
 peu ramifi\'ee by \cite[Remark 2.1.6]{GHLS17},
which coincides with the usual definition of  peu ramifi\'ee (i.e., not tr\`es ramifi\'ee) 
by \cite[Example 2.1.4(1)]{GHLS17}.
Hence, $\delta=0$ in \eqref{weight} and we obtain $k_\lambda=1+a+b+2m$ for infinitely many $\lambda$.
Next, we would like to give an upper bound of the
levels  of eigenforms $f_\lambda$ for all $\lambda$. 
According to Serre's recipe (see \cite[$\mathsection2$]{Da95}), the level
$N_\lambda$ is the Artin conductor of $\bar\phi_{\lambda}\otimes \bar\epsilon_\ell^m$ with power of $\ell$ removed.
Suppose the WCS $\{\rho_{\pi,\lambda}\}_\lambda$ is unramified outside $S$.
For each $p\in S$, 
the image $\rho_{\pi,\lambda}\otimes\epsilon_\ell^m(I^w_p)$ of the wild inertia group at $p$
is isomorphic to the fixed finite group $\rho_p(I^w_p)$ for almost all $\lambda$,
where $\rho_p$ is the representation of the Weil group defined in $\mathsection$\ref{ic}(ii).
Thus, the image $\bar\phi_{\lambda}\otimes \bar\epsilon_\ell^m(I^w_p)$ is a quotient of $\rho_p(I^w_p)$
for almost all $\lambda$. 
By the formula for Artin conductor in \cite[Chap. VI Corollary 1']{Se79} and the finiteness of $S$, 
the level $N_\lambda$ is bounded uniformly independent of $\lambda$.
Since the weights $k_\lambda$ and levels $N_\lambda$ for infinitely many $\lambda$ are bounded,
we deduce that the eigenform $f_\lambda=f$ (independent of $\lambda$) for infinitely many $\lambda$.

Enlarge the number field $E$ if necessary. Consider the 
two-dimensional SCS
$$\{\psi_{f,\lambda}:\Gal_\Q\to \GL_2(E_\lambda)\}_\lambda$$
attached to the eigenform $f$ (in last paragraph). 
The algebraic monodromy group of $\psi_{f,\lambda}$ is $\GL_2$ (i.e., $\psi_{f,\lambda}$ is not induced) for all $\lambda$
since the formal bi-character of $\psi_{f,\lambda}$ is independent of $\lambda$ (Theorem \ref{Hui1})
and for all $\ell\gg0$ the restriction of $\bar\psi_{f,\lambda}=\bar\phi_\lambda\otimes\bar\epsilon_\ell^m$
to $[\Gal_\Q,\Gal_\Q]$ is still irreducible.
Consider the $6$-dimensional SCS compatible system (fulfilling the conditions of Theorem \ref{general} by Remark \ref{cond})
\begin{equation}\label{n+2}
\{U_\lambda:=(\rho_{\pi,\lambda}\otimes\epsilon_\ell^m)\oplus\psi_{f,\lambda}\}_\lambda
\end{equation}
and let $\bM_\lambda$ be the algebraic monodromy group at $\lambda$ and $\uM_\lambda$ be the 
algebraic envelope for almost all $\lambda$. Theorem \ref{general}(iii) implies that 
the semisimple groups $\bM_{\lambda}^{\der}$ and 
$\uM_{\lambda}^{\der}$
have the same formal bi-character on respectively $U_\lambda$ and
\begin{equation}\label{n+2'}
\overline U_\lambda^{\ss}:=(\bar{\rho}_{\pi,\lambda}^{\ss}\otimes\bar\epsilon_\ell^m)\oplus\bar{\psi}_{f,\lambda}^{\ss}
\end{equation}
for almost all $\lambda$.
This is impossible since for infinitely many $\lambda$
we have $\bM_{\lambda}^{\der}=\Sp_4\times\SL_2$ (by Goursat's lemma)
and $\uM_{\lambda}^{\der}\cong\uG_\lambda^{\der}=\SL_2\times\SL_2$ 
(by construction\footnote{The second factor of \eqref{n+2'} is a subrepresentation of the first factor.})
are of different ranks. \qed

\subsubsection{Case $(5B_2)$}\label{case5B}
Proposition \ref{lambdaindep}(ii) asserts that $\bG_\lambda^{\der}=\SO_5$ for almost all $\lambda$.
By twisting $\{\rho_{\pi,\lambda}\}_\lambda$ with the SCS of cyclotomic characters,
we assume $\bG_\lambda=\GO_5$ is connected for $\ell\gg0$.
Since Hypothesis (MFT) ($\mathsection\ref{hypo}$(iv)) is satisfied for $\{\rho_{\pi,\lambda}\}_\lambda$ 
(see $\mathsection\ref{hypo}(v)$, \cite[Chapter III.2]{HT01}, \cite[Theorem 3.6]{Ta04})
which is a union of $[E:\Q]$ subsystems of the rational SCS
$$\{\rho_{\pi,\ell}:=\bigoplus_{\lambda|\ell}\rho_{\pi,\lambda}\}_\ell$$
by restriction of scalars just like \eqref{res11},
we obtain $\bar\rho_{\pi,\lambda}^{\ss}(\Gal_\Q)\subset\uG_\lambda$ for $\ell\gg0$
by Proposition \ref{connect}. 
Since the Galois representation is polarized, it follows that 
$V_\lambda=V_\lambda^\vee\otimes\mu_\lambda$ for some character $\mu_\lambda$
and after some reduction we obtain
\begin{equation}\label{dualmod}
\overline V_\lambda=\overline V_\lambda^{^\vee}\otimes\bar\mu_\lambda.
\end{equation}

If the Galois representation $\overline V_\lambda^{\ss}$ is not irreducible (for infinitely many $\lambda$), then it
admits an irreducible/zero decomposition 
\begin{equation}\label{izd} 
\overline V_\lambda^{\ss}=\overline W\oplus\overline W'\oplus \overline W''
\end{equation}
so that $(\dim \overline W,\dim\overline W',\dim\overline W'')\in\{(4,1,0),(3,2,0),(2,2,1)\}$ 
by the fact that the formal characters of $\SO_5$ and $\uG_\lambda^{\der}$ coincide for almost all $\lambda$ (Theorem \ref{general}(iii)). 
For cases $(3,2,0)$ and $(2,2,1)$, the orthogonal duality \eqref{dualmod} forces 
$\uG_\lambda:=\uG_{V_\lambda}$ to be a subgroup
of $\GO_5$ and thus we can find a two-dimensional $\uG_\lambda$-subrepresentation $\overline U$
of $\overline V_\lambda^{\ss}$ by the duality.
Hence, it follows that the image of $\psi':\uG_\lambda\to\GL_{\overline U}$ is in $\GO_2$.
The map $\psi'$ factors through $\uG_\lambda\to\GL_{\overline U}$ which is isomorphic 
to the natural $\psi:\uG_\lambda\to\GL_{\overline W'}$ in \eqref{izd} (WLOG). This is a contradiction since 
the image of $\psi$ contains $\SL_2$ by the formal character of $\SO_5$.

For case $(4,1,0)$, the formal character and orthogonal duality of $\SO_5$
forces $\uG_\lambda\subset \GO_4^\circ\times\GL_1\subset\GL_{\overline W}\times\GL_1$ for $\ell\gg0$.
Since $\bar\rho_{\pi,\lambda}^{\ss}(\Gal_\Q)\subset\uG_\lambda$, we obtain a four-dimensional irreducible Galois representation
\begin{equation}\label{4d}
\bar\alpha_\lambda:\Gal_\Q\to \GO_4^\circ(\overline\F_\ell)\subset\GL_4(\overline\F_\ell).
\end{equation}
On the other hand,
one has an exact sequence 
\begin{equation}\label{ses}
1\to \GL_1\to \GL_2\times\GL_2\to \GO_4^\circ\to 1
\end{equation}
given by the tensor product of two two-dimensional representations. 
We can lift $\bar\alpha_\lambda$ to $\GL_2\times\GL_2$ by a theorem of Tate (see \cite[$\mathsection6.5$]{Se77b} or \cite{Co11})
to get $\bar\phi_\lambda:\Gal_\Q\to \GL_2(\overline\F_\ell)$ and 
$\bar\phi'_\lambda:\Gal_\Q\to \GL_2(\overline\F_\ell)$.
Since $n=5$ and $V_\lambda$ absolutely irreducible, 
the inequality 
$$|\mathrm{Tr}(c|_{V_\lambda})|\leq 1$$
holds for any complex multiplication $c\in\Gal_\Q$ by \cite[Proposition A]{Ta12},
which implies  that
the representations $\bar\phi_\lambda$ and $\bar\phi'_\lambda$ cannot both be even.
Hence, the irreducibility of $\bar\alpha_\lambda$ \eqref{4d} and $\bar\alpha_\lambda|_{[\Gal_\Q,\Gal_\Q]}$ (Theorem \ref{general}(ii)) imply 
that either $\bar\phi_\lambda$ or $\bar\phi'_\lambda$ is of Serre-type (i.e., odd and irreducible) 
and is also irreducible when restricted to $[\Gal_\Q,\Gal_\Q]$.

Suppose $\bar\phi_\lambda$ is of Serre-type for infinitely many $\lambda$ (WLOG)
with level $N_\lambda$ and weight $k_\lambda$ by Serre's recipe \cite[$\mathsection2$]{Da95}.
If we can show that $\bar\phi_\lambda$ for infinitely 
many $\lambda$ comes from a fixed cuspidal Hecke eigenform $f$, then 
we can reach a contradiction by the same argument in case (4B) 
by considering the SCS $\{\End(V_\lambda)=V_\lambda\otimes V_\lambda^\vee\}_\lambda$
with algebraic monodromy groups $\SO_5$ for almost all $\lambda$.
To do that, it suffices to find a lift $\bar\phi_\lambda$ for infinitely many $\lambda$ such that 
both the level $N_\lambda$ and the weight $k_\lambda$ are bounded 
by a constant independent of $\lambda$.\\

\noindent\ref{case5B}.1. \textit{(Level part)}. Let $S$ be the finite set of ramified places of the SCS $\{V_\lambda\}_\lambda$. 
 The Galois representation
\begin{equation}\label{proj}
\widetilde{\phi_\lambda}:\Gal_\Q\stackrel{\bar\phi_\lambda}{\rightarrow}\GL_2(\overline\F_\ell)\to\PGL_2(\overline\F_\ell)\subset\GL_3(\overline\F_\ell)
\end{equation}
is independent of the lift $\bar\phi_\lambda$ (which always exists by Tate) 
and is a subrepresentation of $\End(\overline V_\lambda^{\ss}\otimes\overline\F_\ell)$
such that 
\begin{enumerate}[(a)]
\item for $p$ outside $S\cup\{\ell\}$, it is unramified;
\item for $p\in S$ not equal to $\ell$,
the Artin conductor at $p$ is uniformly bounded independent of $\ell$.
\end{enumerate}
For uniformly boundedness of the levels $N_\lambda$, we would also like 
our lift $\bar\phi_\lambda$ to satisfy these conditions.

For a prime $p$, denote by $D_p$ and $I_p$ 
a decomposition subgroup and a inertia subgroup at $p$ of $\Gal_\Q$.
We state a theorem of Tate \cite[Theorem 5]{Se77b} which 
also holds for $\overline\F_\ell$ besides $\C$.

\begin{thm}(Tate) \label{Tatethm}
Let $\widetilde{\rho}:\Gal_\Q\to\PGL_2(\overline\F_\ell)$ be a continuous representation,
and for each prime $p$, let $\rho_p'$ be a lifting of $\widetilde{\rho}|_{D_p}$.
Suppose that $\rho_p'|_{I_p}$ is trivial for almost all $p$. Then there is 
a unique lifting $\rho:\Gal_\Q\to\GL_2(\overline\F_\ell)$ of $\widetilde{\rho}$ such that:
$$\rho|_{I_p}=\rho_p'|_{I_p}$$
for all $p$.
\end{thm}

By Theorem \ref{Tatethm}, it suffices to lift \eqref{proj} to $\GL_2$ locally at every prime $p$.
For $p$ outside $S\cup\{\ell\}$, an unramified local lift always exists.
For $p=\ell$, we just take any local lift, e.g., $\bar\phi_\lambda|_{D_p}$. 
For $p\in S$ not equal to $\ell$, the image $\bar\phi_\lambda(D_p)$ is solvable and
 contained in a Borel if the order $|\bar\phi_\lambda(D_p)|$ is divisible by $\ell$.
In this Borel case, let $\bar\theta_\lambda: D_p\to\overline\F_\ell^*$ be one of the two characters. 
By replacing $\bar\phi_\lambda|_{D_p}$ with  $\bar\phi_\lambda\otimes\bar\theta_\lambda^{-1}|_{D_p}$,
we get another local lift such that 
$$\ker(\bar\phi_\lambda\otimes\bar\theta_\lambda^{-1}|_{D_p})=\ker(\widetilde{\phi_\lambda}|_{D_p})$$
(for controlling the Artin conductor at $p$).

On the other hand, suppose $\ell$ does not divide $|\bar\phi_\lambda(D_p)|$.
Fix an isomorphism $\overline\Q_\ell\cong\C$.
We can lift the representation $\widetilde\phi_\lambda:D_p\to\PGL_2(\overline\F_\ell)$ to $\PGL_2(\overline\Q_\ell)\cong\PGL_2(\C)$ (such that $\widetilde\phi_\lambda(D_p)$ and the image in $\PGL_2(\overline\Q_\ell)$ are isomorphic),
then (lift) to $\GL_2(\C)\cong\GL_2(\overline\Q_\ell)$ \cite[$\mathsection6.5$]{Se77b},
and finally after mod $\ell$ reduction we get a representation 
$$\bar\rho_\lambda:D_p\to\GL_2(\overline\F_\ell)$$
which is also a lift of $\widetilde{\phi_\lambda}|_{D_p}$ up to conjugation. 
The solvable image $G$ of $D_p$ in $\PGL_2(\C)$ 
is (up to conjugation in $\PGL_2(\C)$) either (I): cyclic, (II): dihedral, or (III): tetrahedral or octahedral. We fix $G\subset \PGL_2(\C)$.
For case (I), we can lift $G$ to $\GL_2(\C)$ so that 
$$\ker\bar\rho_\lambda=\ker(\widetilde{\phi_\lambda}|_{D_p}).$$
For case (II), we can find a \emph{good lift}\footnote{We refer to \cite[$\mathsection1.2$]{FK04} for the description of good lifts.} 
of $\widetilde\phi_\lambda|_{D_p}$ to $\GL_2(\C)$ such that 
the Artin conductor is bounded in terms of  $|\widetilde\phi_\lambda(I_p^w)|$ (the order 
of the image of the wild inertia) by \cite[Theorem 1, Proposition 1]{FK04}.
Since \eqref{proj} is a subrepresentation of 
$\End(\overline V_\lambda^{\ss}\otimes\overline\F_\ell)$,
the order $|\widetilde\phi_\lambda(I_p^w)|$ is bounded by a constant independent of $\lambda$ by $\mathsection\ref{ic}$(ii).
For case (III), Krasner's lemma implies that there are finitely many tetrahedral or octahedral
extension $F$ of $\Q_p$ and for each such $F$ there are finitely many isomorphisms
$$i:\Gal(F/\Q_p)\stackrel{\cong}{\rightarrow} G\subset\PGL_2(\C)$$
where $G$ is our fixed tetrahedral or octahedral image.
Now for a pair $(F,i)$, pick a lift of
$$D_p\twoheadrightarrow \Gal(F/\Q_p)\stackrel{i}{\rightarrow} G\subset\PGL_2(\C)$$
to $\GL_2(\C)$. After the identification $\GL_2(\C)\cong\GL_2(\overline\Q_\ell)$
and mod $\ell$ reduction, we find finitely many maps $\bar\rho_\lambda: D_p\to \GL_2(\overline\F_\ell)$
that lift all possible $\widetilde{\phi_\lambda}|_{D_p}$ in case (III).

It follows from the analysis on the four cases (Borel, (I), (II), (III)) 
that $\widetilde\phi_\lambda|_{D_p}$ can be lifted to $\GL_2(\overline\F_\ell)$
such that the Artin conductor is bounded independent of $\lambda$.
Therefore, we can find lifts $\bar\phi_\lambda$ to achieve the conditions (a) and (b) (after \eqref{proj})  by Theorem \ref{Tatethm}.\\

\noindent\ref{case5B}.2. \textit{(Weight part)}.
Let $\bar\phi_\lambda$ be the lifts of \eqref{proj} for infinitely many $\lambda$ satisfying 
the conditions (a) and (b) in the level part.
By twisting $\bar\phi_\lambda$ with suitable powers of the cyclotomic character 
$\bar\epsilon_\ell$, 
we may assume the tame inertia weights of $\bar\phi_\lambda|_{\Gal_{\Q_\ell}}$ 
is $\{a_\lambda,0\}$ with $a_\lambda\geq0$.
We would like to show that $a_\lambda>0$ and is uniformly bounded for $\ell\gg0$.
Since $\End(\overline V_\lambda^{\ss}\otimes\overline\F_\ell)|_{\Gal_{\Q_\ell}}$ 
is Fontaine-Laffaille for $\ell\gg0$, 
its tame inertia weights and the Hodge-Tate weights (independent of $\lambda$)
coincide for $\ell\gg0$ by \cite{FL82}.
Since $\widetilde\phi_\lambda$ \eqref{proj} is a subrepresentation of $\End(\overline V_\lambda^{\ss}\otimes\overline\F_\ell)$
and $a_\lambda$ is a tame inertia weight of $\widetilde\phi_\lambda|_{\Gal_{\Q_\ell}}$,
the integer $a_\lambda$ is uniformly bounded.
From \eqref{ses}, we find two lifts $\bar\phi_\lambda$ and $\bar\phi_\lambda'$ 
such that $\bar\phi_\lambda\otimes\bar\phi_\lambda'=\bar\alpha_\lambda$ in \eqref{4d}
with four distinct tame inertia weights for $\ell\gg0$ (by regularity and Fontaine-Laffaille theory \cite{FL82}).
It follows that $a_\lambda>0$ for $\ell\gg0$.
If $\bar\phi_\lambda|_{I_\ell}$ is not semisimple, we may assume that 
it is an extension of the trivial character by $\bar\epsilon_\ell^{a_\lambda}$ 
(\eqref{inertiarp} with $b_\lambda=0$)\footnote{Otherwise, some twist of $\bar\phi_\lambda'|_{I_\ell}$ 
is either in this form or semisimple. This is seen by 
applying \cite[Remark 8.3.7, Exercise 8.4.3]{BrC09} to a two dimensional subrepresentation
of some Fontaine-Laffaille lift of $\bar\phi_\lambda\otimes\bar\phi_\lambda'|_{\Gal_{\Q_\ell}}$ 
\cite[Proposition 2.3.1]{GHLS17}. Therefore, we can use $\bar\phi_\lambda'$ instead of $\bar\phi_\lambda$.}.
According to Serre's recipe, 
the weight $k_\lambda$ of the eigenform $f_\lambda$ attached to $\bar\phi_\lambda$ (for $a_\lambda>0$) is given by
$1+a_\lambda+(\ell-1)\delta$, where $\delta=0$ or $1$. 
And $\delta=1$ arises only when 
$\bar\phi_{\lambda}|_{\Gal_{\Q_\ell}}$ is not peu ramifi\'ee.
If this is the case, then $\bar\phi_{\lambda}|_{\Gal_{\Q_\ell}}$ is reducible and 
we obtain $\Gal_{\Q_\ell}$-subrepresentations
\begin{equation}\label{subrepns}
\bar\phi_{\lambda}\otimes \bar\chi_\lambda\leq \End(\bar\phi_{\lambda})\leq \End(\overline V_\lambda^{\ss}\otimes\overline\F_\ell),
\end{equation}
where $\bar\chi_\lambda$ is some character of $\Gal_{\Q_\ell}$.
Since $\End(\overline V_\lambda\otimes\overline\F_\ell)|_{\Gal_{\Q_\ell}}$
admits a Fontaine-Laffaille lift for $\ell\gg0$, it is peu ramifi\'ee 
with respect to every saturated filtration \cite[Proposition 2.3.1]{GHLS17}.
Hence, $\bar\phi_{\lambda}\otimes \bar\chi_\lambda|_{\Gal_{\Q_\ell}}$ and $\bar\phi_{\lambda}|_{\Gal_{\Q_\ell}}$ are peu ramifi\'ee by \eqref{subrepns} and \cite[Remark 2.1.6]{GHLS17}
which is a contradiction. We conclude that the weight $k_\lambda=1+a_\lambda$ is uniformly bounded.\qed

\section*{Acknowledgments}
The author is indebted to Michael Larsen for 
pointing out an error during the preparation of this work, suggesting to use the structure theory of 
finite groups in $\PGL_2(\C)$ in $\mathsection$\ref{case5B}.1, and giving many helpful comments.
He would like to thank Frank Calegari for answering some questions on automorphy theorems and \cite{CG13}, Luis Dieulefait for answering a question on weight of modular form and referring to \cite{GHLS17}, 
Hui Gao for answering some questions on \cite{GHLS17} and $p$-adic Hodge theory, and Gebhard B\"ockle for his interests in this work.
The author would like to thank the referee for a very careful reading and a lot of helpful comments and suggestions. 
The work described in this paper was partially supported
by Hong Kong RGC (no. 17302321) and NSFC (no. 12222120).

\vspace{.1in}

\begin{thebibliography}{BLGGT14}
\bibitem[AC89]{AC89}
Arthur, James;  Clozel, Laurent:
Simple algebras, base change, and the advanced theory of the trace
formula, Annals of Mathematics Studies, vol. 120, Princeton University Press, Princeton, NJ, 1989.

\bibitem[Ba20]{Ba20}
Bartlett, Robin: 
Inertial and Hodge-Tate weights of crystalline representations,  \textit{Math. Annalen} \textbf{376} (2020), 645--681.



\bibitem[Bo91]{Bo91}
Borel, A.:
\textit{Linear Algebraic Groups}, Graduate Texts in Mathematics, 126 (2nd ed.), Springer-Verlag 1991.



	
\bibitem[BC11]{BC11}
Bella\"iche, Jo\"el; Chenevier, Ga\"etan: 
The sign of Galois representations attached to automorphic
forms for unitary groups, \textit{Compos. Math.} \textbf{147} (2011), no. 5, 1337--1352.

\bibitem[BrC09]{BrC09}
Brinon, Olivier; Conrad, Brian:
CMI Summer School notes on $p$-adic Hodge
theory (preliminary version), available at \url{https://math.stanford.edu/~conrad/papers/notes.pdf}.


\bibitem[BLGGT14]{BLGGT14}
Barnet-Lamb, T.; Gee, T.; Geraghty, D.; Taylor, R.:
Potential automorphy and change of weight,
 \textit{Annals of Math.} \textbf{179}, p. 501--609, 2014.

\bibitem[BLGHT11]{BLGHT11}
Barnet-Lamb, T.; Geraghty, D.; Harris, M.; Taylor, R.:
      A family of Calabi-Yau varieties and potential automorphy II, 
			\textit{Publ. Res. Inst. Math. Sci.} \textbf{47} (2011), no. 1, 29--98.

\bibitem[BR92]{BR92}
Blasius, Don;  Rogawski, Jonathan D.:
   Tate classes and arithmetic quotients of the two-ball, The zeta
   functions of Picard modular surfaces, Univ. Montr\'eal, Montreal, QC, 1992, pp. 421--444.

\bibitem[BT72]{BT72}
	Bruhat, F.; Tits, J.:
	Groupes r\'eductifs sur un corps local: I. Donn\'ees radicielles valu\'ees, 
	\textit{Publ. Math. IHES} \textbf{41} (1972), p. 5--251.
	
\bibitem[BT84]{BT84}
	Bruhat, F.; Tits, J.:
	Groupes r\'eductifs sur un corps local: II. Sch\'emas en groupes. Existence d'une donn\'ee radicielle valu\'ee, 
	\textit{Publ. Math. IHES} \textbf{60} (1984), p. 5--184.
	
\bibitem[BDJ10]{BDJ10}
				Buzzard, Kevin; Diamond, Fred; Jarvis, Frazer:
				On Serre's conjecture for mod l Galois representations over totally real fields,
				\textit{Duke Math. J.} \textbf{55}, p. 105--161, 2010.

\bibitem[Ca15]{Ca15}
  Cadoret, Anna: 
	An open adelic image theorem for abelian schemes,
\textit{I.M.R.N.} vol. \textbf{2015}, p. 10208--10242, 2015.


\bibitem[Ca19]{Ca19}
  Cadoret, Anna:
  An open adelic image theorem for motivic representations over function fields.
\textit{Math. Res. Lett.} \textbf{26}, p. 1--8, 2019.

\bibitem[CHT17]{CHT17}
    Cadoret, Anna; Hui, Chun Yin; Tamagawa, Akio:
    Geometric monodromy -- semisimplicity and maximality,
    \textit{Annals of Mathematics} \textbf{186} (2017), Issue 1, p. 205--236.

\bibitem[CM20]{CM20}
   Cadoret, Anna; Ben, Moonen:
		Integral and adelic aspects of the Mumford-Tate conjecture,
\textit{Journal of the Institute of Mathematics of Jussieu}, Volume 19 , Issue 3 , May 2020 , pp. 869--890.

\bibitem[CG13]{CG13}
Calegari, Frank;  Gee, Toby: 
Irreducibility of automorphic Galois representations of $\GL(n)$, $n$ at most $5$, 
\textit{Annales de l'Institut Fourier} \textbf{63}(5), p. 1881--1912, 2013.

\bibitem[Cl37]{Cl37}
Clifford, Alfred:
 Representations induced in an invariant subgroup, \textit{Ann. of Math.} \textbf{38} (1937), 533--550.

\bibitem[Cl90]{Cl90}
	Clozel, Laurent: 
	Motifs et formes automorphes: applications du principe de fonctorialit\'e.  Automorphic forms, Shimura varieties, and L-functions, Vol. I (Ann Arbor, MI, 1988), 77--159, Perspect. Math., 10, Academic 
	Press, Boston, MA, 1990.
	
\bibitem[Cl16]{Cl16}
	Clozel, Laurent: 
	Motives and automorphic representations. (English, French summary) Asian-French Summer School on Algebraic Geometry and Number Theory. Vol. III, 29-60, Panor. Synth\'eses, 49, Soc. Math. France, Paris, 2016.


\bibitem[Co11]{Co11}
Conrad, Brian: 
Lifting global representations with local properties, preprint available at \url{http://math.stanford.edu/~conrad/papers/locchar.pdf} (2011).

\bibitem[Co14]{Co14}
Conrad, Brian:
\textit{Reductive group schemes}, online notes, 2014.

\bibitem[CR88]{CR88} 
  Curtis, Charles W.; Reiner, Irving: 
  Representation theory of finite groups and associative algebras, reprint of the 1962 original. 
  Wiley Classics Library. A Wiley-Interscience Publication. \textit{John Wiley $\&$ Sons, Inc., New York}, 1988.

\bibitem[Da95]{Da95}
  Darmon, Henri:	
	Serre's conjectures, Seminar on Fermat's Last Theorem: 1993-1994, the Fields
Institute for Research in the Mathematical Sciences, Toronto, Ontario,
Canada, volume \textbf{17} of CMS conference proceedings, pages 135--153.
American Mathematical Society, 1995.

\bibitem[De74]{De74} 
  Deligne, Pierre: 
  La conjecture de Weil I, 
  \textit{Publ. Math. I.H.E.S.}, \textbf{43} (1974), 273-307.

\bibitem[DV08]{DV08}
	Dieulefait, Luis; Vila, N.:
	Geometric families of 4-dimensional Galois representations with generically large images, 
	\textit{Math. Z.} \textbf{259} (2008), 897--893.

\bibitem[DV11]{DV11}
	Dieulefait, Luis; Vila, N.:
	On the Classification of Geometric Families of Four-Dimensional Galois Representations, 
	\textit{Mathematical Research Letters} \textbf{18} (2011), no. 4, p. 805--814.
	
\bibitem[DZ20]{DZ20}
   Dieulefait, L. and Zenteno, A.: On the images of the Galois representations attached to generic
automorphic representations of $\GSp(4)$, \textit{Ann. Sc. Norm. Super. Pisa Cl. Sci.} (5) \textbf{20} (2020), no. 2, 635--655.

\bibitem[Di05]{Di05}
Dimitrov, Mladen:
 Galois representations modulo $p$ and cohomology of Hilbert modular varieties, 
\textit{Annales scientifiques de l'\'Ecole Normale Sup\'erieure} \textbf{38}, Issue 4 (2005), 505--551. 

\bibitem[Ed92]{Ed92}
Edixhoven, Bas: 
 The weight in Serre's conjectures on modular forms,
\textit{Invent. Math.} \textbf{109} (1992), 563--594.

\bibitem[Fi99]{Fi99}
Figueiredo, L.M.: 
 Serre's conjecture for imaginary quadratic fields, \textit{Compos. Math.} \textbf{118 }(1999) 103--122.

\bibitem[FK04]{FK04}
	Frederiksen, P.; Kiming, I.: 
	Galois representations of dihedral type over ${\mathbb Q}_p$, \textit{Acta Arith.} \textbf{111} (2004), 43--59.
	
\bibitem[FL82]{FL82}
Fontaine, J.M.; Laffaille, G.:
 Construction de repr\'esentations $p$-adiques,
\textit{Annales scientifiques de l'\'Ecole Normale Sup\'erieure}, Serie 4, Volume \textbf{15} (1982) no. 4, p. 547--608.

\bibitem[GHLS17]{GHLS17}
  Gee, Toby; Herzig, Florian; Liu, Tong; Savitt, David:
  Potentially crystalline lifts of certain prescribed types,
\textit{Documenta Mathematica} \textbf{22} (2017), 397--422.

\bibitem[Gi19]{Gi19}
Gille, Philippe:
\textit{Introduction to reductive group schemes over rings}, notes, \url{http://math.univ-lyon1.fr/homes-www/gille/prenotes/reductive.pdf}.

\bibitem[HLTT16]{HLTT16}
Harris, Michael; Lan, Kai-Wen; Taylor, Richard; Thorne, Jack:
On the rigid cohomology of certain Shimura varieties, \textit{Res. Math. Sci.} \textbf{3} (2016), article no. 37, 308 pp.

\bibitem[HT01]{HT01}
Harris, Michael;  Taylor, Richard: 
The Geometry and Cohomology of Some
Simple Shimura Varieties, with an appendix by Vladimir G. Berkovich,
Ann. of Math. Stud. no. 151, Princeton Univ. Press, Princeton, NJ, 2001.

\bibitem[Hu13]{Hu13}
	Hui, Chun Yin:
	Monodromy of Galois representations and equal-rank subalgebra equivalence,
	\textit{Math. Res. Lett.} \textbf{20} (2013), no. 4, 705--728.

\bibitem[Hu15]{Hu15}
	Hui, Chun Yin:
	$\ell$-independence for compatible systems of (mod $\ell$) Galois representations,
	\textit{Compositio Mathematica} \textbf{151} (2015), 1215--1241. 
	
\bibitem[Hu18]{Hu18}
  Hui, Chun Yin:
	On the rationality of certain type A Galois representations,
	\textit{Transactions of the American Mathematical Society}, Volume \textbf{370}, Number 9, September 2018, p. 6771--6794.
	
\bibitem[Hu20]{Hu20}
  Hui, Chun Yin:
	The abelian part of a compatible system and $\ell$-independence of the Tate conjecture, 
	\textit{Manuscripta Mathematica}, Vol. 161, Issue 1-2, Jan 2020, pp 223--246.

\bibitem[Hu22]{Hu22}
  Hui, Chun Yin:
	On the rationality of algebraic monodromy groups of compatible systems, accepted in \textit{Journal of European Mathematical Society}.
	
\bibitem[HL13]{HL13} 
  Hui, Chun Yin; Larsen, Michael: 
Adelic openness without the Mumford-Tate conjecture, preprint, arXiv:1312.3812.

\bibitem[HL16]{HL16} 
  Hui, Chun Yin; Larsen, Michael: 
  Type A images of Galois representations and maximality,
  \textit{Mathematische Zeitschrift} \textbf{284} (2016), no.3--4, 989--1003.
	
\bibitem[HL20]{HL20} 
  Hui, Chun Yin; Larsen, Michael: 
  Maximality of Galois actions for abelian and hyper-K\"ahler varieties,
  \textit{Duke Math. J.} \textbf{169} (2020),  no. $6$, 1163--1207.
	
\bibitem[JS81]{JS81}
  Jacquet, H.;  Shalika, J. A.:
	On Euler products and the classification of automorphic forms. II. \textit{Amer. J. Math.}, \textbf{103}(4): 777--815, 1981.

\bibitem[Ja97]{Ja97}
	Jantzen, Jens Carsten:
	Low-dimensional representations of reductive groups are semisimple. 
	Algebraic groups and Lie groups, 255--266, Austral. Math. Soc. Lect. Ser., 9, 
	Cambridge Univ. Press, Cambridge, 1997.
		
\bibitem[KW09a]{KW09a}
  Khare, Chandrashekhar; Wintenberger, Jean-Pierre:
	Serre's modularity conjecture (I),
	\textit{Inventiones mathematicae}, \textbf{178}(3):485--504, 2009.
		
\bibitem[KW09b]{KW09b}
  Khare, Chandrashekhar; Wintenberger, Jean-Pierre:
	Serre's modularity conjecture (II),
	\textit{Inventiones mathematicae}, \textbf{178}(3):505--586, 2009.
		
\bibitem[Lan00]{Lan00}
  Landvogt, Erasmus :
		Some functorial properties of the Bruhat-Tits building,
		\textit{J. reine angew. Math.} \textbf{518} (2000), 213--241.

\bibitem[La95]{La95} 
	Larsen, Michael:
	Maximality of Galois actions for compatible systems. 
	\textit{Duke Math.\ J.} 
	\textbf{80} (1995), no.\ 3, 601--630. 
	
\bibitem[LP92]{LP92}
	Larsen, M.; Pink, R.:
	On $\ell$-independence of algebraic monodromy groups in compatible systems of representations. 
	\textit{Invent.\ Math.} \textbf{107} (1992), no.\ 3, 603--636.

\bibitem[LP95]{LP95} 
  Larsen, Michael; Pink, Richard: 
 Abelian varieties, $l$-adic representations, and $l$-independence,
\textit{Math. Ann.} \textbf{302} (1995), no. 3, 561--579. 

\bibitem[LP97]{LP97} 
	Larsen, Michael; Pink, Richard: 
 	A connectedness criterion for $\ell$-adic Galois representations,
	\textit{Israel J. Math.} \textbf{97} (1997), 1--10.
	
\bibitem[No87]{No87}
	 Nori, Madhav V.:
	 On subgroups of $\GL_n(\F_p)$. 
	 \textit{Invent.\ Math.}
	 \textbf{88} (1987), no.\ 2, 257--275.
	
	
	
\bibitem[PSW18]{PSW18}
  Patrikis, Stefan; Snowden, Andrew; Wiles,	Andrew:
	Residual Irreducibility of Compatible Systems, 
	\textit{International Mathematics Research Notices}, 
	Vol. 2018, Issue 2, Jan 2018, pp. 571--587.
	
\bibitem[PT15]{PT15}
  Patrikis, Stefan; Taylor, Richard:
	Automorphy and irreducibility of some $l$-adic representations,
	\textit{Compos. Math.} \textbf{151} (2015), 207--229.
	

\bibitem[Ra08]{Ra08}
  Ramakrishnan, Dinakar:
	Irreducibility and cuspidality. Representation theory and automorphic forms, 1--27, 
	Progr. Math., 255, Birkh$\mathrm{\ddot{a}}$user Boston, Boston, MA, (2008).
	
\bibitem[Ra13]{Ra13}
  Ramakrishnan, Dinakar:
	Decomposition and parity of Galois representations attached to $\GL(4)$, Automorphic
representations and L-functions, 2013, pp. 427--454.

\bibitem[Ri77]{Ri77}
  Ribet, Kenneth A.:
	Galois representations attached to eigenforms with Nebentypus, Modular functions
of one variable, V (Proc. Second Internat. Conf., Univ. Bonn, Bonn, 1976), Springer, Berlin, 1977,
pp. 17--51. Lecture Notes in Math., Vol. \textbf{601}.

\bibitem[Se72]{Se72}
	Serre, Jean-Pierre:
	Propri\'et\'es galoisiennes des points d'ordre fini des courbes elliptiques. 
	\textit{Invent.\ Math.}
	\textbf{15} (1972), no.\ 4, 259--331.
	
\bibitem[Se77a]{Se77a}
	Serre, Jean-Pierre:
	Linear representations of finite groups, Springer-Verlag 1977.
	
\bibitem[Se77b]{Se77b}
	Serre, Jean-Pierre:
	Modular forms of weight one and Galois representations. In Algebraic number fields:
L-functions and Galois properties (Proc. Sympos., Univ. Durham, Durham, 1975), pages 193--268, 1977.

\bibitem[Se79]{Se79}
	Serre, Jean-Pierre:
	Local fields.  
	Graduate Texts in Mathematics, 67. Springer-Verlag, New York-Berlin, 1979.
	
\bibitem[Se81]{Se81}
	Serre, Jean-Pierre:
	Lettre \`a Ken Ribet du 1/1/1981 et du 29/1/1981 (Oeuvres IV, no.\ 133).

\bibitem[Se86]{Se86} 
  Serre, Jean-Pierre:
  Lettre \'a Marie-France Vign\'eras du 10/2/1986 (Oeuvres IV, no.\ 137).

\bibitem[Se87]{Se87} 
  Serre, Jean-Pierre:
  Sur les repr\'esentations modulaires de degr\'e 2 de $\Gal(\overline \Q/\Q)$, 
	\textit{Duke Math. J.} Vol. \textbf{54}, no. 1, 179--230 (1987).

\bibitem[Se94]{Se94}
	Serre, Jean-Pierre:
	\textit{Propri\'et\'es conjecturales des groupes de Galois motiviques et des repr\'esentations $l$-adiques},
	Motives (Seattle, WA, 1991), 377--400, Proc. Sympos. Pure Math., 55, Part 1, Amer. Math. Soc., Providence, RI, 1994. 
	
\bibitem[Se98]{Se98} 
	Serre, Jean-Pierre:
	Abelian $l$-adic representation and elliptic curves, Research Notes in Mathematics Vol. 7 (2nd ed.), \textit{A K Peters} (1998).

\bibitem[St68]{St68}
	Steinberg, Robert:
	Endomorphisms of linear algebraic groups. 
	Memoirs of the American Mathematical Society, No. 80 American Mathematical Society, Providence, R.I., 1968.

\bibitem[Ta95]{Ta95}
   Taylor, Richard:
	 On Galois representations associated to Hilbert modular forms. II, Elliptic curves,
modular forms, \& Fermat's last theorem (Hong Kong, 1993), Ser. Number Theory, I, Int. Press,
Cambridge, MA, 1995, pp. 185--191.

\bibitem[Ta04]{Ta04}
Taylor, Richard:
Galois representations, \textit{Ann. Fac. Sci. Toulouse} \textbf{13} (2004), 73--119.

\bibitem[Ta12]{Ta12}
   Taylor, Richard:
   The image of complex conjugation in $l$-adic representations associated to automorphic forms,
	\textit{Algebra Number Theory} \textbf{6}(3): 405--435 (2012).

	
\bibitem[Ti79]{Ti79}
Tits, Jacques:
	Reductive groups over local fields. Automorphic forms, representations and L-functions, Part 1, 	
	pp.\ 29--69, Proc. Sympos. Pure Math., XXXIII, Amer. Math. Soc., Providence, R.I., 1979.	

\bibitem[We22]{We22}
 Weiss, Ariel:
   On the images of Galois representations attached to low weight Siegel modular forms, 
	\textit{J. London Math. Soc.}(2), Volume \textbf{106} Issue 1, July 2022,
Pages 358--387.


\bibitem[Xi19]{Xi19}
Xia, Yuhou:
	Irreducibility of automorphic Galois representations of low dimensions,
	\textit{Mathematische Annalen} \textbf{374}, pages 1953--1986 (2019).

\bibitem[Yu09]{Yu09}
Yu, Jiu-Kang:
   Bruhat-Tits theory and buildings, Ottawa lectures on admissible representations of reductive $p$-adic groups, 
	Fields Institute Monographs, vol. \textbf{26}, AMS., Providence, RI, 2009, pp. 53--77.
	
	\end{thebibliography}
\end{document}